\documentclass[11pt]{amsart}
\usepackage{amssymb, latexsym, mathrsfs, arydshln,  color,xcolor, amsmath }
\usepackage[all]{xy}
\usepackage[colorlinks=true, pdfstartview=FitV,linkcolor=blue,citecolor=blue,urlcolor=blue]{hyperref}
\usepackage [latin1]{inputenc}
\usepackage{chngcntr}
\counterwithin{figure}{section}

    \advance \topmargin by -\headheight
    \advance \topmargin by -\headsep

    \evensidemargin \oddsidemargin
\newtheorem {theorem}    {Theorem}[section]

\newtheorem {lemma}      [theorem]    {Lemma}
\newtheorem {corollary}  [theorem]    {Corollary}
\newtheorem {proposition}[theorem]    {Proposition}

\newtheorem {example}[theorem]    {Example}

\newcommand{\bb}{\mathbb}
\renewcommand{\rm}{\mathrm}

\newcommand{\cal}{\mathcal}

\newcommand{\GGL}{\mathrm{GL}}
\newcommand{\UU}{\mathrm{U}}
\renewcommand{\sp}{\mathrm{Sp}}
\renewcommand{\o}{\mathrm{O}}
\newcommand{\so}{\mathrm{SO}}

\newcommand{\ep}{{{\color{pink}{\epsilon_\alpha}}}}
\newcommand{\ev}{{{\color{violet}{\epsilon_\beta}}}}

\newcommand{\Fq}{\bb{F}_q}
\newcommand{\fq}{(\bb{F}_q)}

\theoremstyle{definition}
\newtheorem{definition}[theorem]{Definition}
\newtheorem{remark}[theorem]{Remark}

\newcommand{\E}{\mathcal{E}}
\newcommand{\pl}{\pi_{\Lambda}}
\newcommand{\ps}{\pi^{\star}}

\newcommand{\pll}{\pi_{\Lambda'}}
\newcommand{\prll}{\pi_{\rho,\Lambda_1,\Lambda_{-1}}}
\newcommand{\prllp}{\pi_{\rho',\Lambda_1',\Lambda'_{-1}}}

\newcommand{\prlls}{\pi_{\rho^\star,\Lambda^\star_1,\Lambda^\star_{-1}}}
\newcommand{\prllsp}{\pi_{\rho^{\star\prime},\Lambda^{\star\prime}_1,\Lambda^{\star\prime}_{-1}}}

\newcommand{\pkh}{\pi_{\rho,k,h}}
\newcommand{\pkhp}{\pi_{\rho',k',h'}}
\newcommand{\pkhs}{\pi_{\rho^\star,k^\star,h^\star}}
\newcommand{\pkhsgn}{\pi_{\rho,k,h,\iota}}
\newcommand{\pkhsp}{\pi_{\rho',k',h',\iota'}}
\newcommand{\pkhss}{\pi_{\rho^\star,k^\star,h^\star,\iota^\star}}

\newcommand{\prllps}{\pi_{\rho',\Lambda'_1,\Lambda'_{-1},\iota'}}

\newcommand{\prllsgn}{\pi_{\rho,\Lambda_1,\Lambda_{-1},\iota}}

\newcommand{\prllss}{\pi_{\rho^{\star},\Lambda^{\star}_1,\Lambda^{\star}_{-1},\iota^{\star}}}
\newcommand{\prllspp}{\pi_{\rho^{\star\prime},\Lambda^{\star\prime}_1,\Lambda^{\star\prime}_{-1},\iota^{\star\prime}}}

\newcommand{\clc}{{\cal{L}^{\rm{can}}}}

\newcommand{\dg}{G^{*}_{[\ne \pm 1]}(s)}
\newcommand{\ddg}{G^{*}_{[1]}(s)}

\newcommand{\dddg}{G^{*}_{[- 1]}(s)}

\newcommand{\sgn}{\rm{sgn}\cdot}

\newcommand{\CD}{{\mathcal{D}}}

\newcommand{\e}{{\epsilon'}}
\newcommand{\ee}{\varepsilon(-1)}

\numberwithin{equation}{section}

\newcommand{\cl}{{\cal{L}}}

\newcommand{\Irr}{{\rm{Irr}}}

\newcommand{\id}{\rm{id}}

\newcommand{\sd}{\rm{sd}}
\newcommand{\pd}{\rm{pd}}

\newcommand{\disc}{{\mathrm{disc}}}
\newcommand{\Fn}{{\mathfrak {n}}}
\newcommand{\Fm}{{\mathfrak {m}}}
\newcommand{\rank}{{\mathrm{rank}}}
\newcommand{\bpair}[1]{\left[{#1}\right]}
\newcommand{\bd}{{\bf \mathrm{d}}}
\newcommand{\Ad}{{\mathrm{Ad}}}
\newcommand{\Ind}{{\mathrm{Ind}}}
\newcommand{\CE}{{\mathcal {E}}}

\newcommand{\CL}{{\mathcal {L}}}
\newcommand{\CV}{{\mathcal {V}}}
\newcommand{\Res}{{\mathrm{Res}}}

\newcommand{\nep}{{\bf n}^\epsilon}
\newcommand{\nepo}{{\bf n}^\epsilon_{\rm{odd}}}
\newcommand{\nepe}{{\bf n}^\epsilon_{\rm{even}}}
\newcommand{\nodd}{{\bf n}_{\rm{odd}}}
\newcommand{\neven}{{\bf n}_{\rm{even}}}

\newcommand{\wi}{{\widetilde{\iota}}}
\begin{document}

\title{Lusztig's Jordan decomposition and a finite field instance of relative Langlands duality}

\date{\today}

\author[Zhicheng Wang]{Zhicheng Wang}

\address{School of Mathematical Science, Zhejiang University, Hangzhou 310027, Zhejiang, P.R. China}

\email{11735009@zju.edu.cn}
\subjclass[2010]{Primary 20C33; Secondary 22E50}

\begin{abstract}
Lusztig \cite{L5,L6} gave a parametrization for $\rm{Irr}(G^F)$, where $G$ is a reductive algebraic group defined over $\mathbb{F}_q$, with Frobenius map $F$. This parametrization is known as Lusztig's Jordan decomposition or Lusztig correspondence. However, there is not a canonical choice of Lusztig correspondence. In this paper, we consider classical groups. We pick a canonical choice of Lusztig correspondence which is compatible with parabolic induction and is compatible with theta correspondence. This result extends Pan's result in \cite{P3}. As an application, we give a refinement of the results of the finite Gan-Gross-Prasad problem in \cite{Wang1} and prove a duality between Theta correspondence and finite Gan-Gross-Prasad problem, which can be regarded as a finite field instance of relative Langlands duality of Ben-Zvi-Sakellaridis-Venkatesh \cite{BZSV}.
\end{abstract}

\maketitle

\section{Introduction}\label{sec1}

\subsection{Lusztig's Jordan decomposition}
Let $\overline{\mathbb{F}}_q$ be an algebraic closure of a finite field $\mathbb{F}_q$, which is of characteristic $p>2$.  Suppose that the cardinality of $\Fq$ is large
enough so that the main result in \cite{S} holds. Consider a reductive algebraic group $G$ defined over $\Fq$, with Frobenius map $F$. Let $\E(G)=\rm{Irr}(G^F)$ be the set of irreducible representations of $G^F$. We write
\[
\langle \sigma,\pi \rangle_{H^F} = \dim \mathrm{Hom}_{H^F}(\sigma,\pi ),
\]
where $H$ is a subgroup of $G$, $\pi\in\E(G)$, and $\sigma\in\E(H)$.
We say $\sigma$ occurs in $\pi$ or $\sigma\in \pi$ if $\langle \sigma,\pi \rangle_{H^F}\ne 0$.

In \cite{L5}, Lusztig provided a way to parametrize $\E(G)$ under the assumption that $Z$ is a connected component. This parametrization is known as Lusztig's Jordan decomposition. Later, in \cite{L6}, this parametrization was extended to cover the case where $Z$ is disconnected. With Lusztig's Jordan decomposition, we obtain a partition of $\E(G)$ based on the semisimple conjugacy classes in the dual group $G^{*F}$ of $G^F$, that is,
\begin{equation}\label{eq:irr-decomp-1}
\E(G)=\coprod_{(s)}\mathcal{E}(G,s)
\end{equation}
where $(s)$ runs over $G^{*F}$-conjugacy classes of semisimple elements in  the identity component of $G^{*F}$, and
 the Lusztig series $\mathcal{E}(G,s)$ is defined by
\[
\mathcal{E}(G,s) := \{ \pi \in \E(G)  | \langle \pi,  R_{T^*,s}^G\rangle \ne 0\textrm{ for some }T^*\textrm{ containing }s \}.
\]
Here if $G$ is a connected reductive group, one has a natural bijection between the set of $G^F$-conjugacy classes of the pairs $(T, \theta)$ where $T$ is an $F$-stable maximal torus of $G$ and $\theta$ is a character of $T^F$, and the set of $G^{*F}$-conjugacy classes of the pairs $(T^*, s)$ where $T^*$ is an $F$-stable maximal torus of $G^*$ and contains $s$.
If $(T, \theta)$ corresponds to $(T^*, s)$, we denote  by $R_{T^*,s}^G$  the Deligne-Lusztig character $R^G_T(\theta)$ of $G^F$ in \cite{DL}.
And if $G$ is non-connected, we define  $R^G_{T^*,s}:=\rm{Ind}^{G^F}_{(G^0)^F}R^{G^0}_{T^*,s}$.
 Lusztig shows that there exists a bijection between the Lusztig series $\E(G,s)$ and the unipotent representations $\E(C_{G^*}(s),1)$ of $C_{G^{*}}(s)^F$ (see \cite[Chap 9]{L1}),
\[
\mathcal{L}_s:\mathcal{E}(G,s)\to \mathcal{E}(C_{G^{*}}(s), 1),
\]
extended by linearity to a map between virtual characters satisfying that
\begin{equation}\label{eq1}
\mathcal{L}_s(\varepsilon_G R^G_{T^*,s})=\varepsilon_{C_{G^{*}}(s)} R^{C_{G^{*}}(s)}_{T^*,1},
\end{equation}
where $\varepsilon_G:= (-1)^{\rm{rk}\,G}$ and $\rm{rk}\,G$ is the $\Fq$-rank of $G$, and $C_{G^*}(s)$ is the centralizer  of $s$ in $G^*$.
Moreover, $\CL_s$ sends cuspidal representations to cuspidal representations. Such a bijection $\CL_s$ will be called a Lusztig correspondence
or Lusztig's Jordan decomposition.

It should be noted that in general, there is no unique choice of $\CL_s$ (see \cite[appendix A.5]{GM2}). A natural question arises as to whether there is a canonical choice of Lusztig correspondence. In their work \cite{DM}, Digne and Michel demonstrate that if $G$ is a connected group with a connected center, then $\CL_s$ can be uniquely determined by \eqref{eq1} and some additional conditions.

Let $\CV(G)^\#$ be the subspace of the space of (complex-valued) class functions on $G$ spanned by Deligne-Lusztig virtual characters of $G$. Let $\pi^\#$ be the orthogonal projection of $\pi$ onto $\CV(G)^\#$. We say that $\pi$ is uniform if $\pi=\pi^\#$. Clearly, the image of Lusztig correspondence for uniform representations is uniquely determined by \eqref{eq1}.
If $G$ is a general linear group $\GGL_n$ or a unitary group $\UU_n$, every irreducible representation is uniform. However, if $G$ is a symplectic group or an orthogonal group, there exists some ambiguity in Lusztig correspondence. The reason is that $G$ or $C_{G^*}(s)$ might be disconnected.

Let $[\pi]:=\{\pi'|(\pi')^\#=\pi^\#\}$ be the set of irreducible representations that share the same uniform projection as $\pi$. To differentiate between representations in $[\pi]$, Weil representation and theta lifting are necessary. These have been studied in \cite{AMR, P1, P2, MQZ}. Based on their results, we know that the images of representations in $[\pi]$ under theta correspondence are distinct. In \cite{P3}, Pan provides an explicit description of $[\pi]$ and proves the existence of a unique $\CL_s$ that satisfies specific compatibility conditions regarding theta correspondence. However, Pan's choice of $\CL_s$ depends on the selection of a dual pair of theta correspondence. This means that even orthogonal-symplectic and odd orthogonal-symplectic dual pairs can lead to two different choices of Lusztig correspondence on symplectic groups. Consequently, Pan's result cannot encompass theta lifting for all dual pairs simultaneously. To describe theta lifting comprehensively, it is necessary to assign names to the irreducible representations or, equivalently, determine a specific choice of Lusztig correspondence. However, once Pan's $\CL_s$ is chosen for one dual pair, the theta lifting information for the other dual pair is lost. Therefore, the primary objective of this paper is to establish a canonical choice of $\CL_s$ that is independent of dual pairs. This $\CL_s$ will enable the determination of the relationship between theta correspondences for all dual pairs.

\subsection{A finite field instance of  relative Langlands duality}
For a symmetric space $V_\Fn$, define its discriminant by
\[
\disc(V_\Fn)=(-1)^{\frac{\Fn(\Fn-1)}{2}}\det(V_\Fn)\in \Fq^{\times }/\Fq^{\times 2}.
\]
Write $\o^\epsilon_{\Fn}$ to be the isometry group of $V_\Fn$ with $\disc(V_\Fn)=\epsilon$. For even $\Fn$, the even orthogonal group $\o^\epsilon_{\Fn}$  is split if and only if $\epsilon=+$. For odd $\Fn$, we have $\o^+_{\Fn}$ and $ \o^-_{\Fn}$ are the same as abstract groups; however, they act on two symmetric spaces with different discriminants.

To relate two theta correspondence of dual pairs $(\sp_{2n},\o^\epsilon_{2n'})$ and $(\sp_{2n},\rm{O}^{\epsilon'}_{2n'+1})$, it is natural to consider the restriction of representations from $\rm{O}^{\epsilon'}_{2n'+1}\fq$  to $\o^\epsilon_{2n'}\fq$.
Moreover, the Gan-Gross-Prasad problem, which can be seen as a generalization of the above restriction problem, should also be involved.

 In the case of local fields, Gan-Gross-Prasad problem \cite{GP1, GP2, GGP1, GGP2} is related to theta correspondence in the sense of relative Langlands duality of Ben-Zvi-Sakellaridis-Venkatesh \cite{BZSV}. Let us now recall the relevant context and background as laid out
in \cite{BZSV}. The relative Langlands program aims to study the relationship between certain periods of automorphic forms on a reductive group $G$ and  certain special values of automorphic L-functions attached to its dual group $G^*$. Despite previous knowledge of some instances of this relationship, they were not well understood in a systematic way. However, there has been recent progress in elucidating these relationships \cite{SV,BZSV}, employing the framework of spherical varieties. This can be seen as an arithmetic analogue of the electric-magnetic duality of boundary conditions in four-dimensional supersymmetric Yang-Mills theory. One noteworthy example of this duality is the duality between the Gan-Gross-Prasad problem and theta correspondence. This specific instance is examined in \cite{GW1,GW2}.

 In the case of finite fields, a duality between Gan-Gross-Prasad problem and theta correspondence for unipotent representation has been proven in \cite{LW4, Wang1}. To be more precise, the multiplicity in Gan-Gross-Prasad problem for unipotent representations can be visualized as a diagram:
  \[
\xymatrix{
\pi\in\cal{E}(\sp_{2n},1) \ar[rrrr]^{\textrm{Theta lifting}} \ar[dd]_{\cal{L}_1\circ\cal{AC}} &&&&\bigoplus\pi',\  \pi'\in \cal{E}(\o^\epsilon_{2m},1) \ar[dd]_{\cal{AC}\circ\cal{R}} \\
\\
\pi^\star\in\cal{E}(\so_{2n+1},1) \ar[rrrr]^{\textrm{unipotent part of GGP} } &&&&
\bigoplus\pi^{\prime \star},\  \pi^{\prime \star}\in \cal{E}(\so^\epsilon_{2m},1),
}
\]
where $\cal{L}_1$ is the Lusztig correspondence and $\cal{R}$ is the restriction form $\o^\epsilon_{2m}$ to $\so^\epsilon_{2m}$, and $\cal{AC}$ is the Alvis-Curtis duality. In this paper, we generalize above commutative diagram to all representations under the canonical choice of Lusztig correspondence.

\subsection{Main results}

We recall the description of $C_{G^*}(s)$ and $\E(C_{G^*}(s),1)$ (cf. section \ref{sec4}). We have a decomposition of $C_{G^*}(s)$ via eigenvalues of $s$ as follows:
\begin{equation}
 \begin{aligned}
C_{G^*}(s)
\cong &
G^*_{[ \ne \pm 1]}(s)\times G^*_{[ 1]}(s)\times G^*_{[ -1]}(s)\times
\begin{cases}
\{\pm\},&\textrm{if }G=\o_{2n+1};\\
1,&\textrm{otherwise}.
\end{cases}
 \end{aligned}
\end{equation}
where $G^*_{[ \ne \pm 1]}(s)$ (resp. $G^*_{[ 1]}(s)$, $G^*_{[ -1]}(s)$) is the part of $C_{G^*}(s)$ corresponding to the eigenvalue $\ne 1$ (resp. $1$, $-1$) of $s$ (cf. section \ref{sec4}).
In particular, we have a explicit description of $G^*_{[a]}(s)$ as follows:
\begin{itemize}
    \item The group $G^*_{[ \ne \pm 1]}(s)$ is a product of general linear groups and unitary groups;
\item
If $G$ is a odd orthogonal group, then $G^*_{[ 1]}(s)$ and $G^*_{[ -1]}(s)$ are symplectic groups;
\item
If $G$ is a symplectic group, then $G^*_{[ 1]}(s)$ is a special odd orthogonal group and $G^*_{[ -1]}(s)$ is a even orthogonal group;
\item
If $G$ is a even orthogonal group, then $G^*_{[ 1]}(s)$ and $G^*_{[ -1]}(s)$ are even orthogonal groups.
\end{itemize}
Recall that Lusztig established a bijection between the unipotent representations of orthogonal groups and symplectic groups to equivalence classes of Lusztig's symbols (see section \ref{sec22}). Subsequently, choosing $\CL_s$ leads to a bijection
\[
\pi\in\E(G,s)\to (\rho,\Lambda_1,\Lambda_{-1},\iota)
\]
where $\rho$ is the projection of $\CL_s(\pi)$ on $G^*_{[ \ne \pm 1]}$, and $\Lambda_{\pm1}$ are the symbol corresponding to the projection of $\CL_s(\pi)$ on $G^*_{[ \pm 1]}$, and $\iota\in\{\pm\}\cong\{\bf{1},\rm{sgn}\}$ only occur for $G_n=\o_{2n+1}$. We denote $\pi$ by $\prllsgn$.

\begin{theorem}\label{main1}
 Suppose that the cardinality of $\Fq$ is large
enough so that the main result in \cite{S} holds.
Then, for each symplectic group and orthogonal group, there is a unique choice of Lusztig correspondence denoted by $\clc$. This choice is compatible with parabolic induction and the theta correspondence (cf. Theorem \ref{main}).
\end{theorem}

To outline the strategy of constructing a canonical choice of $\clc$ on $\pi\in\E(G)$, we start by considering an irreducible representation $\pi'$ occurring in the first descent of $\pi$. This representation is of a group that is smaller than $G$. Here the first descent means that
\begin{itemize}
\item the multiplicity $m_\epsilon(\pi,\pi^\star)\ne 0$ (resp. $m_\epsilon(\pi',\pi^{\star\prime})\ne 0$);
\item the multiplicity $m_{\pm}(\pi,\sigma)\equiv 0 $ (resp. $m_\pm(\pi',\sigma)$)) if $\sigma$ lives in a group smaller than the group of $\pi^\star$ (resp. $\pi^{\star\prime}$),
\end{itemize}
where $m_\epsilon(\pi,\pi^\star)$ is the multiplicity in finite Gan-Gross-Prasad problem defined in section \ref{sec6.1} and \ref{sec6.2}.
By induction on the rank of $G$, we have the existence of the canonical Lusztig correspondence $\clc$ on $\pi'$. We define $\clc$ on $\pi$ to be the unique Lusztig correspondence that is compatible with $\clc$ on $\pi'$.  Since $\clc$ on $\pi'$ is compatible with theta correspondence and the first descent $\pi'$ of $\pi$  has similar behavior with $\pi$ in theta correspondence, the Lusztig correspondence $\clc$ on $\pi$ naturally becomes compatible with the theta correspondence (cf. section \ref{sec66}).

Under the canonical choice of Lusztig correspondence in Theorem \ref{main1}, we distinguish representations in representations in $[\pi]$ as follows.
\begin{theorem}\label{main2}
Let $G$ be a symplectic group or orthogonal group. Let $\chi_G$ be the spinor norm of $G$ (cf. \cite[1.6]{Wal1}). Under the canonical choice of Lusztig correspondence in Theorem \ref{main1}, we have
\[
\begin{array}{c|cccc}
G& \pi                                 &\sgn\pi       &\pi^c             &\chi_G\cdot\pi\\ \hline
\sp_{2n}        & \pi_{\rho,\Lambda_1,\Lambda_{-1}} &\empty & \pi_{\rho,\Lambda_1,\Lambda_{-1}^t} &\empty\\
\o^\epsilon_{2n}        & \pi_{\rho,\Lambda_1,\Lambda_{-1}} &\pi_{\rho,\Lambda_1^t,\Lambda_{-1}^t} & \pi_{\rho,\Lambda_1,\Lambda_{-1}^t} &\pi_{\rho^-,\epsilon(\Lambda_{\pm1})\Lambda_{-1},\epsilon(\Lambda_{\pm1})\Lambda_{1}}\\
\o^\epsilon_{2n+1}        & \pi_{\rho,\Lambda_1,\Lambda_{-1},\iota} &\pi_{\rho,\Lambda_1,\Lambda_{-1},-\iota} & &\pi_{\rho^-,\Lambda_{-1},\Lambda_{1},\iota'\iota}\\
\end{array}
\]
where $\Lambda^t$ is defined in section \ref{sec:L-s}, and $\rho^-$ is defined in \eqref{rho-}, $\iota'=\chi_G(-I)$ and $\epsilon(\Lambda_{\pm1}):=(-1)^{\rm{def}(\Lambda_1)+\rm{def}(\Lambda_{-1})}$.

\end{theorem}

Theorem \ref{main1} has an important application in the finite Gan-Gross-Prasad problem. In \cite{Wang1, Wang2}, the calculation of the finite Gan-Gross-Prasad problem has been provided. However, the calculation in \cite{Wang1} depends on how representations behave in theta correspondence. By Theorem \ref{main1} or Theorem \ref{main}, the behavior of representations $\prll$ or $\prllsgn$ in theta correspondence can be determined based on their subscripts. Hence we can calculate the multiplicity in the finite Gan-Gross-Prasad problem using these subscripts and provides an explicit description of the finite Gan-Gross-Prasad problem in a combinatorial way (Theorem \ref{ggpmain}).

Combining Theorem \ref{main1} and Theorem \ref{ggpmain}, we establish the commutativity between the theta correspondence and the multiplicity in the Gan-Gross-Prasad problem. This commutativity can be regarded as a finite field instance of the relative Langlands duality of Ben-Zvi-Sakellaridis-Venkatesh \cite{BZSV}. To elaborate, we can visualize the multiplicity in the Gan-Gross-Prasad problem in the following manner:

\begin{theorem}
 \begin{enumerate}
\item Let $\pi\in\E(\sp_{2n})$, and $\pi^\star\in \E(\sp_{2n^\star})$.
Then we have the following commutative diagram up to a twist of the sgn characters on orthogonal groups:
  \[
\xymatrix{
\pi\in\cal{E}(\sp_{2n}) \ar[rrrr]^{\textrm{GGP}} \ar[dd]_{\cal{AC}\circ\clc} &&&&\bigoplus\pi^\star,\  \pi^\star\in \cal{E}(\sp_{2n^\star},s) \ar[dd]_{\cal{AC}\circ\clc} \\
\\
\rho\otimes\pi_{\Lambda_1}\otimes\pi_{\Lambda_{-1}}  \ar[rrrr]^{\textrm{ theta lfiting }} &&&&
\bigoplus\pi^{\star },\  \rho^\star\otimes\pi_{\Lambda^\star_1}\otimes\pi_{\Lambda^\star_{-1}}
}
\]
where the sum in the bottom right corner runs over representations such that
\begin{itemize}
\item $\rho=\prod_{[a],a\ne \pm 1}\pi[a]$ and $\rho^\star=\prod_{[a],a\ne \pm 1}\pi^\star[a]$ (cf. section \ref{sec4});
\item For each $[a]$, $\pi^\star[-a]$ appears in theta lifting of $\pi[a]$;
\item $\clc(\pi^\star_{\Lambda_1})$ appears in theta lifting of $\pi_{\Lambda_{-1}}$;
\item $\pi^\star_{\Lambda_{-1}}$ appears in theta lifting of $\clc(\pi_{\Lambda_1})$.
\end{itemize}

\item Let $\pi\in\E(\o^\epsilon_{2n})$, and $\pi^\star\in \E(\o^{\epsilon^\star}_{2n^\star+1})$.
Then we have the following commutative diagram up to a twist of the sgn character on orthogonal groups:
  \[
\xymatrix{
\pi\in\cal{E}(\o^\epsilon_{2n}) \ar[rrrr]^{\textrm{GGP}} \ar[dd]_{\cal{AC}\circ\clc} &&&&\bigoplus\pi^\star,\  \pi^\star\in \cal{E}(\o^{\epsilon^\star}_{2n^\star},s) \ar[dd]_{\cal{AC}\circ\clc} \\
\\
\rho\otimes\pi_{\Lambda_1}\otimes\pi_{\Lambda_{-1}}  \ar[rrrr]^{\textrm{ theta lfiting }} &&&&
\bigoplus\pi^{\star },\  \rho^\star\otimes\pi_{\Lambda^\star_1}\otimes\pi_{\Lambda^\star_{-1}}\otimes\iota^\star
}
\]
where the sum in the bottom right corner runs over representations such that
\begin{itemize}
\item $\rho=\prod_{[a],a\ne \pm 1}\pi[a]$ and $\rho^\star=\prod_{[a],a\ne \pm 1}\pi^\star[a]$
\item For each $[a]$, $\pi^\star[a]$ appears in theta lifting of $\pi[a]$;
\item $\clc(\pi^\star_{\Lambda_1})$ appears in theta lifting of $\pi_{\Lambda_1}$;
\item $\clc(\pi^\star_{\Lambda_{-1}})$ appears in theta lifting of $\pi_{\Lambda_{-1}}$.
\end{itemize}

\end{enumerate}
\end{theorem}

\begin{remark}

The sgn characters for unipotent representations are explicitly given in Theorem \ref{8.3}. For the general cases, the sign characters can be calculated using Theorem \ref{ggpmain}. However, this process  for the general cases can be extremely complicated, as the variables $\rho$, $\Lambda_1$, and $\Lambda_{-1}$ can all influence each other.
\end{remark}

Although the description of the multiplicity in finite Gan-Gross-Prasad problem in Theorem \ref{ggpmain} is quite complicated, the first descent case of  Gan-Gross-Prasad problem, which is similar to the first occurrence case of theta lifting, is simple.

\begin{theorem}\label{main4}
 Let $\pi\in \E(G_n)$ where $G_n$ is either a symplectic group, odd orthogonal group or a even orthogonal group. Then we have the following commutative diagram
 \begin{equation}\label{maincm}
\xymatrix{
&\pi^*\ar[rrr]^{\Theta^*_{n,n'}} \ar[dd]_{\CD_\epsilon^*}& &&\pi^{\prime *} \ar[dd]^{\CD^*_\epsilon}\\
\pi \ar[ur]^{\CL^\rm{can}}\ar[dd]_{\CD_{\ell,\epsilon}}\ar[rrr]^{\Theta_{n,n'}}&&&\pi' \ar[ur]^{\CL^\rm{can}} \ar[dd]^{\CD_{\ell',\epsilon}}& \\
&\pi^{\star*} \in \CD^*_\epsilon(\pi^*) \ar[rrr]^{\Theta^*_{m,m'}}& &&\pi^{\star\prime*}\in \CD^*_\epsilon(\pi^{\prime*})\\
\pi^\star\in \CD_{\ell,\epsilon}(\pi)\ar[ur]^{\CL^\rm{can}}\ar[rrr]^{\Theta_{m,m'}}&&&\pi^{\star\prime}\in \CD_{\ell,\epsilon}(\pi')\ar[ur]^{\CL^\rm{can}}&\\
},
\end{equation}
where all notations are defined in section \ref{sec7}.
\end{theorem}

Let us explain the diagram \eqref{maincm} roughly. Here $\pi^\star$ (resp. $\pi^{\star\prime}$) is a irreducible representation occurs in the first descent of $\pi$ (resp. $\pi'$). The index $n'$ in the theta lifting $\Theta_{n,n'}$ is the first occurrence index of $\pi$. Therefore, Theorem \ref{main4} gives a duality between the first descent case of finite Gan-Gross-Prasad problem and first occurrence case of theta correspondence.

%
%

 This paper is organized as follows. In Section \ref{sec22}, we recall the classification  of unipotent representations \cite{L3,L4,LS}. In Section \ref{sec4}, we recall the theory of Deligne-Lusztig characters and Lusztig correspondence. Then we discuss the ambiguity of Lusztig correspondence. In Section \ref{sectheta}, we review some results on the theta correspondence for finite symplectic groups and finite orthogonal groups. We also compute the first occurrence indexes and conservation relations, which will help to resolve any ambiguity in the Lusztig correspondence.
  In Section \ref{sec6}, we recall the  Gan-Gross-Prasad problem, and establish conservation relations for it, which are the dual to conservation relations of the theta correspondence. In Section \ref{sec66}, we have proven Theorem 6, our main result. This theorem provides a choice of Lusztig correspondence that is compatible with parabolic induction and the theta correspondence. In Section \ref{sec7}, we calculate the multiplicity
in the finite Gan-Gross-Prasad problem and established the duality between the Gan-Gross-Prasad problem and theta correspondence by the Lusztig correspondence described in Theorem \ref{main}.

\subsection*{Acknowledgement}  The author acknowledge generous support provided by National Natural Science Foundation of PR China (No. 12201444).

\section{Classification of unipotent representations}\label{sec22}

In this section we assume that $\Fq$ is a finite field of odd characteristic.

\subsection{Lusztig's symbols} \label{sec:L-s}
In this section, we follow the notation of \cite{P1} to review Lusztig's symbol, which is slightly different from that of \cite{L1}.

Let $\lambda=[\lambda_1,\lambda_2,\dots]$ be a partition of $n$ and  denote $|\lambda|$ to the length $\sum_i   \lambda_i$.
We often write $\lambda=[\lambda_1^{k_1},\cdots,\lambda_l^{k_l}]$ where $\lambda_j$ are distinct nonzero parts and $k_j$ are their multiplicity.
As is standard, we realize partitions as Young diagrams and denote by $\lambda^t=[\lambda^t_1,\lambda^t_2,\cdots]$ the transpose of $\lambda$. A partition $\lambda$ is called {\it even} if  all multiplicities $k_i$ are even.

We say that two partitions $\lambda=[\lambda_1,\lambda_2,\dots]$ and $\lambda'=[\lambda'_1,\lambda'_2,\dots]$ are {\it close} if $|\lambda_i-\lambda_i'|\leqslant 1$ for every $i$.
Write
\[
\lambda\cap\lambda' :=[\lambda_i]_{\{ i | \lambda_i=\lambda_i'\}}
\]
be the partition formed by the common parts of $\lambda$ and $\lambda'$. Following \cite{AMR}, we say that $\lambda$ and $\lambda'$ are {\it 2-transverse} if they are close and $\lambda\cap \lambda'$ is even.
In particular, if $\lambda$ and $\lambda'$ are close and $\lambda\cap \lambda'=\varnothing$, then $\lambda$ and $\lambda'$ are 2-transverse, and in this case we say that they are {\it transverse}. For example, let
$\lambda=[\lambda_1,\ldots, \lambda_k]$ be a partition of $n$, and let
\[
\lambda_\star=[\lambda_2,\ldots, \lambda_k]
\]
be the partition of $n-\lambda_1$ obtained by removing the first row of $\lambda$. Then $\lambda^t$ and $\lambda_{\star }^t$ are 2-transverse. Moreover, $\lambda_{\star }^t$ is the unique partition of $n-\lambda_1$ such that $\lambda^t$ and $\lambda_{\star }^t$ are 2-transverse.

A {\it symbol} is an array of the form
\[
\Lambda=
\begin{pmatrix}
A\\
B
\end{pmatrix}
=
\begin{pmatrix}
a_1,a_2,\cdots,a_{m_1}\\
b_1,b_2,\cdots,b_{m_2}
\end{pmatrix}
\]
of two finite sets $A$, $B$ of natural numbers (possibly empty) with $a_i, b_i\geqslant 0$, $a_i>a_{i+1}$ and $b_i>b_{i+1}$.
The {\it rank} and {\it defect} of a symbol $\Lambda$ are defined by
\[
\begin{aligned}
&\rank(\Lambda):=\sum_{a_i\in A}a_i+\sum_{b_i\in B}b_i-\left\lfloor\left(\frac{\#A+\#B-1}{2}\right)^2\right\rfloor, \\
&\textrm{def}(\Lambda):=\#A-\#B.
\end{aligned}
\]
 Note that the definition of ${\mathrm{def}}(\Lambda)$ differs from that of \cite[p.133]{L1}.

For a symbol $\Lambda=\binom{A}{B}$, let $\Lambda^*$ (resp. $\Lambda_*$) denote the first row (resp. second row) of $\Lambda$, i.e. $\Lambda^*=A$ and $\Lambda_*=B$,
and let
$\Lambda^t := \binom{B}{A}$. Write
\[
\sd(\Lambda):=\left\{
\begin{array}{ll}
\rm{sgn}(\rm{def}(\Lambda)),&\textrm{ if }\rm{def}(\Lambda)\ne 0;\\
+,&\textrm{ if }\Lambda=\binom{1}{0};\\
-,&\textrm{ if }\Lambda=\binom{0}{1};\\
\end{array}
\right.
\textrm{ and }\pd(\Lambda):=(-1)^{\left\lfloor\frac{|\rm{def}(\Lambda)|}{2}\right\rfloor}
\]
to be the sign and the parity of $\rm{def}(\Lambda)$, respectively. Here for the case of $\rm{def}(\Lambda)=0$, we only care about the basic case $\Lambda=\binom{1}{0}$ or $\binom{0}{1}$ (cf. Section \ref{sec2}).   Clearly, we have
\[
\sd(\Lambda)=-\sd(\Lambda^t)\textrm{ and }\pd(\Lambda)=\pd(\Lambda^t).
\]

%

\subsection{Bi-partitions} \label{sec:bi}
For partitions $\lambda=[\lambda_1,\lambda_2,\cdots,\lambda_k]$ and $\mu=[\mu_1,\mu_2,\cdots,\mu_l]$, we write
\[
\lambda\preccurlyeq\mu\quad\textrm{if }\mu_i-1\leqslant \lambda_i \leqslant \mu_i\textrm{ for each }i.
\]

Let $\mathcal{P}_2(n)$
denote the set of bi-partitions $\bpair{\begin{smallmatrix}
\lambda\\
\mu\end{smallmatrix}}$ of $n$, where $\lambda$, $\mu$ are partitions
such that $|\lambda| + |\mu| = n$. We can associate a bi-partition to each symbol as follows:
\begin{align*}
\Upsilon: &\Lambda
=\begin{pmatrix}
A\\
B
\end{pmatrix}
=\begin{pmatrix}
a_1,a_2,\dots,a_{m_1}\\
b_1,b_2,\dots,b_{m_2}
\end{pmatrix}\\
&\mapsto
\begin{bmatrix}
\lambda\\
\mu
\end{bmatrix} =
\begin{bmatrix}
a_1-(m_1-1),a_2-(m_1-2),\dots,a_{m_1-1}-1,a_{m_1}\\
b_1-(m_2-1),b_2-(m_2-2),\dots,b_{m_2-1}-1,b_{m_2}
\end{bmatrix}.
\end{align*}
This gives a bijection
\begin{equation}\label{bp}
\Upsilon:\mathcal{S}_{n,\beta}\to
\begin{cases}
\mathcal{P}_2(n-(\frac{\beta+1}{2})(\frac{\beta-1}{2})), &  \textrm{if }\ \beta\textrm{ is odd};\\
\mathcal{P}_2(n-(\frac{\beta}{2})^2), & \textrm{if }\ \beta\textrm{ is even},
\end{cases}
\end{equation}
where $\mathcal{S}_{n,\beta}$ denotes the set of symbols of rank $n$ and defect $\beta$. Then a symbol $\Lambda$ can be characterized by its defect and the bi-partition $\Upsilon(\Lambda)$.

Set
\[
\Upsilon(\Lambda)^\epsilon:=
  \begin{cases}
\lambda, &\textrm{ if }\epsilon=+;\\
\mu, &\textrm{ if }\epsilon=-.
\end{cases}
\]
%

To each symbol $\Lambda$, we associate its Alvis-Curtis dual ${}^{\bd}\Lambda$
 by requiring that
 \[
\textrm{def}({}^{\bd}\Lambda)=\textrm{def}(\Lambda)\quad
\text{ and }\quad
 \Upsilon({}^{\bd}\Lambda)^\epsilon=\left(\Upsilon(\Lambda)^{-\epsilon}\right)^t,
 \]
where $\epsilon=\pm$.



\subsection{Classification of unipotent representations of $\GGL_n\fq$ and $\UU_n(\Fq)$}

In \cite{LS}, Lusztig and Srinivasan have classified the representations of $G^F = \operatorname{GL}_n(\mathbb{F}_q)$ and $\operatorname{U}_n(\mathbb{F}_q)$. They have established a unique one-to-one correspondence between the unipotent representations of $G^F$ and the irreducible representations of the Weyl group $W_n$ of $G$.
It is well-known that $W_n\cong S_n$, where $S_n$ is the symmetry group, and irreducible representations of $S_n$ are parametrized by partitions of $n$. We denote the representation corresponding to a partition $\lambda$ of $n$ by $\pi_\lambda$. For example, $\pi_{(n)}$ corresponds to the trivial representation of $G^F$. By Lusztig's result \cite{L1},  $\pi_\lambda$ is a unipotent cuspidal representation of $\UU_n(\Fq)$ if and only if $n=\frac{k(k+1)}{2}$ for some positive integer $k$ and  $\lambda=[k,k-1,\cdots,1]$.

\subsection{Classification of unipotent representations of symplectic groups and orthogonal groups}\label{sec3.4}

We now recall the classification of irreducible unipotent representations of symplectic and orthogonal groups. Let $\cal{E}(G_n,1)$ be the collection of unipotent representations of $G_n^F=\sp_{2n}\fq$, $\rm{SO}_{2n+1}\fq$, $\o_{2n+1}\fq$ and $\o^\pm_{2n}\fq$. Lusztig gives a bijection between the unipotent representations of these groups to equivalence classes of symbols. Then we have the following bijections:
\[
\left\{
\begin{aligned}
&\cal{E}(\sp_{2n},1)\\
&\cal{E}(\so_{2n+1},1)\\
&\cal{E}(\o_{2n+1},1)\\
&\cal{E}(\o^+_{2n},1)\\
&\cal{E}(\o^-_{2n},1)
\end{aligned}\right.
\longrightarrow
\left\{
\begin{aligned}
&\cal{S}_n:=\big\{\Lambda|\rm{rank}(\Lambda)=n, \rm{def}(\Lambda)=1\ (\textrm{mod }4)\big\};\\
&\cal{S}_n;\\
&\cal{S}_n\times\{\pm\};\\
&\cal{S}^+_n:=\big\{\Lambda|\rm{rank}(\Lambda)=n, \rm{def}(\Lambda)=0\ (\textrm{mod }4)\big\};\\
&\cal{S}^-_n:=\big\{\Lambda|\rm{rank}(\Lambda)=n, \rm{def}(\Lambda)=2\ (\textrm{mod }4)\big\}.
\end{aligned}\right.
\]

We require that the above bijections satisfy \cite[(3.13)]{P3}, which is a reformulation of \cite[theorem 5.8]{L3} and \cite[theorem 5.13]{L4}. Let $\pi_{\Lambda}$ and $\pi_{\Lambda,\iota}$ be the irreducible representations parametrized by $\Lambda$ and $(\Lambda,\iota)$, respectively.
For symplectic and special odd orthogonal groups, the bijection is uniquely determined by \cite[(3.13)]{P3}. For even orthogonal groups, the bijection is characterized by additional conditions mentioned in \cite[section 5.3]{P3}. It is known that $\pi_{\Lambda^t} = \rm{sgn}\cdot\pi_{\Lambda}$.
In the case of odd orthogonal groups, we differentiate $\pi_{\Lambda,\pm}$ by deciding that $\pi_{\Lambda,\pm}(-1)=\pm\rm{Id}$.
Regarding the representation of the trivial symplectic group $\sp_0\fq$, we  associated it with the symbol $\left(\begin{smallmatrix}
0\\
-
\end{smallmatrix}\right)$.
As for the trivial group $\o^+_0\fq$, its trivial representation is associated with the symbol $\left(\begin{smallmatrix}
-\\
-
\end{smallmatrix}\right)$.

\begin{example}
\begin{enumerate}
\item (\cite[Lemma 4.6]{P3}) Let $G=\sp_{2n}$, and $\pi_\Lambda={\bf 1}$, the trivial representation of $G^F$. Then
\[
\Lambda=\begin{pmatrix}
n\\
-
\end{pmatrix}.
\]
\item (\cite[Corollary 5.10]{P3}) Let $G=\o^\epsilon_{2n}$, and $\pi_\Lambda={\bf 1}$, the trivial representation of $G^F$. Then
\[
\Lambda=\left\{
\begin{array}{ll}
\begin{pmatrix}
n\\
0
\end{pmatrix}& \textrm{ if }\epsilon=+;\\
\\
\begin{pmatrix}
-\\
n,0
\end{pmatrix}& \textrm{ if }\epsilon=-.
\end{array}
\right.
\]
\item  Let $G=\o^\epsilon_{2n+1}$, and $\pi_{\Lambda,\iota}={\bf 1}$, the trivial representation of $G^F$. Then
\[
\Lambda=\begin{pmatrix}
n\\
-
\end{pmatrix}
\]
and $\iota=+.$
\end{enumerate}
\end{example}

\begin{proposition}\label{unp}
Let $\pi_{\Lambda,\iota}\in \CE(G_n, \pi_{\Lambda',\iota'})$ (see section \ref{sec2} for the definition of $\CE(G_n, \pi_{\Lambda',\iota'})$), where $\iota,\iota'\in\{\pm\}$, and $\pi_{\Lambda',\iota'}$ is a unipotent cuspidal representation of $G_m^{ F}$ with $m\le n$, and $\iota$ and $\iota'$ only occurs for odd orthogonal groups. Then we have $\iota=\iota'$ and $\rm{def}(\Lambda)=\rm{def}(\Lambda')$.

\end{proposition}
\begin{proof}
Lusztig's bijection $\cal{E}(G_n,1)\to \cal{S}_n$ is compatible with the parabolic induction for symplectic and special odd orthogonal groups. This result is discussed in \cite{L1}. The odd orthogonal case follows from the special odd orthogonal case. For even orthogonal groups, the compatibility with the parabolic induction is one of the additional conditions mentioned in \cite[section 5.3]{P3}.
\end{proof}

\subsection{Symbols of unipotent cuspidal representations}

 In \cite{L1}, we know that $\sp_{2n}\fq$ (resp,  $\rm{SO}_{2n+1}\fq$, $\rm{O}^{\epsilon}_{2n}\fq$) has a unique irreducible unipotent cuspidal representation if and only if $n=k(k+1)$ (resp. $n=k(k+1)$, $n=k^2$). We regarded the trivial representations of $\sp_{0}\fq$ and $\o^+_0\fq$ as unipotent cuspidal representations.

For $\sp_{2k(k+1)}\fq$, the unique unipotent cuspidal representation $\pi_{\Lambda}$ is associated to the symbol:
\begin{equation}\label{cussp}
\Lambda=\left\{
\begin{aligned}
&\begin{pmatrix}
2k,2k-1,\cdots,1,0\\
-
\end{pmatrix}\textrm{ if }k\textrm{ is even};\\
&\begin{pmatrix}
-\\
2k,2k-1,\cdots,1,0
\end{pmatrix}\textrm{ if }k\textrm{ is odd};\\
\end{aligned}\right.
\end{equation}
of defect $(-1)^k(2k + 1)$.

 For $\o^{\epsilon}_{2k^2}\fq$ with $\epsilon=(-1)^k$, there are two unipotent cuspidal representations $\pi_{\Lambda}$ and $\sgn\pl=\pi_{\Lambda^t}$ where \[\Lambda=\begin{pmatrix}
2k-1,2k-2,\cdots,1,0\\
-
\end{pmatrix}.\]
For $\o_{2k(k+1)+1}\fq$, there are two unipotent cuspidal representations $\pi_{\Lambda,+}$ and $\sgn\pi_{\Lambda,+}=\pi_{\Lambda,-}$ where \begin{equation}\label{cusso}
\Lambda=\left\{
\begin{aligned}
&\begin{pmatrix}
2k,2k-1,\cdots,1,0\\
-
\end{pmatrix}\textrm{ if }k\textrm{ is even};\\
&\begin{pmatrix}
-\\
2k,2k-1,\cdots,1,0
\end{pmatrix}\textrm{ if }k\textrm{ is odd}.\\
\end{aligned}\right.
\end{equation}

\section{Lusztig correspondence}\label{sec4}

In this section we assume that $\Fq$ is a finite field of odd characteristic.
\subsection{The structure of Lusztig correspondence}
To parametrize the irreducible representations of $G^F$, we first follow Lusztig's Jordan decomposition (see  \cite[Section 2.3]{Ma} and \cite[Theorem 5.2]{Gec} for instance) to
partition $\Irr(G^F)$ according to the semisimpe conjugacy classes in the dual group $G^{*F}$ of $G^F$, that is,
\begin{equation}\label{eq:irr-decomp-1}
\Irr(G^F)=\coprod_{(s)\in (G^{*})^{\circ F}_{ss}/ \Ad(G^{*F})}\mathcal{E}(G,s)
\end{equation}
where $(s)$ runs over $G^{*F}$-conjugacy classes of semisimple elements in  the identity component of $G^{*F}$, and
 the Lusztig series $\mathcal{E}(G,s)$ is defined by
\[
\mathcal{E}(G,s) := \{ \pi \in \textrm{Irr}(G^F)  | \langle \pi,  R_{T^*,s}^G\rangle \ne 0\textrm{ for some }T^*\textrm{ containing }s \}.
\]
Here if $G$ is a connected reductive group, one has a bijection between the set of $G^F$-conjugacy classes of the pairs $(T, \theta)$ where $T$ is an $F$-stable maximal torus of $G$ and $\theta$ is a character of $T^F$, and the set of $G^{*F}$-conjugacy classes of the pairs $(T^*, s)$ where $T^*$ is an $F$-stable maximal torus of $G^*$ and contains $s$.
If $(T, \theta)$ corresponds to $(T^*, s)$, we denote  by $R_{T^*,s}^G$  the Deligne-Lusztig character $R^G_T(\theta)$ of $G^F$ in \cite{DL}.
And if $G$ is non-connected, we define  $R^G_{T^*,s}:=\Ind^{G^F}_{(G^0)^F}R^{G^0}_{T^*,s}$.

Continuing \eqref{eq:irr-decomp-1},
we recall the correspondence between the Lusztig series $\CE(G,s)$ and the unipotent representations $\CE(C_{G^*}(s),1)$ of $C_{G^{*}}(s)^F$ in \cite[Chap 9]{L1}, \cite[Theorem 8.14]{Sr2} and  \cite[Lemma 2.7]{Ma},
where $C_{G^*}(s)$ is the centralizer  of $s$ in $G^*$.

\begin{proposition}[{\cite[Theorem 11.5.1, Proposition 11.5.6]{DM}}]\label{Lus}
Let $G$ be a reductive connected group defined over $\Fq$.
There is a bijection
\[
\mathcal{L}_s:\mathcal{E}(G,s)\to \mathcal{E}(C_{G^{*}}(s), 1),
\]
extended by linearity to a map between virtual characters satisfying that
\begin{equation}\label{l1}
\mathcal{L}_s(\varepsilon_G R^G_{T^*,s})=\varepsilon_{C_{G^{*}}(s)} R^{C_{G^{*}}(s)}_{T^*,1},
\end{equation}
where $\varepsilon_G:= (-1)^{\mathrm{rk}\,G}$ and $\mathrm{rk}\,G$ is the $\Fq$-rank of $G$;
Moreover, $\CL_s$ sends cuspidal representations to cuspidal representations.
\end{proposition}
Lusztig's Jordan decomposition can be extended to even orthogonal groups (see \cite[Proposition 1.7]{AMR}).

Next, we recall the description of $C_{G^{*}}(s)$ in \cite[Section 1.B]{AMR} and then refine $\CE(C_{G^*}(s),1)$.
The multi-set of  eigenvalues of $s$ on the standard representation of $G^*$ is of the form
\[
(s)= \begin{cases}\{x_1,\ldots,x_n, x_1^{-1}, \ldots, x_n^{-1}, 1\},  & \textrm{if }G^*=\so_{2n+1}; \\
\{x_1,\ldots,x_n, x_1^{-1}, \ldots, x_n^{-1}\}, & \textrm{otherwise}.
\end{cases}
\]
For $a\in \overline{\Fq}^\times$, define
 \[
 \nu_{a}(s):=\#\{i | x_i \textrm{ or }x_i^{-1}=a\} \quad \textrm{and} \quad
  [ a] := \{a^{{q}^k},a^{-{q}^k} | k\in\bb{Z}\},
  \]
where $\#X$ denotes the cardinality of a finite set $X$.
Clearly, $\nu(a)$ and $[a]$ only depend on the orbit of $a$ under $\Gamma_{\Fq}$, the group of automorphisms of $\overline{\Fq}^\times$ generated by the Frobenius map $F$ and the inversion $a\mapsto a^{-1}$.
The centralizer $C_{G^*}(s)$ has a natural decomposition
\[
C_{G^*}(s)\cong \begin{cases}
\prod_{[a]}G^*_{[a]}(s)\times\{\pm\},&\textrm{if }G=\o_{2n+1};\\
 \prod_{[a]}G^*_{[a]}(s),&\textrm{otherwise},
\end{cases}
\]
where $[a]$ runs over the orbits $\overline{\Fq}^\times/\Gamma_{\Fq}$,
and $G^*_{[a]}(s)$ is the restriction of scalars $\Res_{{\Fq}(a)/{\Fq}}$ of a classical group defined over $\Fq(a)$, with $\overline{\Fq}$-rank  $\frac{1}{2}\#[a]\cdot\nu_{a}(s)$.
Clearly, $G^*_{[a]}(s)$ is the trivial group if $\nu_{a}(s)=0$.
For convenience, we set
\[
G^*_{[\ne\pm 1]}(s) :=\prod_{[a], \, a\ne\pm 1 }G^*_{[ a]}(s).
\]
Then we rewrite
\begin{equation}\label{decomp}
 \begin{aligned}
C_{G^*}(s)
\cong &
G^*_{[ \ne \pm 1]}(s)\times G^*_{[ 1]}(s)\times G^*_{[ -1]}(s)\times
\begin{cases}
\{\pm\},&\textrm{if }G=\o_{2n+1};\\
1,&\textrm{otherwise}.
\end{cases}
 \end{aligned}
\end{equation}

In particular, we have the following explicit description of $G^*_{[a]}(s)$:
\begin{itemize}
    \item
If $a\ne \pm 1$, then $G^*_{[a]}(s)$ is either $\Res_{{\Fq}(a)/{\Fq}}\GGL_{\nu_a(s)}$ or the unitary group $\Res_{{\Fq}(a)/{\Fq}}\UU_{\nu_a(s)}$;
\item
If $G=\o_{2n+1}$, then we may take $G^*=\sp_{2n}\times\{\pm 1\} =\sp(V^*)\times \{\pm1\}$, in which case
$
G^*_{[\pm 1]}(s)=\sp_{2\nu_{\pm 1}(s)}= \sp(V^*_{s=\pm 1});
$
\item
If $G=\sp_{2n}$, then we may take $G^*=\so_{2n+1}=\so(V^*)$, in which case
$G^*_{[ 1]}(s)=\so_{2\nu_{1}(s)+1}=\so(V^*_{s=1})$ and $G^*_{[-1]}(s)=\o_{2\nu_{-1}(s)}=\o(V^*_{s=-1})$;
\item
If $G=\o_{2n}(V)$, then $G^*=\o_{2n}(V^*)$ with $V^*=V$, and
$
G^*_{[\pm 1]}(s)=\o_{2\nu_{\pm 1}(s)} = \o(V^*_{s=\pm 1}).
$
\end{itemize}
Here $V^*$ is the (skew) symmetric space defining the dual group $G^*$,  $V^*_{s=\pm 1}$  is the eigenspace of $s$ on $V^*$  corresonding to the eigenvalue $\pm 1$.
We remark that when $G^*$ is orthogonal,  $\o(V^*_{s=\pm 1})$ depends on the discriminant of $V^*_{s=\pm 1}$.

Now, we can plug \eqref{decomp} into Proposition \ref{Lus}
and decompose
\begin{equation}  \label{lcor}
 \begin{aligned}
& \prod_{[ a] }\mathcal{E}(G^{*}_{[ a]}(s),1)
\nonumber\\
=\,& \cal{E}(G^*_{[\ne \pm 1]}(s)\times G^*_{[1]}(s)\times G^*_{[-1]}(s),1)
\times \begin{cases}
 \{\pm\},&\textrm{if }G=\o_{2n+1};\\
  1,&\textrm{otherwise},
\end{cases}
 \end{aligned}
\end{equation}
where by abuse of notation we also denote  $\{\pm\}$ the irreducible representations of $\{\pm \}\subset G^*$ if $G=\o_{2n+1}$.
We may write the image of $\CL_s$ accordingly
\[
\mathcal{L}_s(\pi)=\pi_{\ne\pm1}\otimes\pi_{\Lambda_1}\otimes\pi_{\Lambda_{-1}}\otimes\iota
\]
where $\epsilon \in \{\pm\}$ only occurs for $G=\o_{2n+1}$ and
\[
\pi_{\ne\pm1}=\prod_{[a],a\ne\pm 1}\pi{[a]} \textrm{ with }\pi_{[a]}\in \mathcal{E}(G^{*}_{[ a]}(s),1).
\]
Recall that irreducible unipotent representations of general linear groups and unitary groups are parametrized by partitions. We can associate each $\pi[a]$ with a partition $\lambda[a]$. Let $\cal{P}$ be the set of partitions, and $\cal{A}:=\{[a]\in \overline{\Fq}^\times/\Gamma_{\Fq}|a\ne \pm1 \}$. Let
$\cal{P}^\cal{A}$ be the set
\[
\cal{P}^\cal{A}=\left\{
\begin{matrix}
\rho:&\cal{A}&\to&\cal{P}\\
&[a]&\to&\rho[a]
\end{matrix}
 \right\}.
\]
Therefore, we can identify $\pi_{\ne\pm1}$ with a element $\rho\in \cal{P}^\cal{A}$, and denote $\pi$ by $\pi^{\CL_s}_{\rho,\Lambda_1,\Lambda_{-1},\iota}$ or simply $\prllsgn$, when no confusion raise. By abuse of notations, we regard $\rho$ as a irreducible representation in $ \prod_{[a],a\ne \pm1}\mathcal{E}(G^{*}_{[ a]}(s),1)$, and write
\[
\mathcal{L}_s(\prllsgn)=\rho\otimes\pi_{\Lambda_1}\otimes\pi_{\Lambda_{-1}}\otimes\iota.
\]

Let $\rho\in \cal{P}^\cal{A}$. We define $\rho^-$ to be an element in $\cal{P}^\cal{A}$ such that
\begin{equation}\label{rho-}
\rho^-:[a]\to\rho[-a],
\end{equation}
and we define $\rho_1$ to be an element in $\cal{P}^\cal{A}$ such that
\begin{equation}\label{rho1}
\rho_1:[a]\to\rho_1[a]=\rho[a]_\star,
\end{equation}
which means that $\rho_1[a]$ is the partition obtained by removing the first row of $\rho[a]$.

For a fixed $\rho\in \cal{P}^\cal{A}$, there exists a unique irreducible representation $\pi^{\sp}_\rho$ of a symplectic group such that
the $\pm 1$ part of $\pi^{\sp}_\rho$ is trivial and the $\ne \pm 1$ part of $\pi^{\sp}_\rho$ corresponds to $\rho$. As we will see in Section \ref{sec4.1},
$\pi^{\sp}_\rho$ does not depend on the choice of Lusztig correspondence. Assume that $\pi^{\sp}_\rho\in \cal{E}(\sp_{2n},s)$.
Write
\[
|\rho|=n,
\]
and
\[
\iota_\rho:=\pi^\sp_{\rho}(-I)\in\{\pm\}, \textrm{ and }{\widetilde{\iota}}_\rho:=\iota_\rho\ee^{|\rho|}
\]
and
\[
\epsilon_\rho:=(-1)^{\#\{ [a]\ | \ G^{*}_{[a]}(s)\text{ is unitary}\}},
\]
where $\ee$ is the square class of $-1$ in $\Fq^\times/(\Fq^\times)^2$.
\subsection{The ambiguity of Lusztig correspondence}\label{sec4.1}
It's worth noting that the choice of $\CL_s$ is usually not unique. However, if $G$ is a connected group with a connected center, we can determine $\CL_s$ uniquely by imposing some extra conditions \cite[Theorem 7.1]{DM}. Additionally, Pan has proved in \cite{P3} that for a fixed dual pair, there is a unique choice of $\CL_s$ that is compatible with theta correspondence. But, it's important to note that Pan's choice of $\CL_s$ depends on the dual pair $(G, G')$ and not the group $G$. Thus, even orthogonal-symplectic dual pairs and odd orthogonal-symplectic dual pairs can give us two different choices of the Lusztig correspondence on symplectic groups.

Let $\CV(G)^\#$ be the subspace of the space of (complex-valued) class functions on $G$ spanned by Deligne-Lusztig virtual characters of $G$. We define $\pi^\#$ as the orthogonal projection of $\pi$ on $\CV(G)^\#$. We call $\pi$ uniform if $\pi=\pi^\#$.

 \begin{definition}
Let $\pi$ be an irreducible representation of a symplectic group or an orthogonal group. Let
$
[\pi]:=\{\pi' | (\pi')^\#=\pi^\#\}.
$
\end{definition}
We cannot differentiate between representations in $[\pi]$ using \eqref{l1}. Therefore, we must determine to what extent an irreducible representation $\pi$ can be identified by its uniform projection $\pi^\#$.
According to \eqref{l1}, we have the following result.
\begin{proposition}\label{l2}
Let $\pi,\pi'\in \mathcal{E}(G,s)$.
\begin{enumerate}
\item If $\pi$ is uniform, then $\CL_s(\pi)$ does not depend on the choice of $\CL_s$.
\item  For any two different choice of Lusztig correspondences $\CL_s$ and $\CL_s'$, we have $\CL_s(\pi)\in[\CL'_s(\pi)]$.
\item We have $\pi\in [\pi']$ if and only if $\CL_s(\pi)\in[\CL_s(\pi')]$.
\end{enumerate}
\end{proposition}

Assume that $\pi\in \CE(G,s)$, and
\[
\mathcal{L}_s(\pi)=\rho\otimes\pi_{\Lambda_1}\otimes\pi_{\Lambda_{-1}}\otimes\iota,
\]
where $\iota\in \{\pm\}$ only occurs for $G=\o_{2n+1}$. Note that $\rho$ is uniform, so the $\rho$ part does not depends on the choice of $\CL_s$. For the $\Lambda_{\pm 1}$ part, by \cite[Lemma 4.1, Lemma 4.8 and Corollary 4.14.]{P3}, we know that
\begin{itemize}
\item If $\pi_{\Lambda}$ is a unipotent representation of a symplectic group or a special odd orthogonal group, then $\pi_{\Lambda}$ is unique determined by $\pi_{\Lambda}^\#$;
\item  If $\pi_{\Lambda}$ is a unipotent representation of a even orthogonal group, then $\pi_{\Lambda}^\#=\pi_{\Lambda'}^\#$ if and only if $\Lambda'\in\{\Lambda,\Lambda^t\}$.
    \item  If $\pi_{\Lambda,\iota}$ is a unipotent representation of a odd orthogonal group, then $\pi_{\Lambda,\iota}^\#=\pi_{\Lambda',\iota'}^\#$ if and only if $\Lambda'=\Lambda$.
\end{itemize}
In summary, we describe the ambiguity of Lusztig correspondences as follows.
\begin{proposition}\label{l3}
Let $\CL_s$ and $\CL_s^\star$ be two different choice of Lusztig correspondences from $\CE(G,s)$ to $\CE(C_G(s),1)$. Let $\pi, \pi'\in \CE(G,s)$. Assume that $\pi'\in[\pi]$, and
$\pi=\pi^{\CL_s}_{\rho,\Lambda_1,\Lambda_{-1},\iota}=\pi^{\CL^\star_s}_{\rho^\star,\Lambda_1^\star,\Lambda_{-1}^\star,\iota^\star}$ and $\pi'=\pi^{\CL_s}_{\rho',\Lambda_1',\Lambda_{-1}',\iota'}$.
\begin{enumerate}
\item If $G$ is a symplectic group, then
\begin{itemize}
\item $\rho'=\rho^\star=\rho$;
\item $\Lambda_1'=\Lambda^\star=\Lambda_1$;
\item $\Lambda_{-1}',\Lambda_{-1}^\star\in\{\Lambda_{-1}, \Lambda_{-1}^t\}$.
\end{itemize}
\item If $G$ is a even orthogonal group, then
\begin{itemize}
\item $\rho'=\rho^\star=\rho$;
\item $\Lambda_1',\Lambda_1^\star\in\{\Lambda_{1}, \Lambda_{1}^t\}$;
\item $\Lambda_{-1}',\Lambda_{-1}^\star\in\{\Lambda_{-1}, \Lambda_{-1}^t\}$.
\end{itemize}
\item If $G$ is a odd orthogonal group, then
\begin{itemize}
\item $\rho'=\rho$;
\item $\Lambda_1'=\Lambda_1^\star=\Lambda_1$;
\item $\Lambda_{-1}'=\Lambda_{-1}^\star=\Lambda_{-1}.$
\end{itemize}
\end{enumerate}
\end{proposition}

Combining Proposition \ref{l3} and {\cite[Theorem 1.1]{P3}}, we have the following corollary.
\begin{corollary}\label{l4}
Let $\CL_s$ be a Lusztig correspondences from $\CE(G,s)$ to $\CE(C_G(s),1)$, and $\pi\in \CE(G,s)$. Assume that $\pi=\prllsgn$.
\begin{enumerate}
\item For a symplectic group $G$, we have $\pi^c=\pi_{\rho, \Lambda_1, \Lambda_{-1}^t}$.
\item For a odd orthogonal group $G$, we have $\sgn\pi=\pi_{\rho, \Lambda_1, \Lambda_{-1},-\iota}$.
\end{enumerate}
\end{corollary}

\subsection{Harish-Chandra series}\label{sec2}
Let $G$ be a reductive group defined over $\Fq$, $F$ be the corresponding Frobenius endomorphism, and let $\E_{\rm{cus}}(G)$ be the set of irreducible cuspidal representations of $G^F$.
A parabolic subgroup $P=LV$ of $G$ is the normalizer in $G$ of a parabolic subgroup $P^\circ$ of the connected component $G^\circ$ of $G$.
The Levi subgroup $L$ of $P$ is the normalizer in $G$ of the a Levi subgroup $L^\circ$ of $P^\circ$. We say that a pair $(L, \delta)$ is cuspidal if $L$ is a Levi subgroup of an $F-$stable parabolic subgroup, and $\delta\in \E_{\rm{cus}}(L)$.

It is well know that for $\pi \in \E(G)$, there is  a unique cuspidal pair $(L, \delta)$ up to $G^F$-conjugacy such that $\langle\pi,\Ind_{L^F}^{G^F}(\delta)\rangle_{G^F}\ne 0$, where we lift $\delta$ to a representation of $P^F$ by making it trivial on $V^F$.
Thus we get a partition of $\E(G)$ into series parametrized by $G^F$-conjugacy classes of cuspidal pairs $(L, \delta)$.
We focus on classical groups, and let $L$ be an $F$-stable standard Levi subgroup of $G_n:=\sp_{2n}$, $\o^\pm_{2n}$ or $\o_{2n+1}$. Then $L$ has a standard form
\[
L=\GGL_{n_1}\times \GGL_{n_2}\times\cdots\times \GGL_{n_r}\times G_m
\]
where $G_m=\sp_{2m}$, $\o^\pm_{2m}$ or $\o_{2m+1}$, and $n_1+\cdots+n_r +m=n$. For a cuspidal pair $(L,\delta)$, one has
\[
\delta=\tau_1\otimes\cdots\otimes\tau_r\otimes\sigma
\]
where $\tau_i\in \E_{\rm{cus}}(\GGL_{n_i})$ and $\sigma\in \E_{\rm{cus}}(G_m)$.
Let
\[
\cal{E}(G_n,\sigma)=\{\pi\in \cal{E}(G_n)|\textrm{ the cuspidal pair $(L,\delta)$ of $\pi$ is of form }\delta=\tau_1\otimes\cdots\otimes\tau_r\otimes\sigma\}.
\]
Then we have  a disjoint union
\[
\E(G_n)=\bigcup_{\sigma}\cal{E}(G_n,\sigma),
\]
where $\sigma$ runs over all irreducible cuspidal representations of $G^{F}_m$, $m=0,1,\cdots,n$.

Let us consider the unipotent representation $\pi_{\Lambda}$ of an even orthogonal group $G$. If $\rm{def}(\Lambda)\ne 0$, then $\pi_{\Lambda}$ and $\sgn\pi_{\Lambda}=\pi_{\Lambda^t}$ will belong to two distinct Harish-Chandra series $\cal{E}(G,\sigma)$ and $\cal{E}(G,\sgn\sigma)$, respectively. On the other hand, if $\rm{def}(\Lambda)=0$, the cuspidal pairs of $\pi_{\Lambda}$ and $\pi_{\Lambda^t}$ will be the same. Therefore, to differentiate between $\pi_{\Lambda}$ and $\pi_{\Lambda^t}$, we require a generalization of cuspidal representation, which is known as basic representation (or basic character according to \cite{P3}).
 \begin{definition}
Let $\pi\in \E(G,s)$ be an irreducible representation of a symplectic or an orthogonal group. Fix a Lusztig correspondence $\cl_s$. Assume that
\[
\mathcal{L}_s(\pi)=\rho\otimes\pi_{\Lambda_1}\otimes\pi_{\Lambda_{-1}}\otimes\iota
\]
where $\epsilon \in \{\pm\}$ only occurs for odd orthogonal group. We say $\pi$ is a basic representation if the following conditions are satisfied.
\begin{itemize}
\item $\rho$ is cuspidal and $G^*_{[\ne 1]}(s)$ is a product of unitary groups;
\item $\pi_{\Lambda_{\pm1}}$ is either an irreducible unipotent cuspidal representation or ${\bf 1}_{\o^+_2}$, $\rm{sgn}_{\o^+_2}$.
\end{itemize}
 Denote by $\E_{\rm{bas}}(G)$  the set of irreducible basic representations of $G^F$.
We generalize the notations $\cal{E}(G_n,\sigma)$ and $(L,\delta)$ in the same way and refer to $(L,\delta)$ as the basic pair of $\pi$.
 \end{definition}

\begin{lemma}\label{basic}
Let $\pi_{\Lambda}$ be a unipotent representation of a even orthogonal group $G=\o_{2n}^+\fq$. Assume that $\rm{def}(\Lambda)=0$ and $\Lambda\ne \Lambda^{t}$. Then $\pi_{\Lambda}$ either lie in $\E(G,{\bf 1}_{\o^+_2})$ or $\E(G,\rm{sgn}_{\o^+_2})$.
\end{lemma}
\begin{proof}
It immediately follows from \cite[Lemma 3.18]{P3}.
\end{proof}
\subsection{Union of Lusztig correspondences}\label{sec3.2}
To construct the canonical Lusztig correspondence, we must ensure compatibility between Lusztig series of all groups of the same type, not just consider its behavior in a single Lusztig series.
Consider  $G_n=\sp_{2n}$, $\o_{2n+1}$ and $\o^\pm_{2n}$.
Let
\[
\cal{L}: \bigcup_{n}\E(G_n)\to\bigcup_{n,(s)}\cal{E}((G_n)^*_{[\ne \pm 1]}(s)\times (G_n)^*_{[1]}(s)\times (G_n)^*_{[-1]}(s),1)
\times \begin{cases}
 \{\pm\},&\textrm{if }G_n=\o_{2n+1};\\
  1,&\textrm{otherwise},
\end{cases}
\]
such that for $\pi\in\cal{E}(G,s)$, the map $\cal{L}$ agrees with a Lusztig correspondence $\CL_s$.

It is natural to require that Lusztig correspondence $\cal{L}$ is compatible with parabolic induction. To be more precise, this compatibility means that the Lusztig correspondence $\cal{L}$ satisfies the following conditions. Let $\prllsgn^{\CL_{}}\in \CE(G_n,\prllps^{\CL_{}})$, where $\prllps^{\CL_{}}$ is a basic representation of $G_m^{F}$. Then we have
\begin{enumerate}
\item[(i)] $\pi_{\Lambda_{1}}$ appears in the Harish-Chandra series of $\pi_{\Lambda_1'}$;
\item[(ii)] $\pi_{\Lambda_{-1}}$ appears in the Harish-Chandra series of $\pi_{\Lambda_{-1}'}$;
\item[(iii)] Assume that $(L,\delta)$ is the basic pair of $\prllsgn^{\CL_{}}$. Then $\iota=\delta(-I)$.
\end{enumerate}
It is worth pointing out that if $\rm{def}(\Lambda_{\pm 1})\ne 0$, the conditions (i) and (ii) are equal to $\rm{def}(\Lambda_{\pm1})=\rm{def}(\Lambda'_{\pm1})$.
The compatibility of unipotent representations is established by Proposition \ref{unp}. Proposition \ref{l3} tells us that the Lusztig correspondence can be characterized by $\iota$ and $\Lambda'_{\pm1}$. Therefore, to create the canonical Lusztig correspondence with the aforementioned compatibility, we only need to construct it using basic representations.

We will now select a specific Lusztig correspondence, denoted by $\cal{L}_{}$, for each $\{G_n\}$. We will use $\prllsgn$ (or $\prll$) to represent $\prllsgn^{\cal{L}_{}}$ (or $\prll^{\CL_{{}}}$), respectively. If either $\rho$ or $\Lambda_{\pm}$ is trivial, we will ignore them in the subscript. For instance, if $\Lambda_{\pm1}$ is trivial, we will denote $\prll$ as $\pi_\rho$. This notation coincides with $\pi_\Lambda$ or $\pi_{\Lambda,\iota}$, which were defined in Section \ref{sec3.4}, for unipotent representations.
It is important to note that $\cal{L}_{}$ is not the canonical choice of Lusztig correspondence that we aim to construct. We are using $\CL_{}$ only for the sake of convenience in our statements and proofs.
\begin{proposition}\label{sp-1}
Let $\prll\in \E(\sp_{2n})$. Then $\prll(-I)=\iota_\rho\pd(\Lambda_{-1})\ee^{\rm{rank}(\Lambda_{-1})}$.
\end{proposition}
\begin{proof}
The quadratic unipotent case is proved in \cite[Proposition 4.11]{Wal1}. The proof of the general case is similar.
\end{proof}

\section{Theta correspondence for finite symplectic groups and
finite orthogonal groups} \label{sectheta}
Assume that $\Fq$ a finite field of odd characteristic.
In this section we review the Theta correspondence of irreducible representations for finite symplectic groups and
finite orthogonal groups. We first recall the Theta correspondence for even orthogonal-symplectic
dual pairs. Then we deduce the Theta correspondence odd orthogonal-symplectic dual pairs from even orthogonal-symplectic
dual pairs.

\subsection{Isometry groups}\label{5.1}
Let $V_\Fn$ be an $\Fn$-dimensional $\Fq$-vector space endowed with a $\epsilon$-symmetric form $\langle,\rangle_{V_{\Fn}}$, i.e. $\langle v,w\rangle_{V_{\Fn}}=\epsilon \langle w,v\rangle_{V_{\Fn}}$ for any $v, w\in V_{\Fn}$.
The isometry group is a symlectic group or a orthogonal group.

For a symmetric space $V_\Fn$, define its discriminant by
\begin{equation}\label{eq:disc}
\disc(V_\Fn)=(-1)^{\frac{\Fn(\Fn-1)}{2}}\det(V_\Fn)\in \Fq^{\times }/\Fq^{\times 2}.
\end{equation}
Write $\o^\epsilon_{\Fn}$ to be the isometry group of $V_\Fn$ with $\disc(V_\Fn)=\epsilon$. For even $\Fn$, the even orthogonal group $\o^\epsilon_{\Fn}$  is split if and only if $\epsilon=+$. For odd $\Fn$, we have $\o^+_{\Fn}$ and $ \o^-_{\Fn}$ are the same as abstract groups; however they act on two symmetric spaces with different discriminants.

 For a symmetric space $V_\Fn$, the spinor norm $\chi_{\o^\epsilon_\Fn}$ of $\o^\epsilon_\Fn\fq$ is a homomorphism from $\o^\epsilon_\Fn\fq$ to $\Fq^\times/(\Fq^\times)^2$ defined as follows. Let $v\in V_{\Fn}$ be such that $\langle v, v\rangle_{V_{\Fn}} \ne 0$ and $s_v$ be the reflection with respect to the hyperplane orthogonal to $v$, then $\chi_{\o^\epsilon_\Fn}(s_v) = \langle v, v\rangle_{V_{\Fn}} \textrm { mod }(\Fq^\times)^2$. The value of $\chi_{\o^\epsilon_\Fn}(-I_{\o^\epsilon_\Fn})$ has been calculated in \cite[p,10]{Wal}, where $I_{\o^\epsilon_\Fn}$ is the identity of ${\o^\epsilon_\Fn}\fq$. In particular, we have
 \begin{equation}\label{spinor}
\chi_{\o^+_{\Fn}}(-I_{\o^+_{\Fn}})=-\chi_{\o^-_{\Fn}}(-I_{\o^-_{\Fn}}).
 \end{equation}

\subsection{Weil representation}

We briefly recall the Theta correspondence over finite fields.

Let
$V_\Fn$ and $V'_{{\Fn}'}$ be linear spaces over $\Fq$ equipped with the $\epsilon$-symmetric form $\langle, \rangle_{V_\Fn}$ and $(-\epsilon)$-symmetric form $\langle, \rangle_{V'_{\Fn'}}$, respectively,
where $\epsilon\in \{\pm\}$. Let  $G$ and $G'$  be the isometry groups of $V_\Fn$ and $V'_{\Fn'}$,
respectively. Let $W := V_\Fn\otimes V'_{\Fn'}$ be a symplectic equipped with the symplectic form $\langle, \rangle_W:=\langle, \rangle_{V_\Fn}\otimes \langle, \rangle_{V'_{\Fn'}}$, and $\sp(W)$ be the isometry group of $\langle \cdot, \cdot\rangle_W$.
Then $(G, G')$ is a reductive dual pair inside $\sp(W)$.

Let $\psi$ be a nontrivial additive character of $\Fq$.
Let $\omega_{N, \psi}$ be the Weil representation (cf. \cite{Ger, H}) of the finite symplectic group $\sp(W)^F = \sp_{2N}\fq$ (with $\dim W=2N$), which affords the unique irreducible representation of
the Heisenberg group of $W$ with central character $\psi$.
Write
$\omega_{G\times G'}^\psi$ for the restriction of $\omega_{N, \psi}$ to $G^F \times G'^F$. It is more convenience for us to use Ma-Qiu-Zou's modified Weil representation $\omega_{G,G'}^\psi$ (\cite[Section 1.4]{MQZ}):
\begin{equation}\label{mqz}
\omega_{G,G'}^\psi:=\left\{
\begin{array}{ll}
\left({\bf 1}_G\otimes(\xi\circ\rm{det}_{V'_{\Fn'}})^{\frac{1}{2}\rm{dim}V_{\Fn}}\right)\otimes\omega_{G\times G'}^\psi& \textrm{ if $G'$ is a orthogonal group}; \\
\left((\xi\circ\rm{det}_{V_{\Fn}})^{\frac{1}{2}\rm{dim}V'_{\Fn'}}\otimes {\bf 1}_{G'}\right)\otimes\omega_{G\times G'}^\psi& \textrm{ if $G$ is a orthogonal group},
\end{array}
\right.
\end{equation}
where ${\bf 1}_G$ is the trivial representation of $G^F$, and $\xi$ is the unique quadratic character of $\Fq^\times$:
\[
\xi:\Fq^\times \to \{\pm 1\}, \ a\to a^{\frac{q-1}{2}}.
\]
We decomposes $\omega_{G,G'}^\psi$ into a direct sum
\[
\omega_{G, G'}^\psi=\bigoplus_{\pi,\pi'} m^\psi_{\pi,\pi '} \,\pi\otimes\pi '
\]
where $\pi$ and $\pi '$ run over $\CE(G)$ and $\CE(G')$ respectively, and $m^\psi_{\pi,\pi'}$ are nonnegative integers. We can reassemble this decomposition as
\[
\omega_{G, G'}^\psi=\bigoplus_{\pi} \pi\otimes\Theta_{G,G'}^\psi(\pi )
\]
 where $\Theta_{G, G'}^\psi(\pi ) := \bigoplus_{\pi'} m^\psi_{\pi,\pi '}\pi '$ is a (not necessarily irreducible) representation of $G'^F$, called the (big) theta lift of $\pi$ from $G^F$ to $G'^F$. We will write $\pi'\subset \Theta^\psi_{G, G'}(\pi)$ if $\pi\otimes\pi'$ occurs in $\omega^\psi_{G,G'}$, i.e. $m^\psi_{\pi, \pi'}\neq 0$.  We remark that even if $\Theta^\psi_{G,G'}(\pi)=\pi'$ is irreducible, in general one only has
 \[
 \pi\subset \Theta^\psi_{G', G}(\pi'),
 \]
 where the equality does not necessarily hold.

Let $\psi$ and $\psi'$ be two nontrivial additive characters of $\Fq$ that belong to different square classes.
For the dual pair $(G, G') = (\sp_{2n}, \o^{\epsilon}_{2n'+1})$, we denote $\omega_{G,G'}^\psi$ by $\omega^{\psi,\epsilon}_{n,n'}$.
By \cite[Proposition 11]{Sze}, we have  $\omega^{\psi,\epsilon}_{n,n'}\cong \omega^{\psi',-\epsilon}_{n,n'}$.
Similar notation applies for $ (\sp_{2n}, \o^{\epsilon}_{2n'})$. In this case, by the proof of Lemma 5.3 in \cite{P3}, we know that
 \[
 m^\psi_{\pi,\pi '}= m^\psi_{\pi^c,(\pi ')^c}= m^{\psi'}_{\pi,(\pi ')^c}= m^{\psi'}_{\pi^c,\pi '},
 \] where $\pi\in \CE(G)$, $\pi'\in \CE(G')$. Although we consider Ma-Qiu-Zou's modified Weil representation instead of the original one, Pan's proof still holds.

From now on, we fix a nontrivial additive character $\psi$ of $\Fq$, and remove the superscript $\psi$ in $\omega^{\psi,\epsilon}_{n,n'}$. In particular, by considering the dual pair $(\sp_{2N}, \o^+_1)$, we have $\omega_{N,\psi}\cong\omega^{+}_{N,1}$, and we simply denote it by $\omega_{N}^+$.

  We write ${\bf G}={\bf Sp}$, ${\bf O}^\pm_{\rm{even}}$ and ${\bf O}^\pm_\rm{odd}$ for the Witt tower $\{G_{n}\}_{n\ge0}=\{\sp_{2n}\}_{n\ge0}$, $\{\o^\pm_{2n}\}_{n\ge0}$ and $\{\o^\pm_{2n+1}\}_{n\ge0}$.
Abbreviate the theta lifting from $G_n$ to $G'_{n'}$ as $\Theta^\pm_{n,n'}$ or $\Theta$ when the dual pairs' context is clear.

\subsection{A explicit description of the theta correspondence}
In their paper \cite{AMR}, Aubert, Michel, and Rouquier proposed a conjecture that explicitly describes the theta correspondence. Pan proved this conjecture in subsequent papers \cite{P1, P2}. However, they studied G\'erardin's Weil representation, whereas we consider Ma-Qiu-Zou's Weil representation in this paper because it is compatible with parabolic induction (cf. \cite[1.13]{MQZ}). This means that if $\sigma'$ appears in the theta lifting of $\sigma$ with $\sigma'\in \E_{\rm{cus}}(G_{n'})'$ and $\sigma'\in \E_{cus}(G_n)$ for a dual pair $(G_n, G_{n'}')$ with $G_n\in {\bf G}$ and $G_{n'}'\in {\bf G'}$, then the irreducible components of the theta lifts of elements in the Harish-Chandra series $\E(G_m,\sigma)$  lie in the Harish-Chandra series $\E_{\rm{cus}}(G'_{m'},\sigma')$. It is worth mentioning that Pan's study was based on compatibility with parabolic induction, as described above. Therefore, the results in \cite{P1, P2} actually apply to Ma-Qiu-Zou's Weil representation, not G\'erardin's. Moreover, Theorem 1.4 in \cite{MQZ} shows that Aubert, Michel, and Rouquier's conjecture holds for  Ma-Qiu-Zou's Weil representation. As a consequence, all results in \cite{Wang1} are for Ma-Qiu-Zou's Weil representation,  not G\'erardin's.

The next result shows that the theta lifting and the parabolic induction are compatible.

\begin{proposition}\label{w2}
Let $G_n$ and $G_{n+\ell}$ be two classical groups in the same Witt tower of symplectic groups or orthogonal groups, $\ell\geq 0$. Let $\pi\in\cal{E}(G_n,s)$ with $s\in (G^*_n)^\circ(\Fq)$, and  $\pi':=\Theta_{n,n'}(\pi)$. Let $\chi_{\GGL_{\ell}}$ be the unique linear character of $\GGL_\ell\fq$ of order $2$. Let $\tau\in\cal{E}(\GGL_\ell,s_0)$ with $s_0\in \GGL_\ell\fq$. Let $ \rm{Ind}^{G_{n+\ell}\fq}_{\GGL_\ell\fq\times G_{n}\fq}(\tau \otimes \pi)=\bigoplus_{i}\rho_i$ with $\rho_i$ irreducible.  Assume that $s$ and $s_0$ have no common eigenvalues, and $s_0$ has no eigenvalues $\pm1$. Then we have
\[
\bigoplus_{i}\Theta_{n+\ell, n'+\ell}(\rho_i)=\rm{Ind}^{G^{\prime }_{n'+\ell}\fq}_{\GGL_\ell\fq\times G^{\prime }_{n'}\fq}((\chi\otimes\tau)\otimes \pi').
\]
where
\[
\chi=\left\{
\begin{array}{ll}
\chi_{\GGL_{\ell}}, &  \textrm{if $(G_{n+\ell},G'_{n'+\ell})$ contains an odd orthogonal group,}\\
1, & \textrm{otherwise.}
\end{array}\right.
\]
Hence, by abuse of notations, we write
\[
\Theta_{n+\ell,n'+\ell}\left(\rm{Ind}^{G_{n+\ell}\fq}_{\GGL_\ell\fq\times G_{n}\fq}(\tau \otimes \pi)\right)=\rm{Ind}^{G^{\prime }_{n'+\ell}\fq}_{\GGL\fq_\ell\times G^{\prime }_{n'}\fq}((\chi\otimes\tau)\otimes \pi').
\]
\end{proposition}

The theta lifting of unipotent representations is as follows.

 \begin{theorem}[Pan]\label{p}

 Let $\pl\in\cal{E}(\sp_{2n},1)$ and $\pll\in \cal{E}(\o^\epsilon_{2n'},1)$. Then $\pl\otimes\pll$ occurs in $\omega_{n,n'}^\epsilon$ if and only if $(\Lambda,\Lambda')\in \mathcal{B}^\epsilon_{n,n'}$, where
 \[
\begin{aligned}
&\mathcal{B}^+_{n,n'}:=\left\{(\Lambda,\Lambda')|{}^t(\Upsilon(\Lambda')^-)\preccurlyeq{}^t(\Upsilon(\Lambda)^+),
{}^t(\Upsilon(\Lambda)^-)\preccurlyeq{}^t(\Upsilon(\Lambda')^+),\rm{def}(\Lambda')=-\rm{def}(\Lambda)+1\right\}; \\
&\mathcal{B}^-_{n,n'}:=\left\{(\Lambda,\Lambda')|{}^t(\Upsilon(\Lambda')^+)\preccurlyeq {}^t(\Upsilon(\Lambda)^-),
{}^t(\Upsilon(\Lambda)^+)\preccurlyeq{}^t(\Upsilon(\Lambda')^-),\rm{def}(\Lambda')=-\rm{def}(\Lambda)-1\right\}
\end{aligned}
\]
are two subsets of $\cal{S}_n\times\cal{S}_{n'}^{+}$ and $\cal{S}_n\times\cal{S}_{n'}^{-}$, respectively.
 \end{theorem}

The general case can be deduce from the unipotent case for dual pair $(\sp_{2n},\o^\epsilon_{2n'})$ by Theorem \ref{p1} and Theorem \ref{p2}.

\begin{theorem}[Pan]\label{p1}
 Let $(G, G') = (\sp_{2n},\o^\epsilon_{2n'})$, and let $\pi\in\cal{E}(G,s)$ and $\pi'\in \cal{E}(G',s')$ for
some semisimple elements $s \in G^*\fq$ and $s'\in (G^{\prime *})^0\fq$. Write $\cal{L}_{s}(\pi)=\rho\otimes\pi_{\Lambda_1}\otimes\pi_{\Lambda_{-1}}$ and
and $\cal{L}_{s'}(\pi')=\rho'\otimes\pi_{\Lambda'_1}\otimes\pi_{\Lambda'_{-1}}$.
Then there exists an irreducible representation $\pi''\in[\pi']$
such that $\pi \otimes \pi''$ occurs in $\omega_{n,n'}^\epsilon$ if and only if the following conditions hold:
\begin{itemize}

\item For any $a\ne \pm 1$ in $\overline{\mathbb{F}}_q$, $G^*_{[a]}(s)\cong G^{\prime *}_{[a]}(s')$,  and $\pi[a]\cong \pi'[a]$, where $\dg=\prod G^*_{[a]}(s)$, $G^{\prime * }_{[\ne\pm 1]}(s')=\prod G^{\prime *}_{[a]}(s')$,
$\rho=\prod\pi[a]$, and $\rho'=\prod\pi'[a]$;

\item either $\pi^{\sp}_{\Lambda_1}\otimes\pi_{\Lambda_1}$ or $\pi^{\sp}_{\Lambda_1}\otimes\pi_{\Lambda_1^t}$ occurs in $\omega_{\ddg, G^{\prime * }_{[ 1]}(s')}$;

\item  $\pi_{\Lambda_{-1}}\in[\pi_{\Lambda'_{-1}}]$,
\end{itemize}
where $\pi^{\sp}_{\Lambda_1}:=\CL_1(\pi_{\Lambda_1})$ is the unipotent representation of a symplectic group correspondence to the symbol $\Lambda_1$.
That is, the following diagram:
\[
\setlength{\unitlength}{0.8cm}
\begin{picture}(20,5)
\thicklines
\put(5.8,4){$\pi$}
\put(5.2,2.6){$\cal{L}'_s$}
\put(13.8,2.6){$\cal{L}_{s'}'$}
\put(9.4,4.2){$\Theta$}
\put(4.3,1){$\rho\otimes\pi_{\Lambda_1}\otimes\pi_{\Lambda_{-1}}$}
\put(13.4,4){$\pi''$}
\put(11.8,1){$\rho'\otimes\pi_{\Lambda_1'}\otimes\pi_{\Lambda_{-1}'}$}
\put(8.2,1.3){$\rm{id}\otimes\Theta\circ \CL_1\otimes\rm{id}$}
\put(6,3.6){\vector(0,-1){2.1}}
\put(13.5,3.6){\vector(0,-1){2.1}}
\put(8.2,1.1){\vector(2,0){3}}
\put(8.2,4){\vector(2,0){3}}
\end{picture}
\]
commutes up to a twist of the sgn character.
\end{theorem}

\begin{theorem}[Pan]\label{p2}
Let $(G, G') = (\sp_{2n},\o^\epsilon_{2n'+1})$, and let $\pi \in  \cal{E}(G,s)$ and $\pi'\in \cal{E}(G',s')$ for some semisimple elements $s \in G^*\fq$ and $s'\in (G^{\prime *})^0\fq$. Write $\cal{L}_{s}(\pi)=\rho\otimes\pi_{\Lambda_1}\otimes\pi_{\Lambda_{-1}}$ and
and $\cal{L}_{s'}(\pi')=\rho'\otimes\pi_{\Lambda'_1}\otimes\pi_{\Lambda'_{-1}}\otimes\epsilon'$. Then there exists an irreducible representation $\pi''\in[\pi']$ such that $\pi \otimes \pi''$ occurs in $\omega_{n,n'}^\epsilon$ if and only if the following conditions hold:
\begin{itemize}

\item For any $a\ne \pm 1$ in $\overline{\mathbb{F}}_q$, $G^*_{[a]}(s)\cong G^{\prime *}_{[-a]}(s')$,  and $\pi[a]\cong \pi'[-a]$, where $\dg=\prod G^*_{[a]}(s)$, $G^{\prime * }_{[\ne\pm 1]}(s')=\prod G^{\prime *}_{[a]}(s')$,
$\rho=\prod\pi[a]$, and $\rho'=\prod\pi'[a]$;

\item either $\pi_{\Lambda_{-1}}\otimes\pi_{\Lambda'_{1}}$ or $\pi_{\Lambda^t_{-1}}\otimes\pi_{\Lambda'_{1}}$  occurs in $\omega_{\dddg,G^{\prime * }_{[ 1]}(s')}$;

\item  $\pi^{\sp}_{\Lambda_1}\cong\pi_{\Lambda'_{-1}}$,
\end{itemize}
where $\pi^{\sp}_{\Lambda_1}:=\CL_1(\pi_{\Lambda_1})$ is the unipotent representation of a symplectic group correspondence to the symbol $\Lambda_1$.
That is, the following diagram:
\[
\setlength{\unitlength}{0.8cm}
\begin{picture}(20,5)
\thicklines
\put(5.8,4){$\pi_i$}
\put(5.2,2.6){$\cal{L}'_s$}
\put(13.8,2.6){$f\circ\cal{L}_{s'}'$}
\put(9.4,4.2){$\Theta$}
\put(4.3,1){$\rho\otimes\pi_{\Lambda_1}\otimes\pi_{\Lambda_{-1}}$}
\put(13.4,4){$\pi'$}
\put(11.8,1){$\rho'\otimes\pi_{\Lambda'_{-1}}\otimes\pi_{\Lambda'_{1}}$}
\put(8.5,1.3){$\rm{id}\otimes\CL_1\otimes\Theta$}
\put(6,3.6){\vector(0,-1){2.1}}
\put(13.5,3.6){\vector(0,-1){2.1}}
\put(8.2,1.1){\vector(2,0){3}}
\put(8.2,4){\vector(2,0){3}}
\end{picture}
\]
commutes up to a twist of the sgn character where $f(\rho'\otimes\pi_{\Lambda'_1}\otimes\pi_{\Lambda'_{-1}}\otimes\epsilon')=\rho'\otimes\pi_{\Lambda'_{-1}}\otimes\pi_{\Lambda'_{1}}$.

\end{theorem}

Therefore, the description of general Howe correspondence for dual pair of a symplectic group and an orthogonal group is now completely characterized up to sgn.
\begin{proposition}\label{-1}
Let $\pi\in \CE(G)$, and $\pi'\in \CE(G')$, and let $I_G$ and $I_{G'}$ be the identities of $G$ and $G'$, respectively. Assume that $\pi\otimes\pi'$ occurs in the Weil representation $\omega_{G, G'}^\psi$.
   \begin{enumerate}
 \item [(i)]
If $(G,G')=(\sp_{2n},\o^{\epsilon}_{2n'})$, then $\pi(-I_G)=\pi'(-I_{G'})$.
 \item[(ii)]
If $(G,G')=(\sp_{2n},\o^{\epsilon}_{2n'})$, then $\pi(-I_G)=\varepsilon(-1)^{n}\cdot\pi'(-I_{G'})$.
\end{enumerate}

\end{proposition}
\begin{proof}
It  immediately follows from the fact that the embedding $G\times G'\to \sp_{2N}$ sends both $-I_G\times I_{G'}$ and $I_G\times -I_{G'}$ to $-I_{\sp_{2N}}$. Note that we have a character $(\xi\circ \rm{det})^n$ in Ma-Qiu-Zou's Weil representation \eqref{mqz}. So we need to add $\varepsilon(-1)^{n}$ when $(G,G')=(\sp_{2n},\o^{\epsilon}_{2n'+1})$.
\end{proof}

\subsection{First occurrence index of unipotent basic representations}\label{sec7}

For an irreducible representation $\pi$ of $G_{n}$, the smallest integer $n(\pi)$ such that $\pi$ occurs in $\omega_{G_n,G'_{n(\pi)}}$ is called the {\it first occurrence index} of $\pi$ with respect to the Witt tower $\left\{G'_{n}\right\}_{\ge 0}$, and called $\Theta_{G_n,G'_{n(\pi)}}$ the first theta lifting of $\pi$.

\begin{definition}
\begin{enumerate}
\item  Assume that $G_n=\sp_{2n}$. Let $A\in \{\rm{even},\rm{odd}\}$, and $n^\epsilon_{A}(\pi)$ be the first occurrence index of $\pi$ with respect to the Witt tower ${\bf O}^\epsilon_{A}$. We set
    \[
         {\bf n}^\epsilon_{A}(\pi) :=n-n^\epsilon_{A}(\pi).
     \]
    We abbreviate  $n^\epsilon_{\rm{even}}(\pi)$,  $n^\epsilon_{\rm{odd}}(\pi)$, $\nepe(\pi)$ and $\nepo(\pi)$ by $n^\epsilon(\pi)$ and $\nep(\pi)$ when no confusion arise.

    \item  Assume that $G_n=\o^\epsilon_{2n}$ or $\o^\epsilon_{2n+1}$. Let $n^\epsilon(\pi)$ be the first occurrence index of $\pi$ with respect to the Witt tower ${\bf Sp}$. We set
        \[
         {\bf n}^\epsilon(\pi):=n-n^\epsilon(\pi).
         \]

\item We say that the theta correspondence of $\pi$ goes down (resp. up) with respect to the Witt tower $\left\{G'_{n}\right\}$ if $\nep(\pi)\ge 0$ (resp. $\nep(\pi)< 0$).
 \end{enumerate}
 \end{definition}

For a dual pair $(G_n,G'_{n'})$ and an irreducible cuspidal representation $\pi$ of $G_n$, by \cite[Theorem 2.2]{AM} , there is a cuspidal representation appearing in the theta lifting if and only if $n'$ is the first occurrence index of $\pi$. Moreover, the theta lifting of $\pi$ is an irreducible cuspidal representation.

\begin{theorem}[\cite{AM}, Theorem 5.2] \label{even}
The theta correspondence for dual pairs $(\sp_{2n}, \o^\epsilon_{2n'})$ takes unipotent cuspidal representations to unipotent cuspidal representations as follows :
\begin{enumerate}
 \item $(\sp_{2k(k+ 1)}, \o^\epsilon _{2k^2})$, $\epsilon = \rm{sgn}(-1)^k$,
\[
\Theta^{\epsilon}_{k(k+1), k^2}: \left\{
\begin{aligned}
&\pl,\ \Lambda=\begin{pmatrix}
2k,\cdots,1,0\\
-
\end{pmatrix}
\to
\pll,\ \Lambda'=\begin{pmatrix}
-\\
2k-1,\cdots,1,0\\
\end{pmatrix}
\textrm{ if }k \textrm{ is even;}\\
&\pl,\ \Lambda=\begin{pmatrix}
-\\
2k,\cdots,1,0
\end{pmatrix}
\to
\pll,\ \Lambda'=\begin{pmatrix}
2k-1,\cdots,1,0\\
-
\end{pmatrix}
\textrm{ if }k \textrm{ is odd}.
\end{aligned}\right.
 \]

\item  $(\sp_{2k(k+ 1)}, \o^\epsilon _{2(k+1)^2})$, $\epsilon = \rm{sgn}(-1)^{k+1}$,
\[
\Theta^{\epsilon}_{k(k+1), (k+1)^2}: \left\{
\begin{aligned}
&\pl,\ \Lambda=\begin{pmatrix}
2k,\cdots,1,0\\
-
\end{pmatrix}
\to
\pll,\ \Lambda'=\begin{pmatrix}
-\\
2k+1,\cdots,1,0\\
\end{pmatrix}
\textrm{ if }k \textrm{ is even;}\\
&\pl,\ \Lambda=\begin{pmatrix}
-\\
2k,\cdots,1,0
\end{pmatrix}
\to
\pll,\ \Lambda'=\begin{pmatrix}
2k+1,\cdots,1,0\\
-
\end{pmatrix}
\textrm{ if }k \textrm{ is odd}.
\end{aligned}\right.
 \]
 \end{enumerate}
\end{theorem}

\begin{proposition}({\cite[Example 5.5 (1)]{P3}})
The first occurrence index of ${\bf 1}_{\o^+_2}$ and $\rm{sgn}_{\o^+_2}$ is 0 and 1, respectively.
\end{proposition}

\begin{corollary}
\begin{enumerate}
\item Let $\pi_\Lambda$ be the unipotent cuspidal representation of $\sp_{2k(k+1)}\fq$. Then
\[
\nep_{\rm{even}}(\pi)=\left\{
\begin{array}{ll}
-k-1& \textrm{if $\epsilon\cdot\rm{def}(\Lambda)<0$;}\\
k&\textrm{if $\epsilon\cdot\rm{def}(\Lambda)>0$}.
\end{array}\right.
\]
    \item Let $\pi_\Lambda$ be the unipotent basic representation of $\o^\epsilon_{2k^2}\fq$ with $\epsilon = \rm{sgn}(-1)^{k+1}$. Then
    \[
\nep(\pi)=\left\{
\begin{array}{ll}
-k& \textrm{if $\epsilon\cdot\rm{def}(\Lambda)<0$;}\\
k&\textrm{if $\epsilon\cdot\rm{def}(\Lambda)\ge 0$}.
\end{array}\right.
\]
\end{enumerate}
\end{corollary}

\subsection{Conservation relation for basic representations}\label{sec4.5}

 In this subsection, we will differentiate between representations in $[\pi]$ using theta correspondence. This will enable us to determine a Lusztig correspondence, as there is ambiguity in the Lusztig correspondence on$[\pi]$.

According to {\cite[Theorem 1.1]{P3}}, the set $[\pi]$ consists of $\pi$, $\sgn\pi$, $\pi^c$ and $\sgn\pi^c$, provided such representations exist. For even orthogonal groups, the first occurrence indexes $n^\epsilon(\pi)$ and $n^\epsilon(\sgn\pi)$ are different due to the conservation relation. To distinguish between $\pi$ and $\pi^c$, we examine their conservation relation and show that although their first indexes are the same, their first theta lifting is different.
For odd orthogonal groups, we have two Witt towers ${\bf O}^\pm_\rm{odd}$, and we can prove that $n^\epsilon(\pi)\ne n^\epsilon(\sgn\pi)$ for $\epsilon=\pm$. This allows us to differentiate between $\pi$ and $\sgn\pi$ using their first occurrence indexes. For symplectic groups, there is no sgn character. We give a conservation relation for $\pi$ and $\pi^c$, and prove that $n^\epsilon_{\rm{odd}}(\pi)\ne n^\epsilon_{\rm{odd}}(\pi^c)$ for $\epsilon=\pm$.

Let  $\prllsgn$  be an irreducible cuspidal representation of $\sp_{2n}\fq$, $\o^\pm_{2n}\fq$ or $\o_{2n+1}\fq$, where $\iota\in \{\pm\}$ only occurs for odd orthogonal groups. Assume that $\Lambda_1$ and $\Lambda_{-1}$ correspond to unipotent basic representations of $\ddg$ and $\dddg$, respectively. Let
\[
k=\left\{
\begin{array}{ll}
\frac{|\rm{def}(\Lambda_1)|-1}{2}&\textrm{ if } \Lambda_1\in\cal{S}_m;\\
\frac{\rm{def}(\Lambda_1)}{2}&\textrm{ if } \Lambda_1\in\cal{S}^\pm_m \textrm{ and }\rm{def}(\Lambda_1)\ne 0;\\
0.5&\textrm{ if } \Lambda_1=\binom{1}{0};\\
-0.5&\textrm{ if } \Lambda_1=\binom{0}{1},\\
\end{array}\right.
\]
and
\[
h=\left\{
\begin{array}{ll}
\frac{|\rm{def}(\Lambda_{-1})|-1}{2}&\textrm{ if } \Lambda_{-1}\in\cal{S}_{m'};\\
\frac{\rm{def}(\Lambda_{-1})}{2}&\textrm{ if } \Lambda_{-1}\in\cal{S}^\pm_{m'}\textrm{ and }\rm{def}(\Lambda_{-1})\ne 0;.\\
0.5&\textrm{ if } \Lambda_{-1}=\binom{1}{0};\\
-0.5&\textrm{ if } \Lambda_{-1}=\binom{0}{1}.\\
\end{array}\right.
\]
For abbreviation, we write $\pi_{\rho,k,h,\iota}$ instead of $\prllsgn$. By Corollary \ref{l4}, we have
$\pkh^c=\pi_{\rho,k,-h}$ if $\pkh\in \CE_{\rm{cus}}(\sp_{2n})$, and $\sgn\pi_{\rho,k,h,\iota}=\pi_{\rho,k,h,-\iota}$ if $\pi_{\rho,k,h,\iota}\in \CE_{\rm{bas}}(\o_{2n+1})$.

The first occurrence indexes of cuspidal representations have been calculated in \cite[Section 6]{Wang1}. It is not difficult to extend these results to basic representations.

\begin{proposition}{\cite[Proposition 6.2]{Wang1}}\label{cus1}
Let $\pkh\in \CE_{\rm{bas}}(\sp_{2n})$. Then the following hold:

(i) Consider the Witt tower ${\bf O}^\pm_\rm{even}$. We have
\begin{itemize}
\item (Conservation relation) $\{n^\pm(\pkh)\}=\{ n^\pm(\pi_{\rho,k,-h})\}=\{ n- k,n+k+1\}$;
\item (Conservation relation) $n^\epsilon(\pkh)=n^\epsilon(\pi_{\rho,k,-h})$ and $\Theta_{n,n^\epsilon(\pkh)}(\pkh)\ne \Theta_{n,n^\epsilon(\pi_{\rho,k,-h})}(\pi_{\rho,k,-h})$;
\item if $\pkh$ goes down with respect to the Witt tower ${\bf O}^\epsilon_\rm{even}$, then $\Theta_{n,n^\epsilon(\pkh)}(\pkh)=\pi_{\rho,k',h'}$ with $|k'|=|k|$ and $|h'|=|h|$.
\end{itemize}

(ii) Consider the Witt tower ${\bf O}^\pm_\rm{odd}$. We have
\begin{itemize}
\item (Conservation relation) $\{n^\pm(\pkh)\}=\{n^\pm(\pi_{\rho,k,-h})\}=\{\lfloor n\pm h \rfloor\}$;
\item  if $\pkh$ goes down with respect to the Witt tower ${\bf O}^\epsilon_\rm{odd}$, then $\Theta_{n,n^\epsilon(\pkh)}(\pkh)=\pi_{\rho,k',h',\iota'}$ with $k'= \lceil |h|-1\rceil$ and $h'= k$. Moreover, we have $\Theta_{n^\epsilon(\pkh),n}(\pi_{\rho,k',h',\iota'})=\pkh$.
\end{itemize}

\end{proposition}

\begin{proposition}{\cite[Proposition 6.3]{Wang1}}\label{cus2}

(i) Let $\pkh\in \CE_{\rm{bas}}(\o^\epsilon_{2n})$.  Then the following hold:
\begin{itemize}
\item (Conservation relation) $\{n^\epsilon(\pkh),n^\epsilon(\sgn\pkh)\}\in\{\lfloor n\pm k\rfloor\}$;
\item (Conservation relation) $n^\epsilon(\pkh)=n^\epsilon(\pkh^c)$ and  $(\Theta_{n,n^\epsilon(\pkh)}(\pkh))^c= \Theta_{n,n^\epsilon(\pkh^c)}(\pi_{\rho,k,h}^c)$;
\item
 If $n^\epsilon(\pkh)= \lfloor n+ |k| \rfloor$, then $\Theta_{n,n^\epsilon(\pkh)}(\pkh)=\pi_{\rho,k',h'}$ with $k'=\lceil |k|\rceil $ and $|h'|=| h|$;
\item
 If $n^\epsilon(\pkh)=\lfloor n- |k|\rfloor $, then $\Theta_{n,n^\epsilon(\pkh)}(\pkh)=\pi_{\rho,k',h'}$ with $k'=\lceil|k|-1\rceil$ and $|h'|=|h|$.
\end{itemize}

(ii) Let $\pi_{\rho,k,h,\iota}\in \CE_{\rm{cus}}(\o_{2n+1})$. Consider the Witt tower ${\bf O}^\epsilon_\rm{odd}$. Then the following hold:
\begin{itemize}
\item  (Conservation relation) $\{n^\epsilon(\pi_{\rho,k,h,\iota}),n^\epsilon(\pi_{\rho,k,h,-\iota})\}=\{n+ k+1, n-k\}$.
\item
 If $n^\epsilon(\pi_{\rho,k,h,\iota})=n+ k+1$, then $\Theta_{n,n^\epsilon(\pi_{\rho,k,h,\iota})}^\epsilon(\pi_{\rho,k,h,\iota})=\pi_{\rho,k',h'}$ with $k'= h$ and $|h'|=|k+1|$;
\item
 If $n^\epsilon(\pi_{\rho,k,h,\iota})=n- k$, then $\Theta_{n,n^\epsilon(\pi_{\rho,k,h,\iota})}^\epsilon(\pi_{\rho,k,h,\iota})=\pi_{\rho,k',h'}$ with $k'=h$ and $|h'|=| k|$.
\end{itemize}

\end{proposition}

Let $\pi$ be an irreducible basic representation of $\o_{2n+1}\fq$. We will demonstrate that the first occurrence index is determined solely by $\pi$, regardless of the Witt tower it is located in.
\begin{theorem}\label{t4}

Let $\pi\in \CE_{\rm{bas}}(\o_{2n+1})$. Then $n^+(\pi)=n^-(\pi)$.
\end{theorem}
\begin{lemma}\label{t4l}
Let $\pi\in \CE_{\rm{bas}}(\sp_{2n})$. Then $\pi(-I)=\pi^c(-I)$, where $I$ is the identity of $\sp_{2n}\fq$.
\end{lemma}
\begin{proof}
We prove Theorem \ref{t4} and Lemma \ref{t4l} simultaneously by induction on $n$.
First, we prove Lemma \ref{t4l} for $\sp_{2n}$. Assume that theta correspondence of $\pi=\pkh\in \CE_{\rm{bas}}(\sp_{2n})$ goes down with respect to the Witt tower ${\bf{{O}}}_{\rm{odd}}^{\epsilon}$, the first occurrence index $n^{\epsilon}(\pi)=n'$
and $\Theta^\epsilon_{n,n'}(\pi)=\pi'$. By Proposition \ref{cus1} (ii), we know that $\pi'\in \CE_{\rm{bas}}(\o_{2n'+1})$ and $n^\epsilon(\pi')=n$.

If $n=n'$, according to Proposition \ref{cus1}, we have $h=0$. Therefore, according to Corollary  \ref{l4}, $\pi=\pi^c$. Now, let us assume that $n'<n$.
Utilizing the induction assumption on  $\pi'$, we get $n^{-\epsilon}(\pi')=n^{\epsilon}(\pi')=n$. However, for the Witt tower ${\bf{{O}}}_{\rm{odd}}^{-\epsilon}$, the theta lifting of $\pi'$ to $\sp_{2n'}\fq$ is no longer $\pi$, because by Proposition \ref{cus1} (ii), $\pi$ goes up with respect to the Witt tower ${\bf{{O}}}_{\rm{odd}}^{-\epsilon}$. Thus, from Proposition \ref{cus2} (ii), it follows that the above mentioned theta lifting of $\pi'$ has to be $\pi^c=\pi_{\rho,k,-h}$. Hence, by Proposition \ref{-1},
\[
\pi^c(-I)\cdot\pi'(-I_{\o_{2n'+1}})=\pi(-I)\cdot\pi'(-I_{\o_{2n'+1}}),
\]
where $I_{\o_{2n'+1}}$ is the identity of $\o_{2n'+1}$. This completes the proof of Lemma \ref{t4l}.

We turn to prove Theorem \ref{t4} for $\o_{2n+1}$. Let $\pi^\star\in\CE_{\rm{bas}}(\o_{2n+1})$.  Our objective is to check whether $\pi^\star$ goes up or down simultaneously with respect to two Witt towers ${\bf{{O}}}_{\rm{odd}}^{\pm}$.
 Assume that $\pi^\star$ goes down with respect to the Witt tower ${\bf{{O}}}_{\rm{odd}}^{\epsilon}$
 and $\Theta(\pi^\star)=\pi^{\star\prime}\in\CE_{\rm{cus}}(\sp_{2n^{\star\prime}})$ be the first theta lifting. For the Witten tower ${\bf{{O}}}_{\rm{odd}}^{-\epsilon}$, the representation $\sgn\pi^\star$ can not go down, because if $\sgn\pi^\star$ goes down, then by Proposition \ref{cus2} (ii), its first theta lifting is either $\pi^{\star\prime}$ or $(\pi^{\star\prime})^c$. However, by Lemma \ref{t4l}, we have
\[
(\sgn\pi^\star)(-I_{\o_{2n+1}})\cdot\pi^{\star\prime}(-I)=(\sgn\pi^\star)(-I_{\o_{2n+1}})\cdot(\pi^{\star\prime})^c(-I)\ne\pi^\star(-I_{\o_{2n+1}})\cdot\pi^{\star\prime}(-I)=\varepsilon(-1)^{n^{\star\prime}},
\]
which contradicts Proposition \ref{-1}.
So $\sgn\pi^\star$ goes up and  $\pi^\star$ goes down with respect to the Witt towers ${\bf{O}}_{\rm{odd}}^{-\epsilon}$.

If $\pi^\star$ goes up with respect to the Witt tower ${\bf{{O}}}_{\rm{odd}}^{\epsilon}$, then by Proposition \ref{cus2} (ii), $\sgn\pi^\star$ goes down. Applying the same argument to $\sgn\pi^\star$, we can conclude that $\pi^\star$ goes up with respect to the Witt towers ${\bf{O}}_{\rm{odd}}^{-\epsilon}$.
\end{proof}

\begin{corollary}\label{oddg}
\begin{enumerate}
\item Let $\pi_{\rho,k,h,\iota}\in \CE_{\rm{bas}}(\o_{2n+1})$. Then
\[
n^\pm(\pi_{\rho,k,h,\iota})=
\left\{\begin{array}{ll}
n+(k+1)& \textrm{ if }\iota=\wi_\rho\cdot(-1)^{k+1}  \\
n-k & \textrm{ if }\iota=\wi_\rho\cdot(-1)^{k}.
\end{array}
\right.
\]

\item Let $\pkh\in \CE_{\rm{bas}}(\sp_{2n})$. Then
\[
\pkh(-I)=\pi_{\rho,k,-h}(-I)=\iota_\rho\cdot(-1)^{\lfloor |h|\rfloor}\ee^{\lceil |h|\rceil}.
\]
\end{enumerate}
\end{corollary}
\begin{proof}
We will use induction on $n$ to prove both (1) and (2) at the same time. First, we will prove (i). By Proposition \ref{cus2} (ii), we know that $\{n^\pm(\pi_{\rho,k,h,\iota}),n^\pm(\pi_{\rho,k,h,-\iota})=\{n+(k+1),n-k \}$. WOLG, assume that $n^\pm(\pi_{\rho,k,h,\iota})=n-k$. Then $\Theta^\pm_{n,n-k}(\pi_{\rho,k,h,\iota})=\pi_{\rho,k',h^{\pm \prime}}\in \CE_{\rm{bas}}(\sp_{2n})$ with $k'=h$ and $h^{\pm \prime}\in\{\pm k \}\in\bb{Z}$. By Proposition \ref{-1}, we get
\[
\iota=\pi_{\rho,k,h,\iota}(-I)=\ee^{n-k}\cdot\pi_{\rho,k',h^{\pm \prime}}(-I)=\ee^{|\rho|+k'(k'+1)+(h^{\pm \prime})^2}\cdot\pi_{\rho,k',h^{\pm \prime}}(-I).
\]
Note that we have  the following equations: $\wi_\rho=\iota_\rho\cdot\ee^{|\rho|} $, and $\ee^{k'(k'+1)}=1$ and $\ee^{(h^{\pm \prime})^2}=\ee^{h^{\pm \prime}}$. Hence, (1) follows from our induction assumption on $\sp_{2(n-k)}\fq$.

The proof of (2) is similar. We pick a Witt Tower ${\bf{O}}_{\rm{odd}}^{\epsilon}$ such that $\pkh$ goes down. Then (2) follows from the same see-saw argument and our induction assumption on $\o_{2(n-k)+1}\fq$.

\end{proof}
\section{Gan-Gross-Prasad problem}\label{sec6}
From now until the end of subsection \ref{sec6.2}, we assume that  $\Fq$ is a finite field of odd characteristic.

Let $V_\Fn$ be an $\Fn$-dimensional $\Fq$-vector space endowed with a $\epsilon$-symmetric form $\langle,\rangle_{V_{\Fn}}$. Suppose that $W\subset V_{\Fn}$ is a non-degenerate subspace with dimensional $\Fm$ satisfying:
\begin{itemize}

\item $\epsilon \cdot (-1)^{\rm{dim} W^\bot} = -1$,

\item $W^\bot$ is a split space.
\end{itemize}
According to whether $\Fn-\Fm$ is odd or even, Gan-Gross-Prasad problem is called the Bessel case or Fourier-Jacobi case.

\subsection{Bessel case}\label{sec6.1}

Let $V_\Fn$ be a symmetric space with $\disc(V_\Fn)=\epsilon$.
Consider a decomposition of $V_\Fn$:
\[
V_\Fn=X+V_{\Fn-2\ell}+X^\vee
\]
where $X+X^\vee=V_{\Fn-2\ell}^\perp$ is a polarization. Pick a basis $\{e_1,\ldots, e_\ell\}$ of $X$. Let $P$ be the parabolic subgroup of $\o^\epsilon_\Fn$ stabilizing the flag
\[
\rm{Span}_{\mathbb{F}_{q}}\{e_1\}\subset\rm{Span}_{\mathbb{F}_{q}}\{e_1,e_2\}\subset\cdots\subset \rm{Span}_{\mathbb{F}_{q}}\{e_1,\cdots,e_\ell\}.
\]
The Levi component of $P$ is $M\cong \GGL_1^\ell\times \o^\epsilon_{\Fn-2\ell}$. Its unipotent radical can be written in the form
\[
N_\ell=\left\{n=
\begin{pmatrix}
z & y & x\\
0 & I_{n-2\ell} & y'\\
0 & 0 & z^*
\end{pmatrix}
: z\in U_{\GGL_\ell}
\right\},
\]
where the superscript ${}^*$ denotes the transpose inverse, and $U_{\GGL_\ell}$ is the subgroup of unipotent upper triangular matrices of $\GGL_\ell$.
Define a generic character $\psi_{v_0}$ of $N_\ell$ by
\[
\psi_{ v_0}(n)=\psi\left(\sum^{\ell-1}_{i=1}z_{i,i+1}+\langle y_\ell, v_0\rangle_{V_{\Fn}}\right), \quad n\in N_\ell(\Fq),
\]
where $y_\ell$ is the last row of $y$, and $v_0\in V_{\Fn-2\ell}$ is an anisotropic vector. Let $W$ be the orthogonal complement of $v_0$ in $V_{\Fn-2\ell}$. Assume that $\disc(W)=\epsilon'$.
 The stabilizer of $\psi_{ v_0}$ in $M(\Fq)$ is the orthogonal group $\rm{O}_{\Fn-2\ell-1}^\e\fq$.
Put
\begin{equation}\label{hnu}
H=\rm{O}_{\Fn-2\ell-1}^\e\ltimes N_\ell.
\end{equation}
Let $\pi$ and $\pi'$ be two irreducible cuspidal representations of $\o^\epsilon_{\Fn}\fq$ and $\o^\e_{\Fn-2\ell-1}\fq$ respectively. The Gan-Gross-Prasad problem is concerned with the multiplicity:
\begin{equation}\label{ggpb}
m_{\epsilon'_{v_0}}(\pi,\pi'):=\dim\rm{Hom}_{H(\Fq)}(\pi, \pi'\otimes\psi_{ v_0})=\dim\rm{Hom}_{\rm{O}_{\Fn-2\ell-1}^\e(\Fq)}(\pi', \rm{Hom}_{N_\ell(\Fq)}(\pi,\psi_{ v_0})),
\end{equation}
where $\epsilon_{v_0}=\disc(\langle v_0\rangle)$.
If $\ell=0$, then
\[
m_{\epsilon'_{v_0}}(\pi,\pi'):=\langle \pi,\pi'\rangle_{\rm{O}_{\Fn-1}^\e(\Fq)},
\]
and we call the above multiplicity the basic case for Bessel case.
Let $v_0$, $v_0^*\in V_{n-2\ell}$ be two anisotropic vectors such that $\langle v_0,v_0\rangle_{V_\Fn}/\langle v_0^*,v_0^*\rangle_{V_\Fn}\in \Fq^{\times }/(\Fq^{\times })^2$. Let $W^*$ be the orthogonal complement of $v_0^*$ in $V_{\Fn-2\ell}$. Assume that $\disc(W^*)=\epsilon^*$. Then the stabilizer of $\psi_{v_0^*}$ is $\o^{\epsilon^*}_{\Fn-2\ell-1}$. If $\Fn$ is even, then $\o^\e_{\Fn-2\ell-1}$ and $\o^{\epsilon^*}_{\Fn-2\ell-1}$ are the same as abstract groups; however they act on two symmetric spaces with different discriminants, and $\o^\e_{\Fn-2\ell-1}$, $\o^{\epsilon^*}_{\Fn-2\ell-1}$ are not conjugate in $\o^\epsilon_{\Fn-2\ell}$. If $\Fn$ is odd, then one of $\{W,W^*\}$ is split and the other is not. Thus we get $\o^\e_{\Fn-2\ell-1}\ncong\o^{\epsilon^*}_{\Fn-2\ell-1}$ in this case.

Define the {\sl first descent index} ${\ell}_0:={\ell}_0^\rm{B}(\pi)$ of $\pi$ in the Bessel case to be the largest nonnegative integer ${\ell}_0<\Fn/2$ such that $m_{\epsilon''}(\pi,\pi')\neq 0$ for some irreducible representation $\pi'$ of $\o^{\epsilon'}_{\Fn-2\ell_0-1}\fq$ and some ${\epsilon''}\in\{\pm\}$. Define the {\sl first descent} of $\pi$ with respect to $\epsilon''$ to be
\[
\CD^{\rm{B}}_{\ell_0,\epsilon''}(\pi):=\rm{Hom}_{N_{\ell_0}(\Fq)}(\pi,\psi_{ v_0})
\]
where $v_0\in V_{\Fn-2\ell_0}$ is an anisotropic vector such that $\disc(\langle v_0\rangle)=\epsilon''$.
\subsection{Fourier-Jacobi case}\label{sec6.2}

We next turn to the Fourier-Jacobi case. Let $\Fn=2n$, and $V_{2n}$ be a symplectic space.
Consider pairs of symplectic spaces $V_{2(n-\ell)}\subset V_{2n}$, and the decomposition
\[
V_n=X+V_{2(n-\ell)}+X^\vee
\]
We use similar notations for various subspaces and subgroups as in the Bessel case. Note that if we let $P$ be the parabolic subgroup of $\rm{Sp}_{2n}$ stabilizing $X$ and let $N$ be its unipotent radical, then $N_\ell=U_{\GGL_\ell}\ltimes N$. Let $\omega_{n-\ell}^\epsilon$ be the Weil representation of $\sp_{2(n-\ell)}(\Fq)\ltimes \mathcal{H}_{2(n-\ell)}$, where $\mathcal{H}_{2(n-\ell)}$ is the Heisenberg group of $V_{2(n-\ell)}$. Roughly speaking, there is a natural homomorphism $N(\Fq)\to \mathcal{H}_{2(n-\ell)}$ invariant under the conjugation action of $U_{\GGL_\ell}(\Fq)$ on $N(\Fq)$, which enables us to view $\omega_{n-\ell}^\epsilon$ as a representation of $\sp_{2(n-\ell)}(\Fq)\ltimes N_{\ell}(\Fq)$. Let $\psi^\epsilon_\ell$ be the character of $U_{\GGL_\ell}(\Fq)$ given by
\[
\psi^\epsilon_\ell(z)=\psi\left(a\sum^{\ell-1}_{i=1}z_{i,i+1}\right),\quad z\in U_{\GGL_\ell}(\Fq),\textrm{ and }
\left\{
\begin{array}{ll}
a=1&\textrm{ if }\epsilon=+;\\
a\in \Fq^\times\slash(\Fq^\times)^2&\textrm{ if }\epsilon=-.\\
\end{array}\right.
\]
For the Fourier-Jacobi case, put
\begin{equation}\label{hnu'}
H=\sp_{2(n-\ell)}\ltimes N_\ell.
\end{equation}
Similar to the Bessel case, the Gan-Gross-Prasad problem is concerned with the multiplicity:
\begin{equation}\label{ggpfj}
m_\epsilon(\pi,\pi'):=\rm{dim}\rm{Hom}_{H(\Fq)}(\pi, \pi'\otimes(\omega_{n-\ell}^\epsilon\otimes\psi^\epsilon_\ell))=\dim\rm{Hom}_{\sp_{2(n-\ell)\fq}}(\pi', \rm{Hom}_{N_\ell(\Fq)}(\pi,\omega_{n-\ell}^\epsilon\otimes\psi^\epsilon_\ell)).
\end{equation}
If $\ell=0$, then
\[
m_\epsilon(\pi,\pi'):=\langle \pi,\pi'\otimes\omega^\epsilon_{n}\rangle_{\sp_{2n}(\Fq)},
\]
and we call the above multiplicity the basic case for Bessel case.

Define the {\sl first descent index} ${\ell}_0:={\ell}_0^\rm{FJ}(\pi)$ of $\pi$ in the Fourier-Jacobi case to be the largest nonnegative integer ${\ell}_0<n$ such that $m_\epsilon(\pi,\pi')\neq 0$ for some irreducible representation $\pi'$ of $\sp_{\Fn-2\ell_0}$ and some $\epsilon\in\{\pm\}$. Define the {\sl first descent} of $\pi$ with respect to $\epsilon$ to be
\[
\CD^{\rm{B}}_{\ell_0,\epsilon}(\pi):=\rm{Hom}_{N_{\ell_0}(\Fq)}(\pi,\omega_{n-\ell_0}^\epsilon\otimes\psi^\epsilon_{\ell_0}).
\]

\subsection{Reduction to the basic case}

From now until the end of this section, we assume that the cardinality of $\Fq$ is large
enough so that the main result in \cite{S} holds.
We first show that parabolic induction preserves multiplicities, and thereby make a reduction to the basic case.
\begin{proposition}{\cite[Proposition 7.3]{Wang1}}\label{so2}
Let be a semisimple element of $(G^*)^F=\so_{2n+1}\fq$, and $s'$ be a semisimple element of $(\sp_{2m}(\Fq)^*)^F=\so_{2m+1}\fq$. Let $\pi\in\cal{E}(\sp_{2n},s)$ be an irreducible representation of $\sp_{2n}(\Fq)$, and let $\pi'\in\cal{E}(\sp_{2m},s')$ be an irreducible representation of $\sp_{2m}$ with $n \ge m$. Let $P$ be an $F$-stable maximal parabolic subgroup of $\sp_{2n}$ with Levi factor $\GGL_{n-m} \times \sp_{2m}$. Let $s_0$ be a semisimple element of $\GGL_{n-m}\fq$ and let $\tau\in\cal{E}(\GGL_{n-m},s_0)$ be an irreducible cuspidal representation of $\GGL_{n-m}\fq$ which is nontrivial  if $n-m=1$. Assume that $s_0$ has no common eigenvalues with $s$ and $s'$. Then we have
\begin{equation}\label{sp}
m_\epsilon(\pi, \pi')=\langle  \pi\otimes \omega_n^{\ee\epsilon}, \rm{Ind}_{P\fq}^{\sp_{2n}\fq}(\tau\otimes\pi')\rangle _{\sp_{2n}(\Fq)}.
\end{equation}
\end{proposition}

We also have similar result for Bessel case.
\begin{proposition}{\cite[Proposition 7.4, Corollary 7.5]{Wang1}} \label{7.21}
Let $s$ be a semisimple element of $\rm{SO}^\epsilon_n(\Fq)$, and $s'$ be a semisimple element of $\rm{SO}^{\epsilon'}_m(\Fq)$. Let $\pi\in\cal{E}(\rm{SO}^\epsilon_n,s)$ be an irreducible representation of $\rm{SO}^\epsilon_n(\Fq)$, and let $\pi'\in\cal{E}(\rm{SO}^{\epsilon'}_m,s')$ be an irreducible representation of $\rm{SO}^{\epsilon'}_m(\Fq)$ with $n > m$, $n\equiv m+1$ mod $2$. Let $P$ be an $F$-stable maximal parabolic subgroup of $\rm{SO}^{\epsilon'}_{n+1}$ with Levi factor $\GGL_{\ell} \times \rm{SO}^{\epsilon'}_{m}$, $\ell=(n+1-m)/2$. Let $s_0$ be a semisimple element of $\GGL_{\ell}(\bb{F}_{q})$. Let $\tau_1$ (resp. $\tau_2$) be an irreducible cuspidal representations of $\GGL_{\ell'}(\bb{F}_{q})$ (resp. $\GGL_{\ell-\ell'}(\bb{F}_{q}))$, $\ell'\leq \ell$, which is nontrivial if $\ell'=1$ (resp.  $\ell-\ell'=1$), and
\[
\tau=\rm{Ind}_{\GGL_{\ell'}\fq\times  \GGL_{\ell- \ell'}\fq}^{\GGL_\ell\fq}(\tau_1\times\tau_2).
\]
Assume that $\tau\in\cal{E}(\GGL_\ell,s_0)$, and $s_0$ has no common eigenvalues with $s$ and $s'$. Then we have
 \begin{equation}\label{7.22}
m_{\ee'\e\epsilon}(\pi, \pi')=\langle \rm{Ind}^{\so^\e_{n+1}\fq}_{P\fq}(\tau\otimes\pi'),\pi\rangle _{\so^\epsilon_{n}(\Fq)}=m_{\ee'\ee\e\epsilon}\left(\pi,\rm{Ind}^{\so^\e_{n+1-2\ell'}\fq}_{\GGL_{\ell-\ell'}\fq\times \so^\e_{m}\fq}(\tau_2\otimes\pi')\right),
\end{equation}
where
\[\ee'=\left\{
\begin{array}{ll}
 +& \textrm{ if $n$ is odd};\\
  \ee& \textrm{ if $n$ is even}.\\
\end{array}
\right.
\]
\end{proposition}

\begin{remark}\label{rmk1}
Recall that we assume that the order $q$ of finite filed $\Fq$ is large enough such that the main theorem in \cite{S} holds. This assumption guarantees the existence of $\tau$, which meets the conditions outlined in Proposition \ref{7.21}. For a quadratic unipotent representation $\pi$, we only need to assume that $q$ is odd for $G$ to ensure the existence of $\tau$. In other words, Proposition \ref{so2} and Proposition \ref{7.21} apply to quadratic unipotent representation when $q$ is odd. Moreover, in \cite{Wang1}, the assumption of $q$ is only for obtaining Proposition \ref{so2} and Proposition \ref{7.21}. Therefore, all calculations of Gan-Gross-Prasad in \cite{Wang1} apply to quadratic unipotent representations when $q$ is odd.
\end{remark}

In order to apply the theta correspondence we will work with orthogonal groups instead of special orthogonal groups. In Proposition \ref{7.21}, for $m=0$, assume that $\tau\in\cal{E}(\GGL_\ell,s_0)$ such that $\pm1$ are not eigenvalues of $s_0$.
Let $\bf{1}$ denotes the trivial representation of trivial group. By \cite[Proposition 3.2]{Wang1}, we set
\[
m_{\epsilon\e\ee'}(\pi,{\bf{1}}):=\langle \rm{Ind}^{\o^\e_{n+1}\fq}_{P\fq}(\tau),\pi\rangle _{\o^\epsilon_{n}(\Fq)}=\left\{
\begin{array}{ll}
1,&\textrm{ if }\pi \textrm{ is regular };\\
0,&\textrm{ otherwise},
\end{array}
\right.
\]
where $\ee'$ is defined in Proposition \ref{7.21}.
By the standard arguments of theta correspondence and see-saw dual pairs, we set
\[
m_\epsilon(\pi,{\bf{1}}):=\langle  \pi\otimes \omega_n^{\ee\epsilon}, \rm{Ind}_{P\fq}^{\sp_{2n}\fq}(\tau)\rangle _{\sp_{2n}(\Fq)}=\left\{
\begin{array}{ll}
1,&\textrm{ if }\pi \textrm{ is regular };\\
0,&\textrm{ otherwise.}
\end{array}
\right.
\]

\subsection{Strongly relevant pair of representations}\label{6.3}

For a pair of irreducible representations $(\pi,\pi')$, whether the multiplicity (\ref{ggpb}) and (\ref{ggpfj}) vanishes depends on the behavior of the pair in see-saw.
\begin{example}\label{ex1}
Let $\pi$ be a unipotent cuspidal representations of $\o^+_{5}\fq$. By \cite[Theorem 3.12]{LW1}, there exists a representation $\pi_0\in[\pi]$ such that $\Theta^+_{2,1}(\pi_0)=\pi_1$ where $\pi_1$ is a cuspidal representation of $\sp_2\fq$. Let $\pi'\in\cal{E}(\o^+_4,s)$ where $s$ has no eigenvalues $\pm1$. Then by Theorem \ref{p1}, the first occurrence index of $\pi'$ is $2$, and $\Theta^+_{2,2}(\pi')=\pi'_1$ where $\pi'_1$ is an irreducible representation of $\sp_4\fq$. Consider the see-saw diagram
\[
\xymatrix{
\sp_{2}\times\sp_{2} \ar@{-}[rd] \ar@{-}[d]& \o^+_{5}\ar@{-}[d]
\\
\sp_{2} \ar@{-}[ru] &  \o^+_{4}\times \o^+_1
}.
\]
We have
\[
m_+(\pi_0,\pi')=\langle\pi_0,\pi'\rangle_{\o^+_{4}\fq}\le \langle\Theta^+_{1,2}(\pi_1),\pi'\rangle_{\o^+_{4}\fq}=\langle\pi_1,\Theta^+_{2,1}(\pi')\otimes\omega_1^+\rangle_{\sp_{2}\fq}=0.
\]
\end{example}
We need to choose appropriate pairs of representations to avoid the above situation. These pairs of irreducible representations are called strongly relevant or $(\psi,\ee)$-strongly relevant, as defined in \cite[subsection 6.3]{Wang1}. Although the definitions might appear complex, we will only use them for basic representations. In such cases, the definition is quite straightforward. In this paper, we always fix a character $\psi$ and write  weil representations of $\sp_{2n}\fq$ by $\omega_n^\pm$ instead of $\omega_{n,\psi}$ and $\omega_{n,\psi'}$, where $\psi'(x)=\psi(ax)$ with $x\in \Fq^\times$ and $a\notin (\Fq^\times)^2$. So we use $\{\pm \}$-strongly relevant instead of $(\psi,\ee)$ and $(\psi',\ee)$-strongly relevant.
\begin{definition}
\begin{enumerate}

\item  Let $\pi\in \CE_\rm{bas}(\sp_{2n})$, and $\pi'\in\CE_\rm{bas}(\sp_{2m})$. Pick $\epsilon'\in\{\pm\}$ such that $\pi$ goes down with respect to the Witt tower ${\bf O}^\e_\rm{even}$. We say the pair of representations $(\pi,\pi')$ is $\epsilon$-relevant if $\neven^\e(\pi)=\nodd^{\ee\epsilon\epsilon'}(\pi')-1$ or $\neven^\e(\pi)=\nodd^{\ee\epsilon\epsilon'}(\pi')$. We say the pair of representations $(\pi,\pi')$ is $\epsilon$-strongly relevant if $(\pi,\pi')$ is $\epsilon$-relevant and $(\pi',\pi)$ is $\ee\epsilon$-relevant.

\item Let $\pi\in \CE_\rm{bas}(\rm{O}^\epsilon_{2n+1})$, and let $\pi'\in\CE_\rm{bas}(\rm{O}^{\epsilon'}_{2m})$.
Pick $\chi\in \{1,\rm{sgn}\}$ such that $\chi\cdot\pi$ goes down with respect to the Witt tower ${\bf Sp}$. Let $\chi'$ be a character of $\rm{O}^{\epsilon'}_{2m}$ such that
\[
\chi'=\left\{
\begin{array}{ll}
1& \textrm{ if }\chi=1;\\
\rm{sgn}& \textrm{ if }\chi=\rm{sgn}.
\end{array}
\right.
\]
 We say the pair of representations $(\pi,\pi')$ is relevant if
${\bf{n}}^\epsilon(\chi\cdot\pi)={\bf{n}}^{\epsilon'}(\chi'\cdot\pi')-1\textrm{ or }{\bf{n}}^\epsilon(\chi\cdot\pi)={\bf{n}}^{\epsilon'}(\chi'\cdot\pi')$.
 We say the pair of representations $(\pi,\pi')$ is strongly relevant if both $(\pi,\pi')$ and $(\chi_{\rm{O}^\epsilon_{2n+1}}\otimes\pi,\chi_{\rm{O}^{\epsilon'}_{2n'}}\otimes\pi')$ are relevant where $\chi_{G}$ is the spinor norm of $G^F$.
 \end{enumerate}
 \end{definition}

\begin{example} Consider the following two see-saw diagrams
 \[
\xymatrix{
\sp_{2n}\times\sp_{2n} \ar@{-}[rd] \ar@{-}[d]& \o^\epsilon_{2n'}\ar@{-}[d] &\quad\quad&\sp_{2n}\times\sp_{2n} \ar@{-}[rd] \ar@{-}[d]& \o^\epsilon_{2n'+1}\ar@{-}[d]
\\
\sp_{2n} \ar@{-}[ru] &  \o^{\ee\cdot\epsilon}_{2n'+1}\times \o^{\ee}_1&\quad\quad &\sp_{2n} \ar@{-}[ru] &  \o^{\epsilon}_{2n'}\times \o^{+}_1.
}.
\]
As in Example \ref{ex1}, we can conclude that if $(\pi,\pi')$ is not a $\ee$-strongly relevant pair, then the multiplicity $m_+(\pi,\pi')$ vanishes. In particular, if the multiplicity $m_+(\pi,\pi')$ does not vanish, then $\pi$ and $\pi'$ have to go up or down simultaneously in the same see-saw diagram.

\end{example}

The $\epsilon$-relevant pairs for cuspidal representations have been described in \cite[Section 6]{Wang1}. It is not difficult to extend these results to basic representations.

\begin{corollary}{\cite[Corollary 6.6]{Wang1}}\label{srsp}
Let $\pkh\in \CE_\rm{bas}(\sp_{2n})$, and $\pkhp\in \CE_\rm{bas}(\sp_{2m})$.

(i) If $(\pkh,\pkhp)$ is $\epsilon$-relevant for some $\epsilon\in\{\pm \}$, then $k=\lfloor |h'|\rfloor $ or $k=\lceil |h'|-1\rceil $.

(ii) If $(\pkh,\pkhp)$ is $\epsilon$-relevant for some $\epsilon\in\{\pm \}$, then $(\pkh,\pi_{\rho,k',-h'})$ is not $\epsilon$-relevant.

(iii) If $(\pkh,\pkhp)$ is not $\epsilon$-relevant for some $\epsilon\in\{\pm \}$ and $k=\lfloor |h'|\rfloor $ or $k=\lceil |h'|-1\rceil $, then $(\pkh,\pi_{\rho,k',-h'})$ is $\epsilon$-relevant.
\end{corollary}

\begin{corollary}{\cite[Corollary 6.7]{Wang1}}\label{sro}
(i) Let $\pkh\in \CE_\rm{cus}(\o^\epsilon_{2n})$, and $\pi_{\rho',k',h',\iota'}\in \CE_\rm{cus}(\o^\e_{2m+1})$. If $(\pkh,\pi_{\rho',k',h',\iota'})$ is relevant, then $\lfloor|k|\rfloor=k'$ or $\lceil|k|-1\rceil=k'$.

(ii) Let $\pi_{\rho,k,h,\iota}\in \CE_\rm{cus}(\o^\epsilon_{2n+1})$, and $\pkhp\in \CE_\rm{cus}(\o^\e_{2m})$. If $(\pi_{\rho,k,h,\iota},\pkhp)$ is relevant, then $\lfloor|k'|\rfloor=k$ or $\lceil |k'|\rceil=k+1$.
\end{corollary}

\subsection{First descent of basic representation}\label{sec6.5}

Consider an irreducible cuspidal representation $\pi$. We now turn to calculate the first descent of representations in $[\pi]$.

\begin{proposition}\label{fd1}
Let $\pi$ be a basic representation of $\sp_{2n}\fq$ (resp. $\o^\epsilon_{2n}\fq$, $\o^\epsilon_{2n+1}\fq$). Assume that $\ell$ be the first descent index of $\pi$. Let $\pi'$ be an irreducible representation of $\sp_{2(n-\ell)}\fq$ (resp. $\o^\e_{2(n-\ell)-1}\fq$, $\o^\e_{2(n-\ell)}\fq$) such that $m_\e(\pi,\pi')\ne 0$. Then $\pi'$ is basic.
\end{proposition}
\begin{proof}
The cuspidal case for symplectic groups has been proven in \cite[Theorem 6.16]{PW1}. It is easy to extend the result to basic representations.  The proof for the orthogonal group follows a similar approach
\end{proof}

\begin{proposition}\label{fdr}
(i) Let $\pi=\pi_{\rho}\in \CE_\rm{bas}(\sp_{2n})$. Assume that $\ell$ be the first descent index of $\pi$. Let $\pi'\in \CE(\sp_{2(n-\ell)})$ such that $m_\epsilon(\pi,\pi')\ne 0$. Then we have $m_{-\epsilon}(\pi,\pi')\ne 0$ and $\pi'={\pi_{\rho^-_1}}$ (see \eqref{rho-} and \eqref{rho1} for the definition of $\rho_1^-$).

(ii) Let $\pi=\pi_{\rho}\in \CE_\rm{bas}(\o^\epsilon_{2n})$ (resp. $\pi_{\rho,\iota}\in  \CE_\rm{cus}(\o^\epsilon_{2n+1})$). Assume that $\ell$ be the first descent index of $\pi$. Let $\pi'\in \CE(\o^{\epsilon'}_{2(n-\ell)-1})$ (resp. $\CE(\o^{\epsilon'}_{2n-\ell})$) such that $m_{\epsilon''}(\pi,\pi')\ne 0$ with $\epsilon''=\ee\epsilon\e$ (resp. $\epsilon''=\epsilon\e$). Then we have $\pi'=\pi_{\rho_1,\wi_{\rho_1}}$ (resp. $\pi_{\rho_1}$).

\end{proposition}
\begin{proof}
The symplectic has been proved in \cite[Theorem 6.3]{PW1}. The orthogonal case follows from the symplectic case and see-saw identity. When $\pi'$ is a representation of a odd orthogonal group, we obtain $\pi'(-I)$ by Proposition \ref{-1}.
\end{proof}

\begin{proposition}\label{fd2}

 \begin{enumerate}
\item Let $\pi=\pkh\in \CE_\rm{bas}(\sp_{2n})$. Assume that $\ell$ be the first descent index of $\pi$. Let $\pkhp\in \CE_\rm{bas}(\sp_{2(n-\ell)})$. If $m_\epsilon(\pkh,\pkhp)\ne 0$ for some $\epsilon\in\{\pm\}$, then we have
 \begin{itemize}
 \item (i) $\rho'=\rho^-_1$;
\item (ii) $|k'|=\lceil |h|-1 \rceil$;
\item (iii) $|h'|=|k'|-1$.
\end{itemize}

\item Let $\pi=\pkh\in \CE_\rm{bas}(\o^\epsilon_{2n})$ (resp. $\pkhsgn\in \CE_\rm{bas}(\o^\epsilon_{2n+1})$).
Assume that $\ell$ be the first descent index of $\pi$. Let $\pi'=\pkhsp\in \CE_\rm{bas}(\o^\e_{2(n-\ell)-1})$ (resp. $\pkhp\in \CE_\rm{bas}(\o^\e_{2(n-\ell)}$). If $m_{\epsilon''}(\pi,\pi')\ne 0$ with $\epsilon''=\ee\epsilon\e$ (resp. $\epsilon''=\epsilon\e$), then we have
 \begin{itemize}
 \item (i) $\rho'=\rho_1$;
\item (ii) $|k'|=\lceil|k|-1\rceil$;
\item (iii) $|h'|=\lceil|h|-1\rceil$.
\end{itemize}

 \end{enumerate}

\end{proposition}
\begin{proof}
We only prove (1). The proof of (2) is similar, and left it to readers.
By \cite[Theorem 1.1]{Wang1} we know that $\pi$ and $\pi'$ are $\epsilon$-strongly relevant for some $\epsilon$, and $m_\pm(\pi^\sp_\rho,\pi^\sp_{\rho'})\ne 0$.
  Now, let $\pi^\star=\pkhs$ such that
 \begin{itemize}
 \item (i) $\rho^\star=\rho'$;
\item (ii) $|h^\star|= |k|-1$;
\item (iii) $|k^\star|= \lceil|h|-1 \rceil$.
\end{itemize}
According to Proposition \ref{cus1}, there exists $\epsilon^\star$ such that one of $(\pi,\pi^\star)$ and $(\pi,(\pi^\star)^c)$ is $\epsilon^\star$-strongly relevant.
Hence, by \cite[Theorem 1.1]{Wang1}, either $m_{\epsilon^\star}(\pi,\pi^\star)\ne 0$ or $m_{\epsilon^\star}(\pi,(\pi^\star)^c)\ne 0$. By Corollary \ref{srsp}, if $m_{\epsilon''}(\pi,\pi')\ne 0$, then $|h'|\ge |h^\star|$ and $|k'|\ge |k^\star|$.
Since $\ell$ is the first descent, we have $|k'|=|k^\star|$, $|h'|=|h^\star|$, and $\rho'$ should appear in the first descent of $\rho$. By Proposition \ref{fdr}, we know that $\rho'=\rho^-_1$.
\end{proof}

Although we can not explicitly describe which irreducible representation appears in the first descent of an irreducible basic representation $\pi$ by Proposition \ref{fd2}, we find a irreducible basic representation $\pi'$ such that $[\pi']$ contains all irreducible constituent of first descents of $[\pi]$. Similar to theta correspondence, we are going to prove conservation relation of first descents for basic representations, which allows us to determined which representation in $[\pi']$ appears in the descent of $\pi$.

\begin{theorem}\label{d2}
Let $\pkh\in \CE_\rm{bas}(\sp_{2n})$. Let $\ell$ be the first descent index of $\pkh$. Assume that $h\ne 0$ and $\pkhp$ appears in $\CD^{\rm FJ}_{\ell,\epsilon}(\pkh)$. Then we have
 \begin{itemize}
\item (i) $\CD^{\rm{FJ}}_{\ell,-\epsilon}(\pkh)=0$;
\item (ii) $ \CD^{\rm{FJ}}_{\ell,\epsilon}(\pkh^c)=0$;
\item (iii) $ \CD^{\rm{FJ}}_{\ell,-\epsilon}(\pkh^c)=\pkhp^c$.
\end{itemize}
\end{theorem}
\begin{proof}
By Proposition \ref{so2} and \cite[Proposition 11]{Sze},
we have
\begin{equation}\label{dc}
  \begin{aligned}
m_\epsilon(\pkh,\pkhp)=&\langle \pkh,\rm{Ind}^{\sp_{2n}\fq}_{P\fq}(\tau\otimes\pkhp)\otimes\omega^\epsilon_{n} \rangle_{\sp_{2n}\fq}\\
=&\left\langle \pkh^c,\left(\rm{Ind}^{\sp_{2n}\fq}_{P\fq}(\tau\otimes\pkhp)\otimes\omega_{n}^\epsilon\right)^c \right\rangle_{\sp_{2n}\fq}\\
=&\langle \pi_{\rho,k,-h},\rm{Ind}^{\sp_{2n}\fq}_{P\fq}(\tau^c\otimes\pi_{\rho',k',-h'})\otimes\omega_{n}^{-\epsilon} \rangle_{\sp_{2n}\fq}\\
=&m_{-\epsilon}(\pi_{\rho,k,-h},\pi_{\rho',k',-h'}),
  \end{aligned}
\end{equation}
where $P$ is a parabolic subgroup of $\sp_{2n}$ with Levi factor $\GGL_\ell\times\sp_{2(n-\ell)}$, and $\tau\in \CE(\GGL_{\ell})$ is defined in Proposition \ref{so2}. In particular, we can pick a $\tau$ such that $\rm{Ind}^{\sp_{2n}\fq}_{P\fq}(\tau\otimes\pkhp)$ is irreducible. This completes the proof of (iii).

 Let $n^\pm$ be the first theta occurrence index of $\pkh$ with respect to the Witt tower ${\bf{{O}}}_{\rm{odd}}^\pm$. By Proposition \ref{cus1} (ii), there exists ${\epsilon^\star}$ such that $n^{\epsilon^\star}< n$, the theta lifting $\Theta^{\epsilon^\star}_{n,n^{\epsilon^\star}}(\pi)$ is an irreducible cuspidal representation, and we denote it by $\ps=\pkhss$.
On the other hand, we know that $\pkh$ occurs in $\Theta^{\epsilon^\star}_{n^{\epsilon^\star},n}(\ps)$.
Consider the see-saw diagram
\[
\xymatrix{
\sp_{2n}\times\sp_{2n} \ar@{-}[rd] \ar@{-}[d]& \o^{\epsilon {\epsilon^\star}}_{2(n^{\epsilon^\star}+1)}\ar@{-}[d]
\\
\sp_{2n} \ar@{-}[ru] & \o^{\epsilon^\star}_{2n^{\epsilon^\star}+1}\times \o^{\ee\epsilon}_1
}.
\]
Even though we are considering Ma-Qiu-Zou's Weil representation instead of G\'erardin's, the see-saw identity still remains valid.
We have
 \begin{equation}\label{dt}
  \begin{aligned}
m_{\epsilon}(\pkh,\pkhp)=&\langle \pkh\otimes\omega_{n}^{\ee\epsilon} ,\rm{Ind}^{\sp_{2n}\fq}_{P\fq}(\tau\otimes\pkhp)\rangle_{\sp_{2n}\fq}\\
\le&\langle \Theta^{\epsilon^\star}_{n^{\epsilon^\star},n}(\ps) ,\rm{Ind}^{\sp_{2n}\fq}_{P\fq}(\tau\otimes\pkhp)\rangle_{\sp_{2n}\fq}\\
=&\langle \pkhss ,\Theta^{\epsilon\epsilon^\star}_{n, n^{\epsilon^\star}+1}(\rm{Ind}^{\sp_{2n}\fq}_{P\fq}(\tau\otimes\pkhp))\rangle_{\sp_{2n}\fq}.\\
  \end{aligned}
\end{equation}
Therefore,
$\Theta^{\epsilon\epsilon^\star}_{n, n^{\epsilon^\star}+1}(\rm{Ind}^{\sp_{2n}\fq}_{P\fq}(\tau\otimes\pkhp))\ne 0,
$
 and by Proposition \ref{w2}, we know that $\pkhp$ goes down with respect to the Witt tower ${\bf{{O}}}_{\rm{even}}^{\epsilon\epsilon^\star}$.

By Proposition \ref{fd2}, if $\CD^{\rm FJ}_{\ell,-\epsilon}(\pkh)\ne 0$, then it consists of some of $\pkhp$ and $\pi_{\rho',k',-h'}$. First, we prove that $ \pkhp$ does not appear in $\CD^{\rm FJ}_{\ell,-\epsilon}(\pkh)$.
Consider the the see-saw diagram
\[
\xymatrix{
\sp_{2n}\times\sp_{2n} \ar@{-}[rd] \ar@{-}[d]& \o^{\epsilon {\epsilon^\star}}_{2(n^{\epsilon^\star}+1)}\ar@{-}[d]
\\
\sp_{2n} \ar@{-}[ru] & \o^{-\epsilon^\star}_{2n^{\epsilon^\star}+1}\times \o^{-\ee\epsilon}_1
}.
\]
We can use the same see-saw argument to conclude that if $
m_{-\epsilon}(\pkh,\pkhp)\ne 0
$, then
$\pkhp$ goes down with respect to the Witt tower ${\bf{{O}}}_{\rm{even}}^{-\epsilon\epsilon^\star}$. However, we already proved that $\pkhp$ goes down with respect to the Witt tower ${\bf{{O}}}_{\rm{even}}^{\epsilon\epsilon^\star}$. Therefore
by conservation relation, $\pkhp$ has to go up with respect to the Witt tower ${\bf{{O}}}_{\rm{even}}^{-\epsilon\epsilon^\star}$. Hence we have
$
m_{-\epsilon}(\pkh,\pkhp)=0.$
By Theorem \ref{cus1}, the first theta occurrence indexes of $\pi_{\rho',k',-h'}$ and $\pkhp$ with respect to the Witt towers ${\bf{{O}}}_{\rm{even}}^{\pm}$ are the same. Using the same see-saw argument again, we have $
m_{-\epsilon}(\pkh,\pi_{\rho',k',-h'})=0$. So $\CD^{\rm FJ}_{\ell,-\epsilon}(\pkh)$ contains neither   $\pkhp$ nor $\pi_{\rho',k',-h'}$,  which proves (i). We can conclude (ii) from (i) and \eqref{dc}.
\end{proof}

\begin{theorem}\label{deven}
Let $\pkh$ be an irreducible basic representation $\o_{2n}^\epsilon\fq$. Assume that $k,h\ne 0$. Let $\ell$ be the first descent index of $\pkh$. Assume that $\pkhsp$ appears in $\CD^{\rm{B}}_{\ell,\e}(\pkh)$ for some $\e$. Then we have
 \begin{itemize}
\item[(i)] $\CD^{\rm{B}}_{\ell,\e}(\sgn\pkh)=\sgn\pkhsp$;
\item[(ii)] $\CD^{\rm{B}}_{\ell,\e}(\chi_{\o_{2n}^\epsilon}\cdot\pkh)=\chi_{\o_{2(n-\ell)-1}^\e}\cdot\pkhsp$;
\item[(iii)] $\CD^{\rm{B}}_{\ell,-\e}(\pkh^c)=\pkhsp$;
\item[(iv)] $\CD^{\rm{B}}_{\ell,-\e}(\pkh)=0$;
\item[(v)] $\CD^{\rm{B}}_{\ell,\e}(\pkh^c)=0$,
\end{itemize}
where $\chi_G$ is the spinor norm of $G^F$.
\end{theorem}
\begin{proof}
The statements (i) and (ii) are obvious, and (v) follows from (iii) and (iv). Let us prove (iii). By Proposition \ref{7.21}, for any $\pi'\in \CE(\o^{\ee\epsilon\e}_{2n-2\ell-1})$, we have
\begin{equation}\label{eq5.14}
  \begin{aligned}
  m_\e(\pkh,\pi')&=m_{\e}(\pkh,\rm{Ind}^{\o^{\ee\epsilon\e}_{2n-1}\fq}_{P\fq}(\tau\otimes\pi'))\\
  &=
\langle \pkh ,\rm{Ind}^{\o^{\ee\epsilon\e}_{2n-1}\fq}_{P\fq}(\tau\otimes\pi')\rangle_{\o_{2n}^{\epsilon}\fq},
  \end{aligned}
\end{equation}
where $P$ and $\tau$ are defined in Proposition \ref{7.21}.
Let $g$ be a element of $\rm{CO}_{2n}^\epsilon$ sending $\pkh$ to $\pkh^c$. Then ${}^g(\o_{2n-1}^{\ee\epsilon\e})$ is the isometry group of a $(2n-1)$-dimensional symmetric space $V$ with $\disc(V)=-\ee\epsilon\e$, and $g$ sends $\pi'$ to an irreducible representation ${}^g\pi'$ of ${}^g(\o_{2n-1}^{\ee\epsilon\e})\fq$. By direct calculation, $\pi'$ and ${}^g\pi'$ are the same as irreducible representations of $\o_{2n-1}\fq$. By abuse of notations, we denote these two irreducible representations by $\pi'$.
Applying the conjugate action of $g$ on \eqref{eq5.14}, we have
\[
  \begin{aligned}
\langle \pkh ,\rm{Ind}^{\o^{\ee\epsilon\e}_{2n-1}\fq}_{P\fq}(\tau\otimes\pi')\rangle_{\o_{2n}^{\epsilon}\fq}=\langle \pkh^c ,\rm{Ind}^{\o^{-\ee\epsilon\e}_{2n-1}\fq}_{P'\fq}(\tau\otimes\pi')\rangle_{\o_{2n}^{\epsilon}\fq}=  m_{-\e}(\pkh^c,\pi')
  \end{aligned}
\]
where $P'={}^gP$.
Hence $\CD^{\rm{B}}_{\ell,-\e}(\pkh^c)=\pkhsp$.

We now turn to prove (iv).
Assume that $\CD^{\rm{B}}_{\ell,-\e}(\pkh)\ne0$. Then by Proposition \ref{fd2} , $\CD^{\rm{B}}_{\ell,-\e}(\pkh)$ equals either $\pkhsp$ or $\pi_{\rho',k',h',-\iota'}$. Consider the following see-saw diagrams
 \[
\xymatrix{
\sp_{2n'}\times\sp_{2n'} \ar@{-}[rd] \ar@{-}[d]& \o^\epsilon_{2n}\ar@{-}[d] &\quad\quad&\sp_{2n'}\times\sp_{2n'} \ar@{-}[rd] \ar@{-}[d]& \o^\epsilon_{2n}\ar@{-}[d]
\\
\sp_{2n'} \ar@{-}[ru] &  \o^{\ee\epsilon\e}_{2n-1}\times \o^{\e}_1&\quad\quad &\sp_{2n} \ar@{-}[ru] &  \o^{-\ee\epsilon\e}_{2n-1}\times \o^{-\e}_1.
}.
\]
Since $\CD^{\rm{B}}_{\ell,\e}(\pkh)=\pkhsp$, the representations $\pkh$ and $\pkhsp$ are both go up or go down in the first see-saw diagram. Then by Proposition \ref{cus2}, $\pkh $ and $\pi_{\rho',k',h',-\iota'}$ go in a different direction in the first see-saw diagram.
Recall that the first theta occurrence indexes of irreducible representations of odd orthogonal groups are the same with respect to two different Witt tower ${\bf {O}}^\pm_{\rm{odd}}$. We know that  $\pkh$ and $\pi_{\rho',k',h',-\iota'}$ go in a different direction in the second see-saw diagram, and $\CD^{\rm{B}}_{\ell,-\e}(\pkh)$ can not be $\pi_{\rho',k',h',-\iota'}$.

On the other hand, $\CD^{\rm{B}}_{2\ell,-a}(\pkh)$ can not be $\pkhsp$ either.
By (ii), we have
\[
\CD^{\rm{B}}_{\ell,\e}(\chi_{\o_{2n}^\epsilon}\cdot\pkh)=\chi_{\o_{2(n-\ell)-1}^\e}\cdot\pkhsp.
\]
As irreducible representations of $\o_{2(n-\ell)-1}\fq$,
\[
\chi_{\o_{2(n-\ell)-1}^{-\e}}\cdot\pkhsp\cong\sgn\chi_{\o_{2(n-\ell)-1}^\e}\cdot\pkhsp.
\]
Applying the same see-saw argument on $\chi_{\o_{2(n-\ell)-1}^{-\e}}\cdot\pkhsp$ , we have
\[
\CD^{\rm{B}}_{\ell,-\e}(\chi_{\o_{2n}^\epsilon}\cdot\pkh)\ne \chi_{\o_{2(n-\ell)-1}^{-\e}}\cdot\pkhsp.
\]
and hence
\[
\CD^{\rm{B}}_{\ell,-\e}(\pkh)\ne\pkhsp.
\]
So $\CD^{\rm{B}}_{\ell,-\e}(\pkh)$ is neither $\pkhsp$ nor $\pi_{\rho',\Lambda_1',\Lambda_{-1}',-\epsilon'}$.
\end{proof}
Similarly, we have the following result for odd orthogonal groups.
\begin{theorem}\label{dodd}
Let $\pkhsgn\in\CE_{\rm{bas}}(\o_{2n+1}^\epsilon)$. Assume that $k,h\ne 1$. Let $\ell$ be the first descent index of $\pkhsgn$. Assume that   $\CD^{\rm{B}}_{\ell,\e}(\pkhsgn)=\pkhp$ for some $\e$. Then
 \begin{enumerate}
\item[(i)] We have $\CD^{\rm{B}}_{\ell,-\e}(\pkhsgn)=0$.
\item[(ii)] If we recognize $\pkhsgn$ as an irreducible representation of $\o_{2n+1}^{-\e}\fq$, we have
$\CD^{\rm{B}}_{\ell,\pm \e}(\pkhsgn)=0$;
\item[(iii)] We have $\CD^{\rm{B}}_{\ell,\e}(\sgn\pkhsgn)=\sgn\pi_{\rho,k,h}$;
\item[(iv)] We have $\CD^{\rm{B}}_{\ell,\e}(\chi_{\o_{2n+1}^\epsilon}\cdot\pkhsgn)=\chi_{\o_{2n-\ell}^{\epsilon'}}\cdot\pi_{\rho,k,h}$.
 \end{enumerate}
\end{theorem}

Combing our statements of first descents and first theta lifting for basic representations, we have the following commutative diagram
 \begin{equation}\label{cm1}
\xymatrix{
\pi\in \CE_{\rm{bas}}(G_{n}) \ar[r]^{\Theta} \ar[d]_{\CD_{\ell,\epsilon}}& \pi'\in \CE_{\rm{bas}}(G'_{n'})\ar[d]^{\CD_{\ell',\epsilon}}
\\
\pi^\star\in \CE_{\rm{bas}}(G_{m})\ar[r]^{\Theta}& \pi^{\star\prime}\in \CE_{\rm{bas}}(G'_{m'})
},
 \end{equation}
where \begin{itemize}
\item $\Theta$ is the first theta lifting;
\item $m=n-\ell$ (resp. $m'=n'-\ell'$), $\ell$ (resp. $\ell'$) is the first descent index of $\pi$ (resp. $\pi'$);
\item $\CD_{\ell,\epsilon}$ (resp. $\CD_{\ell',\epsilon}$ ) is $\CD^{\rm{FJ}}_{\ell,\epsilon}$ or $\CD^{\rm{B}}_{\ell,\epsilon}$,  it depends on $G_n$ (resp. $G'_{n'}$);
\item $n'$ is the first occurrence index of $\pi$.
\end{itemize}

\begin{remark}\label{rmk2}
As mentioned in Remark \ref{rmk1}, all statements in this subsection apply to quadratic unipotent representations when $q$ is odd.
\end{remark}

\section{Construction of canonical choice of Lusztig correspondence}\label{sec66}

\subsection{The canonical choice of Lusztig correspondence}
 The goal of this section is to prove our main theorem.
\begin{theorem}\label{main}
 Suppose that the cardinality of $\Fq$ is large
enough so that the main result in \cite{S} holds. Let $G_n=\sp_{2n}$, $\o^{\epsilon}_{2n}$ or $\o^{\epsilon}_{2n+1}$.
One has a unique choice of Lusztig correspondence $\cal{L}^{\rm can}$ satisfying the following conditions.
\begin{enumerate}
\item  It is compatible with parabolic induction (cf. Section \ref{sec3.2}).

\item  Consider the dual pair $(G_n,G'_{n'})=(\sp_{2n},\o^{\epsilon'}_{2n'+1})$.
 Let $\pi=\prll\in \CE(\sp_{2n})$, and $\pi'=\prllps\in \CE(\o^{\epsilon'}_{2n'+1})$. Then $\prllps$ occurs in $\Theta(\prll)$ if and only if the following conditions hold:
\begin{itemize}
    \item For any $a\ne \pm 1$ in $\overline{\mathbb{F}}_q$, $G^*_{[a]}\cong (G')^*_{[-a]}$  and $\pi[a]\cong \pi'[-a]$ where $\rho=\prod\pi[a]$, and $\rho^\prime=\prod\pi'[a]$;
    \item $\Lambda_{-1}^\prime=\Lambda_{1}$;
    \item $\pi_{\Lambda^\prime_1}$ occurs in $\Theta_{\epsilon'\epsilon_\rho}(\pi_{\Lambda_{-1}})$;
    \item $\iota'=\wi_{\rho}\pd(\Lambda_{-1})\ee^{\rm{rank}(\Lambda_{1})}$,
\end{itemize}
where
\[
\Theta_{\epsilon'\epsilon_\rho}(\pi):=\begin{cases}
    \Theta(\pi), &  \textrm{if }\epsilon'\epsilon_\rho=+;\\
     \Theta(\rm{sgn}\cdot\pi), & \textrm{otherwise}.
\end{cases}
\]
In other words, we have the following commutative diagram
\[
\xymatrix{
\prll\in \CE(\sp_{2n}) \ar[rr]^\Theta \ar[d]_{\CL_{\rm can}} && \prllps\in  \CE(\o^{\epsilon'}_{2n'+1}) \ar[d]^{\CL_{\rm can}} \\
\rho\otimes\pi_{\Lambda_{-1}}\otimes\pi_{\Lambda_1} \ar[rr]^{ \rm{iso}\otimes\Theta_{\epsilon'\epsilon_\rho}\otimes\CL_1 } &  & \rho^\prime\otimes\pi_{\Lambda_1^\prime}\otimes\pi_{\Lambda^\prime_{-1}}\otimes\iota',
}
\]
where the first isomorphism sends $\rho$ to $\rho^\prime$ by sending $\pi[a]$ to $\pi^\prime[-a]$, and $\CL_1$ is the Lusztig correspondence between unipotent representations.

\item Consider the dual pair $(G_n,G_{n'}')=(\sp_{2n},\o^{\epsilon'}_{2n'})$. For  $\pi=\prll\in \CE(\sp_{2n})$, and $\pi'=\prllp\in \CE(\o^{\epsilon'}_{2n'})$. Then $\prllp$ occurs in $\Theta(\prll)$ if and only if the following conditions hold:
\begin{itemize}
\item For any $a\ne \pm 1$ in $\bar{\Fq}$, $ G^*_{[a]}\cong(G')^*_{[a]}$ and $\pi'[a]\cong \pi[a]$ where $\rho=\prod\pi[a]$, and $\rho'=\prod\pi'[a]$;
    \item $\Lambda_{-1}'=\Lambda_{-1}$;
    \item $\pi_{\Lambda_1'}$ occurs in $\Theta(\pi_{\Lambda_1})$.
\end{itemize}
In other words, we have the following commutative diagram
\[
\xymatrix{
\prll\in \CE(\sp_{2n})\ar[rr]^\Theta  \ar[d]_{\CL_{\rm can}} &&  \prllp\in  \CE(\o_{2n'}^{\epsilon'}) \ar[d]^{\CL_{\rm can}} \\
\rho\otimes\pi_{\Lambda_1}\otimes\pi_{\Lambda_{-1}}\ar[rr]^{\rm{iso}\otimes\Theta\otimes\rm{id}} && \rho\otimes\pi_{\Lambda_1'}\otimes\pi_{\Lambda_{-1}}.
}
\]

    \end{enumerate}
\end{theorem}

\begin{proof}
By our discussion in section \ref{sec3.2}, we know that $\clc$ is uniquely determined by condition (1) and its behavior on basic representations.

Consider conditions (2) and (3).
By Proposition \ref{sp-1} and Proposition \ref{-1}, we have
\[
\iota'=\prll(-I)\epsilon(-1)^{n}=\iota_\rho\pd(\Lambda_{-1})\ee^{n-\rm{rank}(\Lambda_{-1})}.
\]
Note that $n=|\rho|+\rm{rank}(\Lambda_{1})+\rm{rank}(\Lambda_{-1})$. Then
\[
\iota'=\prll(-I)\epsilon(-1)^{n}=\iota_\rho\pd(\Lambda_{-1})\ee^{|\rho|+\rm{rank}(\Lambda_{1})}=\wi_{\rho}\pd(\Lambda_{-1})\ee^{\rm{rank}(\Lambda_{1})}.
\]
Theorem \ref{p1} and Theorem \ref{p2} state that the requirement of $\rho$ holds, and the symbol part requirement holds up to the transpose.
If a symbol equals its transpose, then there is nothing left to prove. We now assume that $\Lambda_{\pm 1}$ and $\Lambda'_{\pm 1}$ are not equal to their transpose, respectively. As discussed in Section \ref{sec2}, we can find basic representations $\sigma$ and $\sigma'$ such that $\pi \in \CE(G_n, \sigma)$ and $\pi' \in \CE(G'_{n'}, \sigma')$. Using the same argument in the proof of \cite[Th\'eor\`em 3.7]{AMR} and \cite[Theorem 4.1]{LW1}, we conclude that $\pi'$ occurs in $\Theta(\pi)$ implies that $\sigma'$ occurs in $\Theta(\sigma)$. In other words, the theta lifting  the parabolic induction are compatible. It's important to note that we are dealing with Ma-Qiu-Zou's Weil representation, not Gerardin's. Therefore, we should refer to \cite[Proposition 2.1]{MQZ} instead of \cite[Chap. 3, IV th.5]{MVW} to prove this compatibility. According to this, we only need to demonstrate that there exists a  $\clc$ that satisfies the conditions (2) and (3) for basic representations. The proof of the basic case is provided in section \ref{sec7.3} and section \ref{sec7.3}.
The uniqueness of $\clc$ on basic representations follows from the fact that we can resolve the ambiguity in the Lusztig correspondence using the first occurrence of theta lifting (see section \ref{sec4.5}).

\end{proof}

\begin{remark}\label{rmk3}
The only place in the proof of Theorem \ref{main} we use the assumption of $q$ is when we refer to the conclusions in section \ref{sec6}.
As in Remark \ref{rmk2}, Theorem \ref{main} apply to quadratic unipotent representations when $q$ is odd. In other words, there exists a unique Lusztig correspondence $\clc$ on quadratic unipotent representations satisfying the above conditions.
\end{remark}

\subsection{Strategy for the proof of the basic representation Theorem}  \label{sec66}
We have already established that there is commutativity between theta correspondence and the first descent through a see-saw diagram. It is important to note that we aim to create a choice ${\CL^\rm{can}}$ of Lusztig that is commutative with theta correspondence. Assuming that such a standard choice ${\CL^\rm{can}}$ exists, we are looking to identify functors $\CD^*_\pm$ so that the diagram below is commutative.
 \begin{equation}\label{cm2}
\xymatrix{
&\pi^*\in\CE_{\rm{bas}}(C_{G_n^*}(s))\ar[rrr]^{\Theta^*_{n,n'}} \ar[dd]_{\CD_\epsilon^*}& &&\pi^{\prime *}\in\CE_{\rm{bas}}(C_{G_{n'}^{\prime *}}(s')) \ar[dd]^{\CD^*_\epsilon}\\
\pi\in \CE_{\rm{bas}}(G_{n})\ar[ur]^{\CL^\rm{can}}\ar[dd]_{\CD_{\ell,\epsilon}}\ar[rrr]^{\Theta_{n,n'}}&&&\pi'\in \CE_{\rm{bas}}(G'_{n'})\ar[ur]^{\CL^\rm{can}} \ar[dd]^{\CD_{\ell',\epsilon}}& \\
&\pi^{\star*}\in\CE_{\rm{bas}}(C_{G_m^*}(s^\star))\ar[rrr]^{\Theta^*_{m,m'}}& &&\pi^{\star\prime*}\in\CE_{\rm{bas}}(C_{G_{m'}^{*\prime}}(s^{\star\prime}))\\
\pi^\star\in \CE_{\rm{bas}}(G_{m})\ar[ur]^{\CL^\rm{can}}\ar[rrr]^{\Theta_{m,m'}}&&&\pi^{\star\prime}\in \CE_{\rm{bas}}(G'_{m'})\ar[ur]^{\CL^\rm{can}}&\\
},
\end{equation}
where
\begin{itemize}
\item $\epsilon\in \{\pm\}$;
\item $m=n-\ell$ (resp. $m'=n'-\ell'$) where $\ell$ (resp. $\ell'$) is the first descent index of $\pi$ (resp. $\pi'$);
\item $\CD_{\ell,\epsilon}$ (resp. $\CD_{\ell',\epsilon}$ ) is $\CD^{\rm{FJ}}_{\ell,\epsilon}$ or $\CD^{\rm{B}}_{\ell,\epsilon}$, depends on $G_n$ (resp. $G'_{n'}$);
\item $\sigma^*$ is the image of $\sigma$ under $\CL^{\rm{can}}$ with $\sigma\in \{\pi, \pi',\pi^\star,\pi^{\star\prime}\}$;
\item $n'$ (resp. $m'$) is the (theta) first occurrence index of $\pi$ (resp. $\pi^\star$);
\item If dual pair $(G_n,G'_{n'})$ contains a even orthogonal group, then
\[
\Theta_{n,n'}^* :\pi^*  =\rho\otimes\pi_{\Lambda_1}\otimes \pi_{\Lambda_{-1}}\to \pi^{\prime *}   =\rho\otimes\pi_{\Lambda'_1}\otimes \pi_{\Lambda_{-1}},
\]
Here $\pi_{\Lambda'_1}$ is the first theta lifting of $\pi_{\Lambda_1}$;
\item If dual pair $(G_n,G'_{n'})=(\sp_{2n},\o^\epsilon_{2n'+1})$, then
\[
\Theta_{n,n'}^* :\pi^*  =\rho\otimes\pi_{\Lambda_1}\otimes \pi_{\Lambda_{-1}}\to \pi^{\prime *}   =\rho^-\otimes \pi_{\Lambda_{-1}}\otimes\pi_{\Lambda'_1}\otimes \wi_{\rho}\pd(\Lambda_{-1})\ee^{\rm{rank}(\Lambda_{1})}.
\]
Here $\pi_{\Lambda'_1}$ is the first theta lifting of
\[
\left\{
\begin{array}{ll}
\pi_{\Lambda_1} &\textrm{ if }\ee\epsilon\epsilon_\rho=+;\\
\sgn\pi_{\Lambda_1} &\textrm{ if }\ee\epsilon\epsilon_\rho=-;\\
\end{array}
\right.
\]
\item If dual pair $(G_n,G'_{n'})=(\o^\epsilon_{2n+1},\sp_{2n'})$, then
\[
\Theta_{n,n'}^* :\pi^*  =\rho\otimes\pi_{\Lambda_1}\otimes \pi_{\Lambda_{-1}}\otimes\iota\to \pi^{\prime *}   =\rho^-\otimes \pi_{\Lambda_{-1}}\otimes\pi_{\Lambda'_1}.
\]
Here $\pi_{\Lambda'_1}$
\[
\pi_{\Lambda'_1}=\left\{
\begin{array}{ll}
\pi_{\Lambda'_1}' &\textrm{ if }\epsilon\epsilon_\rho=+;\\
\sgn\pi_{\Lambda'_1}' &\textrm{ if }\epsilon\epsilon_\rho=-,\\
\end{array}
\right.
\]
and $\pi_{\Lambda'_1}'$ is the first theta lifting of $\pi_{\Lambda_1} $;
\item The map $\Theta_{m,m'}$ is defined in a similar way.

\end{itemize}

In order to construct $\CL^{\rm{can}}$, we follow an inductive approach using $\CD^*_{\epsilon}$. We assume that the canonical Lusztig correspondence $\CL^{\rm{can}}$ in the bottom side of \eqref{cm2} and functors $\CD^*_{\epsilon}$ have already been constructed.
The commutative diagram at the front side of \eqref{cm2}, also known as the theta-descent commutative diagram, has been proven to be commutative in \eqref{cm1}. To ensure the entire diagram \eqref{cm2} is commutative, we need to find the unique Lusztig correspondences $\CL^{\rm{can}}$ on $G_n$ and $G'_{n'}$ to make the left, right and top diagrams commutative, which will automatically result in the back side of \eqref{cm2} being commutative as well.

We set $\CD_{\epsilon}^*$ as follows. Recall that
\begin{equation}
 \begin{aligned}
& C_{G_n^*}(s)
=\,&  (G_n)^*_{[ \ne \pm 1]}(s)\times (G_n)^*_{[ 1]}(s)\times (G_n)^*_{[ -1]}(s)
\times \begin{cases}
 \{\pm\},&\textrm{if }G=\o_{2n+1};\\
  1,&\textrm{otherwise}.
\end{cases}
 \end{aligned}
\end{equation}

Assume that $G_n=\sp_{2n}$, and $\rho\otimes\pi_{\Lambda_1}\otimes\pi_{\Lambda_{-1}}\in \CE(C_{G_n^*}(s))$. Then we set
\[
 \begin{aligned}
\CD^*_\epsilon(\rho\otimes\pi_{\Lambda_1}\otimes\pi_{\Lambda_{-1}})=
\begin{cases}
\rho^\star\otimes\pi_{\Lambda^\star_1}\otimes\pi_{\Lambda^\star_{-1}} & \textrm{ if }\epsilon=-\pd(\Lambda_1)\sd(\Lambda_{-1})\epsilon_{\rho}\epsilon_{\rho^\star};\\
0& \textrm{ otherwise, }
\end{cases}
 \end{aligned}
\]
where
\begin{itemize}
\item $\rho^\star=\rho^-_1$;
\item The symbol $\Lambda_1^\star$ is obtained by (1) removing the maximal number in $\Lambda_{-1}$, (2) swapping the arrays of $\Lambda_{-1}$ if $\sd(\Lambda_{-1})\pd(\Lambda_{-1})=+$;
    \item The symbol $\Lambda^\star_{-1}$ is obtained by (1) removing the maximal number in $\Lambda_{1}$, (2) swapping the arrays of $\Lambda_{1}$ if $\ee\sd(\Lambda_{-1})\pd(\Lambda_{-1})=+$, i.e. we have $\sd(\Lambda^\star_{-1})=-\pd(\Lambda_{1})\ee\sd(\Lambda_{-1})\pd(\Lambda_{-1})$.
\end{itemize}

Assume that $G_n=\o^\epsilon_{2n}$, and $\rho\otimes\pi_{\Lambda_1}\otimes\pi_{\Lambda_{-1}}\in \CE(C_{G_n^*}(s))$. Then we set
\[
\CD^*_{\epsilon'}(\rho\otimes\pi_{\Lambda_1}\otimes\pi_{\Lambda_1})=
\begin{cases}
\rho^\star\otimes\pi_{\Lambda^\star_1}\otimes\pi_{\Lambda^\star_{-1}}\otimes\iota^\star & \textrm{ if }\epsilon'=\sd(\Lambda_1)\sd(\Lambda_{-1})\epsilon_{\rho}\epsilon_{\rho^\star};\\
0& \textrm{ otherwise, }
\end{cases}
\]
where
\begin{itemize}
\item $\rho^\star=\rho_1$;
\item The symbol $\Lambda_1^\star$ is obtained by (1) removing the maximal number in $\Lambda_{1}$, (2) swapping the arrays of $\Lambda_{1}$ if $\rm{sd}(\Lambda_{1})\pd(\Lambda_{1})=+$;
\item The symbol $\Lambda_{-1}^\star$ is obtained by (1) removing the maximal number in $\Lambda_{-1}$, (2) swapping the arrays of $\Lambda_{-1}$ if $\rm{sd}(\Lambda_{-1})\pd(\Lambda_{-1})=+$;
    \item $\iota^\star=-\wi_{\rho^\star}\rm{sd}(\Lambda_1)$.
\end{itemize}
Assume that $G_n=\o^\epsilon_{2n+1}$, and $\rho\otimes\pi_{\Lambda_1}\otimes\pi_{\Lambda_{-1}}\otimes\iota\in \CE(C_{G_n^*}(s))$. Then we set
\[
\CD^*_{\epsilon'}(\rho\otimes\pi_{\Lambda_1}\otimes\pi_{\Lambda_1}\otimes\iota)=
\begin{cases}
\rho^\star\otimes\pi_{\Lambda^\star_1}\otimes\pi_{\Lambda^\star_{-1}} & \textrm{ if }\epsilon'=\pd(\Lambda_1)\pd(\Lambda_{-1})\epsilon\epsilon_{\rho^\star};\\
0& \textrm{ otherwise, }
\end{cases}
\]
where
\begin{itemize}
\item $\rho^\star=\rho_1$;
\item The symbol $\Lambda_1^\star$ is obtained by (1) removing the maximal numbet in $\Lambda_{1}$, (2) swapping the arrays of $\Lambda_{1}$ if $\iota\wi_\rho\rm{pd}(\Lambda_{1})=-$, i.e. $\sd(\Lambda_1^\star)=\iota\wi_\rho$;
\item The symbol $\Lambda_{-1}^\star$ is obtained by (1) removing the maximal numbet in $\Lambda_{-1}$, (2) swapping the arrays of $\Lambda_{-1}$ if $\iota\wi_\rho\epsilon\epsilon_{\rho}\rm{pd}(\Lambda_{1})\ee=-$, i.e. we have $\sd(\Lambda_{-1}^\star)=\iota\wi_\rho\epsilon\epsilon_{\rho}\rm{pd}(\Lambda_{1})\ee\pd(\Lambda_{-1})$.
\end{itemize}

It is worth noting that although the construction of $\CD_{\epsilon}^*$ may seem artificial at first glance, there is a deeper reason behind it. Specifically, if the canonical Lusztig correspondence $\CL^{\rm{can}}$ exists and \eqref{cm2} is commutative, then $\CD_{\epsilon}^*$ must be of the same form as we have constructed it. To elaborate further, the commutativity of the left and right side of equation \eqref{cm2} implies that the functor $\CD_{\epsilon}^*$ must also possess the same characteristic as that of $\CD_{\ell,\epsilon}$, which we previously established in Section \ref{sec6.5}. This is precisely why the construction of $\CD_{\epsilon}^*$ involves signs like $\pd(\Lambda_{\pm})$ and $\sd(\Lambda_{\pm})$.

\subsection{Proof of Theorem \ref{main} for basic representations}\label{sec7.3}

 We prove it by induction on $n$. We will only prove it on $\sp_{2n}\fq$. The proof for orthogonal groups side is similar and will be left to readers. Assume that we have already construct $\cal{L}^{\rm{can}}$ on basic representations of $G_n$ for $n<k$, and  $\cal{L}^{\rm{can}}$ satisfies the theta conditions  (2) and (3) in Theorem \ref{main} for basic representations of dual pair $(G_n,G_m)$ with $m<k$.

 {\bf Construction of $\clc$ on $\pi$.}
 Assume that $\pi\in\CE_{\rm{bas}}(\sp_{2k})$, and $\pi=\prll$. Recall that different choices of Lusztig correspondence differ only on the $\Lambda_{-1}$ part. If $\Lambda_{-1}$ is the trivial symbol, then there is nothing to prove. So we assume that $\Lambda_{-1}$ is not trivial.
 Consider the first descent of $\pi$. Let $\ell$ be the first descent index of $\pi$.
By Proposition \ref{fd1}, we know that the first descents of basic representations are still basic. Assume that
 \[
 \CD^{\rm{FJ}}_{\ell,\epsilon}(\pi)=\pi^\star\in \CE_{\rm{cus}}(\sp_{2k^\star}),
 \]
 where $k^\star=k-\ell$. We know that the subscript $\epsilon$ in $ \CD^{\rm{FJ}}_{\ell,\epsilon}$ and the representation $ \pi^\star $ are determined by $\pi$, not the choice of Lusztig correspondence. Therefore we will construct $\clc$ on $\pi$ by them.
As mentioned in section \ref{main}, we define $\clc$ to be the Lusztig correspondence such that the left side of \eqref{cm2} is commutative,
   \[
\xymatrix{
\pi \ar[rr]^{\CD^{\rm{FJ}}_{\ell,\epsilon}} \ar[d]_{\clc} && \prlls \ar[d]^{\clc} \\
\rho\otimes\pi_{\Lambda_1}\otimes\pi_{\Lambda^{\rm{can}}_{-1}}\ar[rr]^{\CD^*_\epsilon } && \quad \rho^{\star}\otimes\pi_{\Lambda_1^{\star} }\otimes\pi_{\Lambda_{-1}^{\star}}.
}
\]
Assume that $\clc(\pi)=\rho\otimes\pi_{\Lambda_{1}}\otimes\pi_{\Lambda^\rm{can}_{-1}}$, where $\Lambda^\rm{can}_{-1}\in\{\Lambda_{-1},\Lambda^t_{-1}\}$.
Since $\CD^*_\epsilon(\clc(\pi))\ne  0 $, we have
\begin{equation}\label{eepp0}
 \sd(\Lambda^\rm{can}_{-1})=-\epsilon\epsilon_{\rho}\epsilon_{\rho^\star}\pd(\Lambda_1),
\end{equation}
which determines the symbol $\Lambda^\rm{can}_{-1}$, and so is $\clc$.
By abuse of notations, from now on, we redefine $\Lambda_{-1}$ to be $\Lambda^\rm{can}_{-1}$.

On the other hand, By our induction assumption, we have already construct $\cal{L}^{\rm{can}}$ on $\pi^\star$.  Write
\[
\clc(\pi^\star)=\rho^\star\otimes\pi_{\Lambda_{1}^\star}\otimes\pi_{\Lambda_{-1}^\star}.
\]
By our construction of $\CD^*_\epsilon$, if $\Lambda_{-1}^\star\ne\left(\begin{smallmatrix}
-\\
-
\end{smallmatrix}\right) $, the sign $\sd(\Lambda_{-1}^\star)$ is unique determined by $\pi$:
\begin{equation}\label{eepp}
\sd(\Lambda^\star_{-1})=-\pd(\Lambda_{1})\ee\sd(\Lambda_{-1})\pd(\Lambda_{-1})=\epsilon\epsilon_{\rho}\epsilon_{\rho^\star}\pd(\Lambda_{-1})\ee,
\end{equation}
where the last equation follows from \eqref{eepp0}.
Therefore, to ensure that our construction of $\clc$ on $\pi$ does not contradict $\clc$ on $\pi^\star$, we have to check \eqref{eepp}, when $\Lambda_{-1}^\star\ne\left(\begin{smallmatrix}
-\\
-
\end{smallmatrix}\right) $.

 {\bf Proof of \eqref{eepp}.} From now on, let us forget \eqref{eepp0}, and try to obtain $\epsilon$ and prove \eqref{eepp}. Our strategy is sending $\pi$ and $\pi^\star$ to smaller orthogonal groups by Theta correspondence.
 Let $k^\pm$ be the first theta occurrence indexes of $\pi$ with respect to the Witt towers ${\bf{{O}}}_{\rm{odd}}^\pm$. Recall that $\Lambda_{-1}$ is not trivial. Then by Proposition \ref{cus1}, there exists $k^{\epsilon'}$ such that $k^{\epsilon'}< k$, and $\Theta^{\epsilon'}_{k,k^{\epsilon'}}(\pi)$ is an irreducible basic representation, and we denote it by $\pi'=\prllps$.
 By Proposition \ref{cus1}, we also know that $\Theta^{\epsilon'}_{k^{\epsilon'},k}(\pi')=\pi$. Similarly, let $k^{\star\pm}$ be the first theta occurrence index of $\ps$ with respect to the Witt tower ${\bf{{O}}}_{\rm{even}}^\pm$, and denote the smaller one by $k^{\star\prime}$, and the correspondence group by $\o_{2k^{\star\prime}}^{\epsilon^{\star\prime}}\fq$. Write $\Theta^{\epsilon^{\star\prime}}_{k^\star,k^{\star\prime}}(\ps)=\prllsp$.
  By our induction assumption, we have the following theta commutative diagram for $\ps$ (the bottom diagram of \eqref{cm2}):
  \[
\xymatrix{
\ps\ar[rr]^\Theta \ar[d]_{\clc} && \prllsp \ar[d]^{\clc} \\
\rho^\star\otimes\pi_{\Lambda^\star_1}\otimes\pi_{\Lambda^\star_{-1}} \ar[rr]^{\Theta^*_{k^\star,k^{\epsilon^{\star\prime}}} } && \quad \rho^{\star\prime}\otimes\pi_{\Lambda_1^{\star\prime}}\otimes\pi_{\Lambda_{-1}^{\star\prime}}.
}
\]
Since the Theta correspondence in the above diagram goes down, by Proposition \ref{cus1} and Theorem \ref{p}, we have
 \begin{itemize}
 \item $\rho^{\star\prime}=\rho^{\star}$;
  \item $\Lambda_{-1}^{\star\prime}=\Lambda_{-1}^{\star}$;
    \item $|\rm{def}(\Lambda_{1}^{\star\prime})|=|\rm{def}(\Lambda_{1}^{\star})|-1$, and $\sd(\Lambda_{1}^{\star\prime})=-\pd(\Lambda_{1}^{\star})$.
\end{itemize}
Then the superscript $\epsilon^{\star\prime}$ on $\o_{2k^{\star\prime}}^{\epsilon^{\star\prime}}$  equals to $\epsilon_{\rho^\star}\pd(\Lambda_1^\star)\pd(\Lambda_{-1}^\star)$.

By \eqref{cm1}, we have $
 \CD^{\rm{B}}_{\ell-1,\epsilon}(\pi')=\prllsp.
 $
  Applying our induction assumption on $\pi'$, we have the following descent commutative diagram (the right diagram of \eqref{cm2}):
 \[
\xymatrix{
\pi' \ar[rr]^{\CD^{\rm{B}}_{\ell-1,\epsilon}} \ar[d]_{\clc} && \prllsp \ar[d]^{\clc} \\
\rho'\otimes\pi_{\Lambda_1'}\otimes\pi_{\Lambda_{-1}'}\otimes\iota' \ar[rr]^{\CD^*_\epsilon } && \quad \rho^{\star\prime}\otimes\pi_{\Lambda_1^{\star\prime} }\otimes\pi_{\Lambda_{-1}^{\star\prime}}.
}
\]
 According to the construction of $\CD^*_\epsilon$, we have
  \begin{itemize}
 \item $|\rm{def}(\Lambda_{1}')|=|\rm{def}(\Lambda_{1}^{\star\prime})|+1$;
  \item $|\rm{def}(\Lambda_{-1}')|=|\rm{def}(\Lambda_{-1}^{\star\prime})|+1$;
    \item $ \CD^*_\epsilon(\pi')\ne 0$ implies that
      \begin{equation}\label{ep}
      \epsilon=\pd(\Lambda'_1)\pd(\Lambda'_{-1})\epsilon'\epsilon_{\rho'}
      \end{equation};
\item $\epsilon'=\epsilon_{\rho'}\wi_{\rho'}\iota'\pd(\Lambda_{1}')\pd(\Lambda_{-1}')\sd(\Lambda_{-1}^{\star\prime})\ee=\epsilon_{\rho}\wi_\rho\iota'\pd(\Lambda_{1}')\pd(\Lambda_{-1}')\sd(\Lambda_{-1}^{\star})\ee$.
\end{itemize}
 The last equation follows from our assumption that $\sd(\Lambda_{-1}^{\star\prime})=\sd(\Lambda_{-1}^{\star})\ne 0$ and the fact that $\pi'$ occurs in $\Theta^{\epsilon'}_{k,k^{\epsilon'}}(\pi)$ implies $\epsilon_\rho=\epsilon_{\rho'}$ and $\wi_\rho=\wi_{\rho'}$. Consider the theta correspondence from $\pi$ to $\pi'$, by Corollary \ref{oddg} (i), we have
\begin{equation}\label{iota}
\iota'=-\wi_{\rho'}\pd(\Lambda'_{1}).
\end{equation}
Then we have
 \begin{equation}\label{ep2}
\epsilon'=-\epsilon_{\rho}\pd(\Lambda_{-1}')\sd(\Lambda_{-1}^{\star})\ee.
   \end{equation}
 Combing \eqref{ep} and \eqref{ep2}, we have
     \begin{equation}\label{ep3}
 \epsilon=\pd(\Lambda'_1)\pd(\Lambda'_{-1})\epsilon'\epsilon_{\rho'}=-\pd(\Lambda'_1)\epsilon_{\rho}\epsilon_{\rho'}\sd(\Lambda_{-1}^{\star})\ee.
      \end{equation}

 By Proposition \ref{cus1} and Theorem \ref{fd2} (ii), we have
  \begin{itemize}
 \item $\pd(\Lambda_{1}^\star)=-\pd(\Lambda_{-1})$;
  \item  $\pd(\Lambda_{-1}^\star)=\pd(\Lambda_{1})$;
   \item $\pd(\Lambda_{1}^\prime)=-\pd(\Lambda_{-1})$;
    \item $\pd(\Lambda_{-1}')=\pd(\Lambda_{1})$.
\end{itemize}
 Then (\ref{ep}) becomes
 \begin{equation}\label{a1}
 \epsilon=\epsilon'\epsilon^{\star\prime}=\epsilon_{\rho}\epsilon_{\rho^\star}\pd(\Lambda_{-1})\sd(\Lambda_{-1}^{\star})\ee,
 \end{equation}
which completed the proof of \eqref{eepp}. Hence, our construction of $\clc$ on $\pi$ makes sense.

 {\bf Construction of $\clc$ on $[\pi]$.}
%
%

According the definition of $\clc$ on $\pi$, the Lusztig correspondence $\clc$ has to send $\pi^c$ to
$
\rho\otimes\pi_{\Lambda_1}\otimes\pi_{{}^t\Lambda_{-1}}.
$
 Then by Theorem \ref{d2}, we have
 \begin{itemize}
\item $\CD^{\rm{FJ}}_{\ell,-\epsilon}(\pi)=0$;
\item $ \CD^{\rm{FJ}}_{\ell,\epsilon}(\pi^c)=0$;
\item $ \CD^{\rm{FJ}}_{\ell,-\epsilon}(\pi^c)=(\pi^\star)^c$.
\end{itemize}
 To check whether $\clc$ is well defined, we have to check the commutativity of the following diagram
 \[
\xymatrix{
\pi^c \ar[rr]^{\CD^{\rm{FJ}}_{\ell,-\epsilon}} \ar[d]_{\clc} && (\ps)^c \ar[d]^{\clc} \\
\rho\otimes\pi_{\Lambda_1}\otimes\pi_{{}^t\Lambda_{-1}} \ar[rr]^{\CD^*_{-\epsilon} } && \quad \rho^\star\otimes\pi_{\Lambda_1^\star}\otimes\pi_{{}^t\Lambda_{-1}^\star},
}
\]
which can be obtained from direct calculations, and will be left to the reader.

{\bf Compatibility with theta correspondence for odd orthogonal group.}
We have constructed that our choice of canonical Lusztig correspondence $\clc$ on $\sp_{2k}\fq$. Now, we are going to check the condition (2) in Theorem \ref{main}.
In other words, we need to check
\[
\iota'=\wi_\rho\pd(\Lambda_{-1})\ee^{\rm{rank}(\Lambda_{1})}.
\]
and the commutativity of the following diagram (the top diagram of \eqref{cm2})
 \[
\xymatrix{
\pi \ar[rr]^\Theta \ar[d]_{\clc} && \pi' \ar[d]^{\clc} \\
\rho\otimes\pi_{\Lambda_1}\otimes\pi_{\Lambda_{-1}} \ar[rr]^{\Theta_{k,k^{\epsilon'}}^*} && \quad \rho^{\prime}\otimes\pi_{\Lambda_{-1}^{\prime}}\otimes\pi_{\Lambda_{1}^{\prime}}\otimes\iota^\prime.
}
\]

We have already calculated $\iota'$ by Corollary \ref{oddg} (i) (see \eqref{iota}).
Recall that $\wi_{\rho}=\wi_{\rho'}$ and $\pd(\Lambda_{1}^\prime)=-\pd(\Lambda_{-1})$. We have
\[
\iota'=\wi_{\rho}\pd(\Lambda_{-1}).
\]
Since $\pi$ is basic, we conclude that $\Lambda_1$ corresponds to a unipotent cuspidal representation and $\rm{rank}(\Lambda_1)$ is of form $h(h+1)$. So
$
\iota'=\wi_\rho\pd(\Lambda_{-1})\ee^{\rm{rank}(\Lambda_{1})}.
$

The rest is to check
 \[
 \Theta_{\epsilon'\epsilon_\rho}(\pi_{\Lambda_{-1}})=\pi_{\Lambda_1^\prime}.
 \]
It is equivalent to say that the theta correspondence $\Theta_{\epsilon'\epsilon_\rho}(\pi_{\Lambda_{-1}})$ goes down. By (\ref{ep2}), we have
\[
\epsilon'\epsilon_\rho=-\ee\pd(\Lambda_{1})\sd(\Lambda_{-1}^{\star})=\pd(\Lambda_{-1})\sd(\Lambda_{-1}).
\]
According to Theorem \ref{even}, $\Theta_{\epsilon'\epsilon_\rho}(\pi_{\Lambda_{-1}})$ goes down.

{\bf Compatibility with theta correspondence for even orthogonal group.}

We have proved that our choice of canonical Lusztig correspondence $\clc$ satisfies the condition (2) in Theorem \ref{main}. Next, we are going to check the condition (3).
In order to reduce notations, we still use superscript $\prime$ and $\star\prime$ in this section. However, they have different meaning. Readers should not confuse them.

Let $k^\pm$ be the first theta occurrence indexes with respect to the Witt towers ${\bf{{O}}}_{\rm{even}}^\pm$. By Proposition \ref{cus1}, there exists $k^{\epsilon'}$ such that $k^{\epsilon'}\le k$. Denote the irreducible basic representation $\Theta^{\epsilon'}_{k,k^{\epsilon'}}(\pi)$ by $\pi'=\prllp$. And we have $\Theta^{\epsilon'}_{k^{\epsilon'},k}(\pi')=\pi$. Similarly, let $k^{\star\pm}$ be the first theta occurrence index of $\ps$ with respect to the Witt tower ${\bf{{O}}}_{\rm{odd}}^\pm$, and denote the smaller one by $k^{\star\prime}$, and denote the correspondence group by $\o_{2k^{\star\prime}+1}^{\epsilon^{\star\prime}}\fq$. Let $\Theta^{\epsilon^{\star\prime}}_{k^\star,k^{\star\prime}}(\pi^\star)=\pi^{\star\prime}=\prllspp$.

Consider the following diagram:
 \begin{equation}\label{diageven}
\xymatrix{
&\rho\otimes\pi_{\Lambda_1}\otimes\pi_{\Lambda_{-1}}\ar[rrr]^{\Theta_{k,k^{\e}}^*} \ar[dd]_{\CD_\epsilon^*}& &&\rho'\otimes\pi_{\Lambda'_1}\otimes\pi_{\Lambda'_{-1}} \ar[dd]^{\CD^*_\epsilon}\\
\pi \in \E_{\rm{cus}}(\sp_{2k})\ar[ur]^{\CL^\rm{can}}\ar[dd]_{\CD_{\ell,\epsilon}}\ar[rrr]^{\Theta_{k,k^{\e}}}&&&\pi' \in \E_{\rm{cus}}(\o^\e_{2k^{\e}})\ar[ur]^{\CL^\rm{can}} \ar[dd]^{\CD_{\ell',\epsilon}}& \\
&\rho^\star\otimes\pi_{\Lambda^\star_1}\otimes\pi_{\Lambda^\star_{-1}} \ar[rrr]^{\Theta^*_{k^\star,k^{\e}}}& &&\rho^{\star\prime}\otimes\pi_{\Lambda^{\star\prime}_1}\otimes\pi_{\Lambda^{\star\prime}_{-1}}\otimes\iota^{\star\prime}\\
\pi^\star\in  \E_{\rm{cus}}(\sp_{2k^\star})\ar[ur]^{\CL^\rm{can}}\ar[rrr]^{\Theta_{k^\star,k^{\star\prime}}}&&&\pi^{\star\prime}\in \E_{\rm{cus}}(\o_{2k^{\star\prime}+1}^{\epsilon^{\star\prime}})\ar[ur]^{\CL^\rm{can}}&\\
},
\end{equation}
The left side of the above diagram has been determined by previous calculation. We now are going to check that the left side is compatible with the rest of the above diagram.

Assume that $k^{\epsilon'}=k$. We construct the canonical Lusztig correspondence $\clc$ on $\o_{2k^{\epsilon'}}^{\epsilon'}\fq$ to be the Lusztig correspondence such that the top diagram  is commutative. It is easy to check that this choice of canonical Lusztig correspondence $\clc$ makes the right diagram commutative. Asume that $k^{\epsilon'}<k$. Similar to the odd orthogonal group case, we first use the bottom theta diagram to get information of $\prllspp$. Then describe $\prllp$ by the right side  diagram. Finally, we check whether the top theta diagram is commutative.

From the bottom commutative diagram, as before, we know that
\begin{itemize}
\item $\rho^{\star\prime}=\rho^\star$;
\item $\Lambda_{-1}^{\star\prime}=\Lambda_1^\star$;
\item $\iota^{\star\prime}=\ps(-I)\epsilon(-1)^{k^\star}=\wi_{\rho^\star}\pd(\Lambda^\star_{-1})$.
\end{itemize}
Since the theta lifting in the bottom of diagram \eqref{diageven} goes down, the theta lifting $\Theta_{\epsilon^{\star\prime}\epsilon_{\rho^\star}}(\pi_{\Lambda_{-1}^\star})$ still goes down. Then by Theorem \ref{even}, we have
\[
\pd(\Lambda_{-1}^\star)=\left\{
\begin{array}{ll}
\epsilon^{\star\prime}\epsilon_{\rho^\star} & \textrm{ if }\sd(\Lambda_{-1}^\star)\ne \left(\begin{smallmatrix}
-\\
-
\end{smallmatrix}\right);\\
\epsilon^{\star\prime}\epsilon_{\rho^\star}\sd(\Lambda_{-1}^\star) & \textrm{ otherwise}.\\
\end{array}
\right.
\]
By our construction of $\CD^*_\epsilon$, we have
$
\pd(\Lambda_{-1}^\star)=\pd(\Lambda_1).
$
Since the theta lifting of $\pi$ goes down, by Theorem \ref{even}, we have
$
\pd(\Lambda_1)=\pd(\Lambda'_{1}).
$
Therefore
\[
\epsilon^{\star\prime}=\epsilon_{\rho^\star}\sd(\Lambda_{-1}^\star)\pd(\Lambda'_1).
\]
Consider the two descent diagrams and $\CD^*_\epsilon$. We have
\begin{itemize}
\item $\sd(\Lambda_1')=-\wi_{\rho^{\star\prime}}\iota^{\star\prime}=-\wi_{\rho^{\star\prime}}\wi_{\rho^\star}\pd(\Lambda^\star_{-1})=-\pd(\Lambda^\star_{-1})=-\pd(\Lambda_{1})$;
\item Note that $\epsilon=\ee\epsilon'\epsilon^{\star\prime}$ and $\epsilon'=\epsilon_{\rho'}\pd(\Lambda_1)\pd(\Lambda_{-1})$. By our construction of $\CD^*_\epsilon$, we know that $\CD_{\ell',\epsilon}(\pi')\ne 0$ implies that $\epsilon=\epsilon_{\rho^{\star\prime}}\epsilon_{\rho'}\sd(\Lambda_1)\sd(\Lambda_{-1})$. Hence, we have
\[
\begin{aligned}
\sd(\Lambda_{-1}')&=\ee\epsilon^{\star\prime}\epsilon_{\rho^\star}\sd(\Lambda_{1}')\pd(\Lambda'_{1})\pd(\Lambda'_{-1})\\
    =&\ee\sd(\Lambda_{1}')\pd(\Lambda'_{-1})\sd(\Lambda_{-1}^\star)=-\ee\pd(\Lambda_{1})\pd(\Lambda_{-1})\sd(\Lambda_{-1}^\star).
    \end{aligned}
    \]
\end{itemize}
If some of $\Lambda'_{\pm1}$ are trivial, we need not to check the corresponding part. So we assume that they are not trivial, that is, $\sd(\Lambda'_{\pm1})$ are always well defined.
Now we going to check that the above data makes the following diagram commutative
  \[
\xymatrix{
\pi \ar[rr]^\Theta \ar[d]_{\cal{L}} && \prllp \ar[d]^{\cal{L}} \\
\rho\otimes\pi_{\Lambda_{-1}}\otimes\pi_{\Lambda_1} \ar[rr]^{\id\otimes\Theta_{}\otimes\id } && \quad \rho^{\prime}\otimes\pi_{\Lambda_1^{\prime}}\otimes\pi_{\Lambda_{-1}^{\prime}},
}
\]
or equivalently checking
\begin{itemize}
\item $\sd(\Lambda_1')=-\pd(\Lambda_{1})$ \  (Since $\pi$ goes down in the above theta lifting);
\item $\sd(\Lambda_{-1}')=\sd(\Lambda_{-1}) $.
\end{itemize}
The first equation has already been proved. For the second one, by our construction of $\CD^*_\epsilon$, we have
\[
\sd(\Lambda_{-1}^\star)=-\ee\pd(\Lambda_{1})\pd(\Lambda_{-1})\sd(\Lambda_{-1}).
\]
Hence
 \[
\begin{aligned}
\sd(\Lambda_{-1}^\prime)=&-\ee\pd(\Lambda_{1})\pd(\Lambda_{-1})\sd(\Lambda_{-1}^\star)\\
=&-\ee\pd(\Lambda_{1})\pd(\Lambda_{-1})(-\ee\pd(\Lambda_{1})\pd(\Lambda_{-1})\sd(\Lambda_{-1}))=\sd(\Lambda_{-1}).
 \end{aligned}
\]

\begin{corollary}
Each facet of the diagram \eqref{cm2} is commutative.
\end{corollary}

\section{A finite field instance of  relative Langlands duality}\label{sec7}

\subsection{Calculation of Gan-Gross-Prasad problem}
Using the Lusztig correspondence described in Theorem \ref{main}, we can refine the conditions that are strongly relevant for a pair of representations. This refinement allows us to rephrase the main theorems in \cite{Wang1} and determine the multiplicity in the finite Gan-Gross-Prasad problem in a combinatorial manner.

\begin{theorem}\label{thm7.1}
\begin{enumerate}
\item Let $\pi=\prll\in\CE_{\rm{bas}}(\sp_{2n})$ and $\pi^\star=\prlls\in\CE_{\rm{bas}}(\sp_{2n^\star})$. Then $(\pi,\pi^\star)$ is $\epsilon$-strongly relevant  if and only if there  are $\ep,\ev\in\{\pm\}$ such that
\begin{itemize}
\item $\epsilon=\epsilon_\rho\epsilon_{\rho^\star}\ep\ev$;
\item ${\rm{def}}(\Lambda_1^\star)=-\ev\pd(\Lambda_{-1}){\rm{def}}(\Lambda_{-1})+\pd(\Lambda_{-1})$,
\item ${\rm{def}}(\Lambda_{-1}^\star)=-\ev\pd(\Lambda_{-1})\ee{\mathrm{def}}(\Lambda_1)-\ep\ev\pd(\Lambda_{-1})\ee$.
 \end{itemize}

 \item Let $\pi=\prll\in\CE_{\rm{bas}}(\o^\epsilon_{2n})$ and $\pi^\star=\prllss\in\CE_{\rm{bas}}(\o^{\epsilon^\star}_{2n^\star+1})$. Then $(\pi,\pi^\star)$ is strongly relevant  if and only if there  are $\ep,\ev\in\{\pm\}$ such that
\begin{itemize}
\item $\epsilon^\star=\epsilon_{\rho^\star}\ep\ev\pd(\Lambda_1)\pd(\Lambda_{-1})\ee$;
\item $\iota^\star=-\wi_{\rho^\star}\ep$;
\item ${\rm{def}}(\Lambda_1^\star)=-\ep\pd(\Lambda_{1}){\rm{def}}(\Lambda_{1})+\pd(\Lambda_{1})$,
\item ${\rm{def}}(\Lambda_{-1}^\star)=-\ev\pd(\Lambda_{-1}){\mathrm{def}}(\Lambda_{-1})+\pd(\Lambda_{-1})$.
 \end{itemize}
 \end{enumerate}
\end{theorem}
\begin{proof}
We only prove (i). The proof of (ii) is similar.

According to Corollary \ref{srsp}, we have
\begin{equation}\label{80}
|\rm{def}(\Lambda_1^\star)|=|\rm{def}(\Lambda_{-1})|\pm 1 \textrm{ and } |\rm{def}(\Lambda_1)|=|\rm{def}(\Lambda^\star_{-1})|\pm 1.
\end{equation}

Pick $\epsilon'\in\{\pm\}$ such that $\pi$ goes down with respect to the Witt tower ${\bf O}^\e_\rm{even}$, and consider its first theta lifting $\Theta(\pi)$. By Theorem \ref{even} and Theorem \ref{main} (iii), we have
\begin{itemize}
\item $\Theta(\pi)=\prllp$
\item $\rho'=\rho$;
  \item $\Theta(\pi_{\Lambda_1})=\pi_{\Lambda_1'}$ is the unipotent basic representation of $\o^{\epsilon'}_{2n'}\fq$ with  $\rm{def}(\Lambda_1')=-\rm{def}(\Lambda_{1})+\pd(\Lambda_1)$, $\epsilon'=\pd(\Lambda_1')$, and $n'=\left(\frac{\rm{def}(\Lambda_1')}{2}\right)^2$;
\item $\Lambda'_{-1}=\Lambda_{-1}$.
\end{itemize}
 Hence \begin{equation}\label{81}
\epsilon'=\epsilon_{\rho'}\pd(\Lambda_{1}')\pd(\Lambda'_{-1})=\epsilon_{\rho}\pd(\Lambda_{1})\pd(\Lambda_{-1}),
\end{equation}
and
$
\neven^\e(\pi)=\neven^{\epsilon'}(\pi_{\Lambda_1})=\frac{|\rm{def}(\Lambda_1)-1|}{2}.
$

Consider the first theta lifting $\Theta(\pi^\star)$ of $\pi^\star$ with respect to ${\bf O}^{\ee\cdot\epsilon\epsilon'}_\rm{odd}$. With the same theta argument, we have
\[
\nodd^{\ee\cdot\epsilon\epsilon'}(\pi^\star)=
\left\{
\begin{array}{ll}
0 & \textrm{ if }\rm{def}(\Lambda_{-1}^\star)=0;\\
\frac{|\rm{def}(\Lambda_{-1}^\star)|}{2} & \textrm{ if }\epsilon=\epsilon'\epsilon_{\rho^\star}\sd(\Lambda_{-1}^\star)\pd(\Lambda_{-1}^\star)\ee;\\
\frac{-|\rm{def}(\Lambda_{-1}^\star)|}{2} & \textrm{ if }\epsilon=-\epsilon'\epsilon_{\rho^\star}\sd(\Lambda_{-1}^\star)\pd(\Lambda_{-1}^\star)\ee.
\end{array}
\right.
\]
To make sure $(\pi,\pi^\star)$ $\epsilon$-strongly relevant, firstly, $(\pi,\pi^\star)$ has to be $\epsilon$-relevant. In other words, either $\neven^\e(\pi)=\nodd^{\ee\cdot\epsilon\epsilon'}(\pi^\star)-1$ or $\neven^\e(\pi)=\nodd^{\ee\cdot\epsilon\epsilon'}(\pi^\star)$.
Hence $\nodd^{\ee\cdot\epsilon\epsilon'}(\pi^\star)$ can not be negative. We will only prove the case  $\nodd^{\ee\cdot\epsilon\epsilon'}(\pi^\star)>0$. The proof for the case  $\nodd^{\ee\cdot\epsilon\epsilon'}(\pi^\star)=0$ is similar. In the positive case, we have $\epsilon=\epsilon'\epsilon_{\rho^\star}\sd(\Lambda_{-1}^\star)\pd(\Lambda_{-1}^\star)\ee$. Using \eqref{81},
\begin{equation}\label{82}
\epsilon=\epsilon_{\rho}\epsilon_{\rho^\star}\pd(\Lambda_{1})\pd(\Lambda_{-1})\sd(\Lambda_{-1}^\star)\pd(\Lambda_{-1}^\star)\ee.
\end{equation}
On the other hand, the pair $(\pi^\star,\pi)$ have to be $\ee\epsilon$-relevant. With the similar argument, we have
\begin{equation}\label{83}
\epsilon=\epsilon_{\rho}\epsilon_{\rho^\star}\sd(\Lambda_{-1})\pd(\Lambda_{-1})\pd(\Lambda_{1}^\star)\pd(\Lambda_{-1}^\star).
\end{equation}
Combining \eqref{82} and \eqref{83}, we have
\begin{equation}\label{84}
\sd(\Lambda_{-1}^\star)=\pd(\Lambda_{1})\sd(\Lambda_{-1})\pd(\Lambda_{1}^\star)\ee.
\end{equation}
Set
\begin{equation}\label{85}
\ep=-\pd(\Lambda_{-1}^\star),
\textrm{ and }\ev=-\pd(\Lambda_{1}^\star)\pd(\Lambda_{-1})\sd(\Lambda_{-1}).
\end{equation}
Substitute \eqref{84} and \eqref{85} into \eqref{82}, we have
\[
\epsilon=\epsilon_{\rho}\epsilon_{\rho^\star}\ep\ev.
\]
The defect of $\Lambda_1^\star$ and $\Lambda^\star_{-1}$ is obtained by \eqref{80}, \eqref{84}, and \eqref{85}.
\end{proof}

Theorem \ref{thm7.1} allows us to rewrite \cite[Theorem 1.4]{Wang1} as follows.

\begin{theorem}\label{ggpmain}

\begin{enumerate}
\item Let $\prll\in \CE(\sp_{2n})$ and $\prlls\in \CE(\sp_{2n^\star})$, where $n\geqslant n^\star$. Assume that $\pi_{\Lambda_{-1}}\in \CE(\o^{\epsilon_{-1}}_{2m})$,
$
\rho=\prod\pi[a]
$
and
$
\rho^\star=\prod\pi^\star[a],
$
where $\pi[a]$ (resp. $\pi^\star[a]$) is the irreducible unipotent representation corresponding to the partition $\lambda[a]$ (resp. $\lambda^\star[a]$).
\[
m_\epsilon(\prll,\prlls)\ne 0
\]
if and only if $\prll$ and $\prlls$ satisfy the following conditions:
\begin{itemize}
\item[(i)] there  are $\ep,\ev\in\{\pm\}$ such that
\[
\begin{aligned}
& \epsilon=\ep\ev\epsilon_{\rho}\epsilon_{\rho^\star},\\
& \Upsilon(\Lambda_{-1}^\star)^{-\ep\ev{\epsilon_{-1}\ee} }\preccurlyeq \Upsilon(\Lambda_1)^{\ep }, \quad \Upsilon(\Lambda_1)^{-\ep } \preccurlyeq \Upsilon(\Lambda^\star_{-1})^{\ep\ev{\epsilon_{-1}}\ee}, \\
& \Upsilon(\Lambda_{1}^\star)^{\epsilon_{-1}} \preccurlyeq \Upsilon(\Lambda_{-1})^{-\ev}, \quad \Upsilon(\Lambda_{-1})^{\ev } \preccurlyeq \Upsilon(\Lambda^\star_{1})^{-\epsilon_{-1}}, \\
& {\mathrm{def}}(\Lambda_1^\star)=-\epsilon_{-1}\ev{\mathrm{def}}(\Lambda_{-1})+\epsilon_{-1}, \\
& {\mathrm{def}}(\Lambda_{-1}^\star)=-\epsilon_{-1}\ev\ee{\mathrm{def}}(\Lambda_1)-\ep\ev\epsilon_{-1}\ee.
\end{aligned}
\]

\item[(ii)] $
\begin{cases}
 \textrm{$\lambda[a]$ and $\lambda'[-a]$ are 2-transverse}, &\textrm{ if $\#[a]$ is odd};\\
   \textrm{$\lambda[a]$ and $\lambda'[-a]$ are close}, &\textrm{ if $\#[a]$ is even}.
\end{cases}
$
\end{itemize}

\item Let $\prll\in \CE(\o_{2n}^\epsilon)$ and $\prllss\in \CE(\o^{\epsilon^\star}_{2n^\star+1})$, where $n> n^\star$. Assume that $\pi_{\Lambda_{\pm1}}\in \CE(\o^{\epsilon_{\pm}}_{2m_\pm})$,
$
\rho=\prod\pi[a]
$
and
$
\rho^\star=\prod\pi^\star[a],
$
where $\pi[a]$ (resp. $\pi^\star[a]$) is the irreducible unipotent representation corresponding to the partition $\lambda[a]$ (resp. $\lambda^\star[a]$).
\[
m_\epsilon(\prll,\prllss)\ne 0
\]
if and only if $\prll$ and $\prllss$ satisfy the following conditions:
\begin{itemize}
\item[(i)]
there  are $\ep,\ev\in\{\pm\}$ such that
\[
\begin{aligned}
& \epsilon=\ep\ev\epsilon_{\rho}\epsilon_{\rho^\star},\\
& \Upsilon(\Lambda_{1}^\star)^{\epsilon_+ }\preccurlyeq \Upsilon(\Lambda_1)^{-\ep}, \quad \Upsilon(\Lambda_1)^{\ep } \preccurlyeq \Upsilon(\Lambda^\star_{1})^{-\epsilon_+}, \\
& \Upsilon(\Lambda_{-1}^\star)^{\epsilon_-} \preccurlyeq \Upsilon(\Lambda_{-1})^{-\ev}, \quad \Upsilon(\Lambda_{-1})^{\ev } \preccurlyeq \Upsilon(\Lambda^\star_{-1})^{-\epsilon_{-}}, \\
& {\rm{def}}(\Lambda_1^\star)=-\pd(\Lambda_{1})\ep{\rm{def}}(\Lambda_{1})+\pd(\Lambda_{1});\\
& {\rm{def}}(\Lambda_{-1}^\star)=-\pd(\Lambda_{-1})\ev{\mathrm{def}}(\Lambda_{-1})+\pd(\Lambda_{-1});\\
&\epsilon^\star=\epsilon_{\rho^\star}\pd(\Lambda_1)\pd(\Lambda_{-1})\ep\ev\ee;\\
&\iota^\star=-\wi_{\rho^\star}\ep\ee^{\rm{rank}(\Lambda^\star_{-1})}.\\
\end{aligned}
\]

\item[(ii)] $
\begin{cases}
 \textrm{$\lambda[a]$ and $\lambda'[a]$ are 2-transverse}, &\textrm{ if $\#[a]$ is odd};\\
   \textrm{$\lambda[a]$ and $\lambda'[a]$ are close}, &\textrm{ if $\#[a]$ is even}.
\end{cases}
$
\end{itemize}

\item Let $\prllsgn\in \CE(\o_{2n+1}^\epsilon)$ and $\prlls\in \CE(\o^{\epsilon^\star}_{2n^\star+1})$, where $n\geqslant n^\star$. Assume that $\pi_{\Lambda^\star_{\pm1}}\in \CE(\o^{\epsilon^\star_{\pm}}_{2m^\star_\pm})$,
$
\rho=\prod\pi[a]
$
and
$
\rho^\star=\prod\pi^\star[a],
$
where $\pi[a]$ (resp. $\pi^\star[a]$) is the irreducible unipotent representation corresponding to the partition $\lambda[a]$ (resp. $\lambda^\star[a]$). Then
\[
m_\epsilon(\prllsgn,\prlls)\ne 0
\]
if and only if $\prllsgn$ and $\prlls$ satisfy the following conditions:
\begin{itemize}
\item[(i)] there  are $\ep,\ev\in\{\pm\}$ such that
\[
\begin{aligned}
& \epsilon=\epsilon\ep\ev\epsilon_{\rho^\star},\\
& \Upsilon(\Lambda_{1}^\star)^{-\iota' }\preccurlyeq \Upsilon(\Lambda_1)^{-\ep}, \quad \Upsilon(\Lambda_1)^{\ep } \preccurlyeq \Upsilon(\Lambda^\star_{1})^{\iota'}, \\
& \Upsilon(\Lambda_{-1}^\star)^{-\iota'\epsilon\epsilon_{\rho}\ep\ev\ee} \preccurlyeq \Upsilon(\Lambda_{-1})^{-\ev}, \quad \Upsilon(\Lambda_{-1})^{\ev } \preccurlyeq \Upsilon(\Lambda^\star_{-1})^{\iota'\epsilon\epsilon_{\rho}\ep\ev\ee}, \\
& {\rm{def}}(\Lambda_1^\star)=\iota'\ep\rm{def}(\Lambda_{1})-\iota';\\
& {\rm{def}}(\Lambda_{-1}^\star)=\iota'\epsilon\epsilon_{\rho}\ep\ee{\mathrm{def}}(\Lambda_{-1})-\iota'\epsilon\epsilon_{\rho}\ep\ev\ee;\\
\end{aligned}
\]
where $\iota'=\iota\wi_{\rho}\ee^{\rm{rank}(\Lambda_{-1})}$.

\item[(ii)] $
\begin{cases}
 \textrm{$\lambda[a]$ and $\lambda'[a]$ are 2-transverse}, &\textrm{ if $\#[a]$ is odd};\\
   \textrm{$\lambda[a]$ and $\lambda'[a]$ are close}, &\textrm{ if $\#[a]$ is even}.
\end{cases}
$
\end{itemize}
\end{enumerate}
\end{theorem}

\subsection{The duality between Gan-Gross-Prasad problem and Theta correspondence}

The Gan-Gross-Prasad problem is naturally related to the theta correspondence in the context of relative Langlands duality of Ben-Zvi-Sakellaridis-Venkatesh \cite{BZSV} in the local field case. In the finite field case, a similar duality for unipotent representations of unitary groups has been proven in \cite{LW4}. In \cite{Wang1}, we have established the duality between the Gan-Gross-Prasad problem for special orthogonal groups and the theta correspondence for the symplectic-even orthogonal dual pair.
Based on the above description of the Gan-Gross-Prasad problem, we can extend the aforementioned duality to full orthogonal groups and symplectic groups.
 \begin{theorem}\label{8.3}
 \begin{enumerate}
\item Let $\pi=\pi_{-,\Lambda_1,-}\in\E(\sp_{2n},1)$, and $\pi^\star=\pi_{-,-,\Lambda_{-1}^\star}\in \E(\sp_{2n^\star},s)$ where  $s = (-I,1)$ with $I$ being the identity in $\o_{2n}^\e \in \so_{2n+1}\fq=\sp_{2n}\fq^*$. Let $\pi^{\prime}:=\clc(\pi^\star) \in \E(\o^\e_{2n^{\star}},1)$.
Then
\[
m_\epsilon(\pi,\pi^\star)= \langle \cal{AC}(\pi)\otimes\cal{AC}(-\ee\epsilon\e\cdot\pi'),\omega_{n,n^\star}^\e\rangle,
\]
where $\cal{AC}$ is the  Alvis-Curtis duality, and $\omega_{n,n^\star}^\e$ is the Weil representation for $(\sp_{2n},\o^\e_{2n^\star})$ and
\[
-\ee\epsilon\e\cdot\pi'=\left\{
\begin{array}{ll}
\pi' & \textrm{ if }-\ee\epsilon\e=+;\\
\sgn\pi' & \textrm{ if }-\ee\epsilon\e=-.
\end{array}
\right.
\]
In other words, the above equation can be visualized as a diagram:
  \[
\xymatrix{
\pi\in\cal{E}(\sp_{2n},1) \ar[rrrr]^{\textrm{GGP}} \ar[dd]_{\cal{AC}} &&&&\bigoplus\pi^\star,\  \pi^\star\in \cal{E}(\sp_{2n^\star},s) \ar[dd]_{\cal{AC}\circ(-\ee\epsilon\e)\circ\clc} \\
\\
\pi\in\cal{E}(\sp_{2n},1) \ar[rrrr]^{\textrm{theta lfiting} } &&&&
\bigoplus\pi^{\prime },\  \pi^{\prime}\in \cal{E}(\o^\e_{2n^\star},1).
}
\]

 \item Let $\pi=\pi_{-,\Lambda_1,-}\in\E(\o^\epsilon_{2n},1)$, and $\pi^\star=\pi_{-,\Lambda_{1}^\star,-,\iota^\star}\in \E(\o^{\epsilon^\star}_{2n^\star+1},1)$. Let $\pi^{\prime}:=\clc(\pi^\star) \in \E(\sp_{2n^{\star}},1)$.
Then
\[
m_{\ee\epsilon\epsilon^\star}(\pi,\pi^\star)= \langle \cal{AC}(\pi')\otimes\cal{AC}(-\iota^\star\epsilon\cdot\pi),\omega_{n,n^\star}^\epsilon\rangle,
\]
 The above equation can be visualized as a diagram:
  \[
\xymatrix{
\pi\in\cal{E}(\o^\epsilon_{2n},1) \ar[rrrr]^{\textrm{GGP}} \ar[dd]_{\cal{AC}\circ (-\iota^\star\epsilon)} &&&&\bigoplus\pi^\star,\  \pi^\star\in \cal{E}(\o^{\epsilon^\star}_{2n^\star+1},1) \ar[dd]_{\cal{AC}\circ\clc} \\
\\
\pi\in\cal{E}(\o^\epsilon_{2n},1) \ar[rrrr]^{\textrm{theta lfiting} } &&&&
\bigoplus\pi^{\prime },\  \pi^{\prime}\in \cal{E}(\sp_{2n^\star},1).
}
\]

 \item Let $\pi=\pi_{-,\Lambda_1,-,\iota}\in\E(\o^\epsilon_{2n+1},1)$, and $\pi^\star=\pi_{-,\Lambda_{1}^\star,-}\in \E(\o^{\epsilon^\star}_{2n^\star},1)$. Let $\pi^{\prime}:=\clc(\pi) \in \E(\sp_{2n},1)$.
Then
\[
m_{\epsilon\epsilon^\star}(\pi,\pi^\star)= \langle \cal{AC}(\pi')\otimes\cal{AC}(-\iota\epsilon^\star\cdot\pi),\omega_{n,n^\star}^\epsilon\rangle,
\]
 The above equation can be visualized as a diagram:
  \[
\xymatrix{
\pi\in\cal{E}(\o^\epsilon_{2n+1},1) \ar[rrrr]^{\textrm{GGP}} \ar[dd]_{\cal{AC}\circ \clc} &&&&\bigoplus\pi^\star,\  \pi^\star\in \cal{E}(\o^{\epsilon^\star}_{2n^\star},1) \ar[dd]_{\cal{AC}\circ(-\iota^\star\epsilon)} \\
\\
\pi\in\cal{E}(\sp_{2n},1) \ar[rrrr]^{\textrm{theta lfiting} } &&&&
\bigoplus\pi^{\prime },\  \pi^{\prime}\in \cal{E}(\o^{\epsilon^\star}_{2n^\star},1).
}
\]
\end{enumerate}
\end{theorem}
\begin{proof}
We only prove (1). In this case, the conditions in Theorem \ref{ggpmain} means that
there  are $\ep,\ev\in\{\pm\}$ such that
\[
\begin{aligned}
& \epsilon=\ep\ev,\\
& \Upsilon(\Lambda_{-1}^\star)^{-\ep\ev{\ee} }\preccurlyeq \Upsilon(\Lambda_1)^{\ep }, \quad \Upsilon(\Lambda_1)^{-\ep } \preccurlyeq \Upsilon(\Lambda^\star_{-1})^{\ep\ev{}\ee}, \\
& {\mathrm{def}}(\Lambda_{-1}^\star)=-\ev\ee({\mathrm{def}}(\Lambda_1)-\ep).
\end{aligned}
\]
The last equation implies that $\ep=\epsilon^\star$. Then (1) immediately follows from Theorem \ref{p1}.
\end{proof}

The general case is much more complicated because many signs are involved. However, once we ignore them, there exists a duality between theta lifting and Gan-Gross-Prasad problem.

\begin{theorem}
 \begin{enumerate}
\item Let $\pi\in\E(\sp_{2n})$, and $\pi^\star\in \E(\sp_{2n^\star})$.
Then we have the following commutative diagram up to a twist of the sgn character on orthogonal groups:
  \[
\xymatrix{
\pi\in\cal{E}(\sp_{2n}) \ar[rrrr]^{\textrm{GGP}} \ar[dd]_{\cal{AC}\circ\clc} &&&&\bigoplus\pi^\star,\  \pi^\star\in \cal{E}(\sp_{2n^\star},s) \ar[dd]_{\cal{AC}\circ\clc} \\
\\
\rho\otimes\pi_{\Lambda_1}\otimes\pi_{\Lambda_{-1}}  \ar[rrrr]^{\textrm{ theta lfiting }} &&&&
\bigoplus\pi^{\star },\  \rho^\star\otimes\pi_{\Lambda^\star_1}\otimes\pi_{\Lambda^\star_{-1}}
}
\]
where the sum in the bottom right corner runs over representations such that
\begin{itemize}
\item $\rho=\prod_{[a],a\ne \pm 1}\pi[a]$ and $\rho^\star=\prod_{[a],a\ne \pm 1}\pi^\star[a]$
\item For each $[a]$, $\pi^\star[-a]$ appears in theta lifting of $\pi[a]$;
\item $\clc(\pi^\star_{\Lambda_1})$ appears in theta lifting of $\pi_{\Lambda_{-1}}$;
\item $\pi^\star_{\Lambda_{-1}}$ appears in theta lifting of $\clc(\pi_{\Lambda_1})$.
\end{itemize}

\item Let $\pi\in\E(\o^\epsilon_{2n})$, and $\pi^\star\in \E(\o^{\epsilon^\star}_{2n^\star+1})$.
Then we have the following commutative diagram up to a twist of the sgn character on orthogonal groups:
  \[
\xymatrix{
\pi\in\cal{E}(\o^\epsilon_{2n}) \ar[rrrr]^{\textrm{GGP}} \ar[dd]_{\cal{AC}\circ\clc} &&&&\bigoplus\pi^\star,\  \pi^\star\in \cal{E}(\o^{\epsilon^\star}_{2n^\star},s) \ar[dd]_{\cal{AC}\circ\clc} \\
\\
\rho\otimes\pi_{\Lambda_1}\otimes\pi_{\Lambda_{-1}}  \ar[rrrr]^{\textrm{ theta lfiting }} &&&&
\bigoplus\pi^{\star },\  \rho^\star\otimes\pi_{\Lambda^\star_1}\otimes\pi_{\Lambda^\star_{-1}}\otimes\iota^\star
}
\]
where the sum in the bottom right corner runs over representations such that
\begin{itemize}
\item $\rho=\prod_{[a],a\ne \pm 1}\pi[a]$ and $\rho^\star=\prod_{[a],a\ne \pm 1}\pi^\star[a]$
\item For each $[a]$, $\pi^\star[a]$ appears in theta lifting of $\pi[a]$;
\item $\clc(\pi^\star_{\Lambda_1})$ appears in theta lifting of $\pi_{\Lambda_1}$;
\item $\clc(\pi^\star_{\Lambda_{-1}})$ appears in theta lifting of $\pi_{\Lambda_{-1}}$.
\end{itemize}

\end{enumerate}
\end{theorem}

\subsection{A description of the first descent case}

Although the description of the multiplicity in the finite Gan-Gross-Prasad problem (as shown in Theorem \ref{ggpmain}) is quite complicated, the first descent case of the problem is simple. This case is similar to the first occurrence case of theta lifting. We have already constructed a functor $\CD^*_\epsilon$ for basic representations, and we can read the action of $\CD^*_\epsilon$ for all representations by referring to Theorem \ref{ggpmain}. To calculate the first descent of an irreducible representation $\pi$, we need to find the representations that live in the smallest groups, and satisfy the conditions in Theorem \ref{ggpmain}. This process is a direct but complex combinatorial calculation, which we will leave to the readers. Here, we present the results.

Let $\Lambda=\left(\begin{smallmatrix}a_1,\cdots,a_m\\b_1,\cdots,b_{m'}\end{smallmatrix}\right)$ be a symbol. Set
\[
\zeta(\Lambda)=
\left\{
\begin{array}{ll}
\{+\}& \textrm{ if } a_1>b_1\\
\{-\}& \textrm{ if } a_1<b_1\\
\{\pm\}& \textrm{ if } a_1=b_1\\
\end{array}
\right.
\]
Assume that $G_n=\sp_{2n}$, and $\rho\otimes\pi_{\Lambda_1}\otimes\pi_{\Lambda_{-1}}\in \CE(C_{G_n^*}(s))$. Then we set
\[
 \begin{aligned}
\CD^*_\epsilon(\rho\otimes\pi_{\Lambda_1}\otimes\pi_{\Lambda_{-1}})=
\begin{cases}
\bigoplus\rho^\star\otimes\pi_{\Lambda^\star_1}\otimes\pi_{\Lambda^\star_{-1}} & \textrm{ if }\epsilon=\ep\ev\epsilon_{\rho}\epsilon_{\rho^\star};\\
0& \textrm{ otherwise, }
\end{cases}
 \end{aligned}
\]
where the direct sum runs over irreducible representations $\rho^\star\otimes\pi_{\Lambda^\star_1}\otimes\pi_{\Lambda^\star_{-1}} $ such that
\begin{itemize}
\item $\rho^\star=\rho^-_1$;
\item $-\ep\in \zeta({}^{\bd}\Lambda_1)$, and $\ev\in \zeta({}^{\bd}\Lambda_{-1})$;
\item The symbol $\Lambda_1^\star$ is obtained by following these steps:  (1) Take the Alvis-Curtis dual ${}^\bd\Lambda_{-1}$, (2) Removing the maximal number in the top (resp. bottom) line in ${}^\bd\Lambda_{-1}$ if $\ev=+$ (resp. $\ev=-$ ), (3) Swapping the arrays of $\Lambda_{-1}$ if $\ev\pd(\Lambda_{-1})=+$, (4) Retake the Alvis-Curtis dual;

\item The symbol $\Lambda_{-1}^\star$ is obtained by following these steps:  (1) Take the Alvis-Curtis dual ${}^\bd\Lambda_{1}$, (2) Removing the maximal number in the top (resp. bottom) line in ${}^\bd\Lambda_{1}$ if $-\ep=+$ (resp. $-\ep=-$ ), (3) Swapping the arrays of $\Lambda_{1}$ if $\ee\ev\pd(\Lambda_{-1})=+$, (4) Retake the Alvis-Curtis dual.
\end{itemize}
Assume that $G_n=\o^\epsilon_{2n}$, and $\rho\otimes\pi_{\Lambda_1}\otimes\pi_{\Lambda_{-1}}\in \CE(C_{G_n^*}(s))$. Then we set
\[
\CD^*_{\epsilon'}(\rho\otimes\pi_{\Lambda_1}\otimes\pi_{\Lambda_1})=
\begin{cases}
\bigoplus\rho^\star\otimes\pi_{\Lambda^\star_1}\otimes\pi_{\Lambda^\star_{-1}}\otimes\iota^\star & \textrm{ if }\epsilon'=\ep\ev\epsilon_{\rho}\epsilon_{\rho^\star};\\
0& \textrm{ otherwise, }
\end{cases}
\]
where the direct sum runs over irreducible representations $\rho^\star\otimes\pi_{\Lambda^\star_1}\otimes\pi_{\Lambda^\star_{-1}}\otimes\iota^\star$ such that
\begin{itemize}
\item $\rho^\star=\rho_1$;
\item $\ep\in \zeta({}^{\bd}\Lambda_1)$, and $\ev\in \zeta({}^{\bd}\Lambda_{-1})$;
\item The symbol $\Lambda_1^\star$ is obtained by following these steps: (1) Take the Alvis-Curtis dual ${}^\bd\Lambda_{1}$, (2) Removing the maximal number in the top (resp. bottom) line in ${}^\bd\Lambda_{1}$ if $\ep=+$ (resp. $\ep=-$ ), (3) Swapping the arrays of $\Lambda_{1}$ if $\ep\pd(\Lambda_{1})=+$ (4) Retake the Alvis-Curtis dual;
\item The symbol $\Lambda_{-1}^\star$ is obtained by following these steps: (1) Take the Alvis-Curtis dual ${}^\bd\Lambda_{-1}$, (2) Removing the maximal number in the top (resp. bottom) line in ${}^\bd\Lambda_{-1}$ if $\ev=+$ (resp. $\ev=-$ ), (3) Swapping the arrays of $\Lambda_{-1}$ if $\ev\pd(\Lambda_{-1})=+$ (4) Retake the Alvis-Curtis dual;
    \item $\iota^\star=-\wi_{\rho^\star}\ep\ee^{\rm{rank}(\Lambda_{-1})}$.
\end{itemize}
Assume that $G_n=\o^\epsilon_{2n+1}$, and $\rho\otimes\pi_{\Lambda_1}\otimes\pi_{\Lambda_{-1}}\otimes\iota\in \CE(C_{G_n^*}(s))$. Then we set
\[
\CD^*_{\epsilon'}(\rho\otimes\pi_{\Lambda_1}\otimes\pi_{\Lambda_1}\otimes\iota)=
\begin{cases}
\bigoplus\rho^\star\otimes\pi_{\Lambda^\star_1}\otimes\pi_{\Lambda^\star_{-1}} & \textrm{ if }\epsilon'=\ep\ev\epsilon\epsilon_{\rho^\star};\\
0& \textrm{ otherwise, }
\end{cases}
\]
where the direct sum runs over irreducible representations $\rho^\star\otimes\pi_{\Lambda^\star_1}\otimes\pi_{\Lambda^\star_{-1}} $ such that
\begin{itemize}
\item $\rho^\star=\rho_1$;
\item $\ep\in \zeta({}^{\bd}\Lambda_1)$, and $\ev\in \zeta({}^{\bd}\Lambda_{-1})$;
\item let $\iota'=\iota\wi_{\rho}\ee^{\rm{rank}(\Lambda_{-1})}$;
\item The symbol $\Lambda_1^\star$ is obtained by following these steps:
 (1) Take the Alvis-Curtis dual ${}^\bd\Lambda_{1}$, (2) Removing the maximal number in the top (resp. bottom) line in ${}^\bd\Lambda_{1}$ if $\ep=+$ (resp. $\ep=-$ ), (3) Swapping the arrays of $\Lambda_{1}$ if $\ep\iota'=-$ (4) Retake the Alvis-Curtis dual;
 \item The symbol $\Lambda_{-1}^\star$ iis obtained by following these steps:
 (1) Take the Alvis-Curtis dual ${}^\bd\Lambda_{-1}$, (2) Removing the maximal number in the top (resp. bottom) line in ${}^\bd\Lambda_{1}$ if $\ev=+$ (resp. $\ev=-$ ), (3) Swapping the arrays of $\Lambda_{-1}$ if $\iota'\epsilon\epsilon_{\rho}\ep\ee=-$, (4) Retake the Alvis-Curtis dual.

\end{itemize}
For basic representations, the definition of $\CD^*_\epsilon$ coincides with $\CD^*_\epsilon$ defined in Section \ref{sec7}.

In most cases, the first descent of an irreducible representation is also irreducible. However, there is a non-irreducible first descent case that occurs only in a specific situation. In the $\CD^*_\epsilon$ functor, we need to remove the largest element of some symbols. In most cases, the largest element is unique, which is the irreducible first descent case. However, when the largest elements are not unique, we need to choose one of them. Each choice will then give us an irreducible constituent of the first descent.

\begin{example}
Consider the trivial representation $\bf{1}$ of $\sp_2\fq$, which corresponds to the symbol
\[
\Lambda_1=\begin{pmatrix}
1\\
-
\end{pmatrix}, \quad\textrm{with }
{}^{\bd}\Lambda_1=\begin{pmatrix}
1,0\\
1
\end{pmatrix}.
\]
In this case,  $\ep$ can be both $+$ and $-$. Since the representation ${\bf 1}$ does not have the $(-1)$-part and the empty $(-1)$-part corresponds to the trivial symbol
     \[
 \begin{pmatrix}
-\\
-
\end{pmatrix},
\]
 the sign $\ev$ can be both $+$ and $-$.  Then we have the following descent diagram
\[
\xymatrix{
&  {\bf{1}} \ar[dl] \ar[dr]  &  \\
\pi_{\Lambda_{-1}^{\prime }} \oplus \pi_{\Lambda_{-1}^{\star }}  \ar[d] & & \pi_{\Lambda_{-1}^{\prime t}}\oplus\pi_{\Lambda_{-1}^{\star t }} \ar[d] \\
\bf{1} \oplus \bf{1} & & \bf{1} \oplus \bf{1}
}
\]
where
\[
\Lambda_{-1}'=\begin{pmatrix}
1\\
0
\end{pmatrix}\
\textrm{ and }\
\Lambda_{-1}^*=\begin{pmatrix}
1,0\\
-
\end{pmatrix}.
\]
The representations $\pi_{\Lambda_{-1}^{\prime }}$ and $\pi_{\Lambda_{-1}^{\star}}$ in the descent of $\bf{1}$ correspond to the pairs $(\ep,\ev)=(-,+)$ and $(\ep,\ev)=(+,-)$, respectively. Note that $\pi_{\Lambda_{-1}^{\prime }}$ and $\pi_{\Lambda_{-1}^{\star}}$ have different cuspidal supports. Namely,  $\pi_{\Lambda_{-1}^{\star}}$ itself is cuspidal, while the cuspidal support of $\pi_{\Lambda_{-1}^{\prime}}$ is the trivial representation of $\sp_0\fq$.
Thus in general  different choices of $\ep$ in the descent may give rise to representations with different cuspidal supports.
\end{example}

Combining Theorem \ref{main} and Theorem \ref{ggpmain}, we find that the diagram \eqref{cm2} is not only commutative for basic representations, but also for all representations.
\begin{corollary}\label{8.6}
 Let $\pi\in \E(G)$ where $G$ is either a symplectic group, odd orthogonal group or a even orthogonal group. Then we have the following commutative diagram
 \begin{equation}\label{cm3}
\xymatrix{
&\pi^*\ar[rrr]^{\Theta^*_{n,n'}} \ar[dd]_{\CD_\epsilon^*}& &&\pi^{\prime *} \ar[dd]^{\CD^*_\epsilon}\\
\pi \ar[ur]^{\CL^\rm{can}}\ar[dd]_{\CD_{\ell,\epsilon}}\ar[rrr]^{\Theta_{n,n'}}&&&\pi' \ar[ur]^{\CL^\rm{can}} \ar[dd]^{\CD_{\ell',\epsilon}}& \\
&\pi^{\star*} \in \CD^*_\epsilon(\pi^*) \ar[rrr]^{\Theta^*_{m,m'}}& &&\pi^{\star\prime*}\in \CD^*_\epsilon(\pi^{\prime*})\\
\pi^\star\in \CD_{\ell,\epsilon}(\pi)\ar[ur]^{\CL^\rm{can}}\ar[rrr]^{\Theta_{m,m'}}&&&\pi^{\star\prime}\in \CD_{\ell,\epsilon}(\pi')\ar[ur]^{\CL^\rm{can}}&\\
}.
\end{equation}
\end{corollary}

\begin{corollary}
Let $G$ be a symplectic group or orthogonal group. Then we have
\[
\begin{array}{c|cccc}
G& \pi                                 &\sgn\pi       &\pi^c             &\chi\cdot\pi\\ \hline
\sp_{2n}        & \pi_{\rho,\Lambda_1,\Lambda_{-1}} &\empty & \pi_{\rho,\Lambda_1,\Lambda_{-1}^t} &\empty\\
\o^\epsilon_{2n}        & \pi_{\rho,\Lambda_1,\Lambda_{-1}} &\pi_{\rho,\Lambda_1^t,\Lambda_{-1}^t} & \pi_{\rho,\Lambda_1,\Lambda_{-1}^t} &\pi_{\rho',\Lambda'_{-1},\Lambda'_{1}}\\
\o^\epsilon_{2n+1}        & \pi_{\rho,\Lambda_1,\Lambda_{-1},\iota} &\pi_{\rho,\Lambda_1,\Lambda_{-1},-\iota} & &\pi_{\rho',\Lambda_{-1},\Lambda_{1},\iota'\iota}\\
\end{array}
\]
where we set
\[
\Lambda_{\pm1}'=
\left\{
\begin{array}{ll}
\Lambda_{\pm1} & \textrm{ if }\pd(\Lambda_1)\pd(\Lambda_{-1})=+;\\
\Lambda_{\pm1}^t & \textrm{ if }\pd(\Lambda_1)\pd(\Lambda_{-1})=-,\\
\end{array}
\right.
\]
and $\iota'=\chi_G(-I)$.

\end{corollary}
\begin{proof}

We will start by providing the first three columns of the table. The symplectic and odd orthogonal cases have already been proven in Corollary \ref{l4}.

Next, we need to calculate $\sgn\pi$ of $\o^\epsilon_{2n}\fq$. We claim that there exists an irreducible representation $\pi'$ of an odd orthogonal group such that the first descent $\CD^{\rm{B}}_{\ell,\epsilon'}(\pi')=\pi$. In fact, by Corollary \ref{8.6}, we know that $\CD^{\rm{B}}_{\ell,\epsilon}(\pi')=\pi$ if and only if $\CD_\epsilon^*(\clc(\pi'))=\clc(\pi)$. Recall that $\CD_\epsilon^*$ is defined as taking the Alvis-Curtis dual of certain symbols and removing certain elements. To find $\pi'$, we can do the opposite, which is to add some elements. By doing this, we can find an irreducible representation $\pi'$.
Using Theorem \ref{dodd} (iii), we have $\CD^{\rm{B}}_{\ell,\epsilon'}(\sgn\pi')=\sgn\pi$. Although Theorem \ref{dodd} (iii) is only for basic representation, it is not difficult to generalize the same statement to all representations. Thus, the description of $\sgn\pi$ follows from the description of $\sgn\pi'$ and the definition of $\CD^*_\epsilon$.

Let us now consider $\pi^c$ of $\o^\epsilon_{2n}\fq$. Let $\pi''$ be an irreducible representation of a symplectic group such that $\pi''$ appears in the theta lifting of $\pi$. Then by \cite[Lemma 5.3]{P3}, we know that $(\pi'')^c$ appears in the theta lifting of $\pi^c$. Hence the description of $\pi^c$ follows from the description of $(\pi'')^c$ and  Theorem \ref{main} (3).

We now turn to proving the fourth column of the table. Let $G$ be an odd or even orthogonal group. Consider the first descent $\CD^{\rm{B}}_{\ell,\epsilon^\star}(\pi)=\pi^\star$ of $\pi$, where $\pi^\star$ is a representation of $G^\star$ (a smaller orthogonal group than $G$), which may not be irreducible. By Theorem \ref{deven} (ii) and Theorem \ref{dodd} (iv) (which can be generalized to all representations), we have $\CD^{\rm{B}}_{\ell,\epsilon^\star}(\chi_G\cdot\pi)=\chi_{G^\star}\cdot\pi^\star$. Then, we can obtain the description of $\chi_G\cdot\pi$ by induction on $n$.
\end{proof}

\begin{remark}
The above description of representations of even orthogonal group is compatible with the fact
\[
(\chi_G\cdot\pi^c)^c=\sgn\chi_G\cdot\pi.
\]
\end{remark}%

\end{document}